\documentclass[cupthm]{CUP-JNL-JMJ}

\usepackage{latexsym}
\usepackage{graphicx}
\usepackage{multicol,multirow}
\usepackage{amsmath,amssymb,amsfonts}
\usepackage{mathrsfs}
\usepackage{amsthm}
\usepackage{rotating}
\usepackage{appendix}
\usepackage[numbers]{natbib}

\usepackage{float}

\def\C{{\mathbb C}}
\def\F{{\mathbb F}}
\def\N{{\mathbb N}}
\def\Q{{\mathbb Q}}
\def\R{{\mathbb R}}
\def\Z{{\mathbb Z}}

\def\cA{{\mathcal A}}
\def\cE{{\mathcal E}}
\def\cP{{\mathcal P}}
\def\cS{{\mathcal S}}
\def\cW{{\mathcal W}}

\def\cH{{\mathcal H}}

\def\fourier{\F}

 \def\nt{{\N^{\times}}}
 \def\qqq{\,,\,~\forall}
 \def\rep{\vartheta}
 \def\Qer{{\rm Ker}}
 \def\sr0{{\cS^{\rm ev}_0}}
 
 \def\fourierer{\fourier_{e_\R}}
\def\scal2{{\mathscr S}}
 \def\nt{{\N^{\times}}}

\newcommand{\ie}{{\it i.e.\/}\ }

\numberwithin{equation}{section}

\articletype{RESEARCH ARTICLE}
\jname{J. Inst. Math. Jussieu}
\jyear{2020}
\jvol{1}
\jissue{1}
\jdoi{10.1017/S1474748020000237}

\theoremstyle{cupplain}
\newtheorem{theorem}{Theorem}[section]
\newtheorem{lemma}[theorem]{Lemma}
\newtheorem{corollary}[theorem]{Corollary}
\newtheorem{proposition}[theorem]{Proposition}

\theoremstyle{cupdefinition}
\newtheorem{definition}{Definition}[section]

\theoremstyle{cupremark}

\theoremstyle{cupproof}
\newtheorem{proof}{Proof}

\begin{document}

\begin{Frontmatter}

\title[Spectral triples and $\zeta$-cycles 
]{SPECTRAL TRIPLES and $\zeta$-CYCLES \thanks{The second author is partially supported by the Simons Foundation collaboration grant n. 691493.}}

\author[1]{ALAIN CONNES}
\author[2]{CATERINA CONSANI}

\address[1]{\orgname{Coll\`ege de France}, \orgaddress{\state{3 rue d'Ulm}, \city{Paris F-75005, France}}}

\address[1]{\orgname{IHES}, \orgaddress{\state{35 rte de Chartres}, \city{91440 Bures-sur-Yvette, France}} \email{alain@connes.org}}

\address[2]{\orgdiv{Dept. of Mathematics} 
\orgname{Johns Hopkins University}, \orgaddress{\city{Baltimore 21218} USA} \email{cconsan1@jhu.edu}}

\received{2021}
\revised{2021}
\accepted{2021}

\maketitle

\authormark{A. Connes and C. Consani}

\abstract{We  exhibit very small eigenvalues of the quadratic form associated to  the Weil explicit formulas restricted to test functions whose support is within a fixed interval with upper bound S. We show both numerically and conceptually that the associated eigenvectors are obtained by a simple arithmetic operation of finite sum using  prolate spheroidal wave functions  associated to the scale S. Then we use these functions to condition the canonical spectral triple of the circle of  length L=2 Log(S) in such a way that they belong to the kernel of the perturbed Dirac operator. We give numerical evidence that, when one varies  L, the low lying spectrum of the perturbed spectral triple resembles the low lying zeros of the Riemann zeta function. We justify conceptually  this result and show that, for each eigenvalue, the coincidence is perfect for the special values of the length L of the circle for which the two natural ways of realizing the perturbation give the same eigenvalue. This fact is tested numerically by reproducing the first thirty one zeros of the Riemann zeta function from our spectral side, and estimate  the probability of having obtained this agreement at random, as a very small number whose first fifty decimal places are all  zero. The theoretical concept which emerges is that of zeta cycle and our main result  establishes its relation with the critical zeros of the Riemann zeta function and with the spectral realization of these zeros obtained by the first author.}


\keywords[\textup{2010} Mathematics subject classification]{11M55 (primary), 11M06, 46L87, 58B34  (secondary)}

%\begin{policy}[Impact Statement]
%Provide a 200 word impact statement that summarises the significance of the work, so that it can be quickly grasped by a wide audience (including industry, government and wider academia).
%\end{policy}

\end{Frontmatter}

\section{Introduction} 
When contemplating the low lying zeros of the Riemann zeta function  one is tempted to speculate that they may form the spectrum of an operator of the form $\frac 12+iD$  with $D=D^*$  self-adjoint, and  to search for the  geometry provided by a spectral triple\footnote{A triple $(\cA,\cH,D)$ where $\cA$ is an algebra acting in the Hilbert space $\cH$ and $D$ is an unbounded self-adjoint operator in $\cH$, this is the basic paradigm of noncommutative geometry \cite{Cobook}} for which $D$ is the Dirac operator. 
In this paper we give the construction of a spectral triple $\Theta(\lambda,k)=(\cA(\lambda),\cH(\lambda),D(\lambda,k))$ which  admits, as shown  for small values of $\lambda>1$, a spectrum of $\frac 12+iD$ very similar to the low lying zeros of the Riemann zeta function (this fact is exemplified in Figure \ref{dirac10p5}, for  $\lambda^2= 10.5$).
\begin{figure}
\centering
    \includegraphics[width=0.2\textwidth]{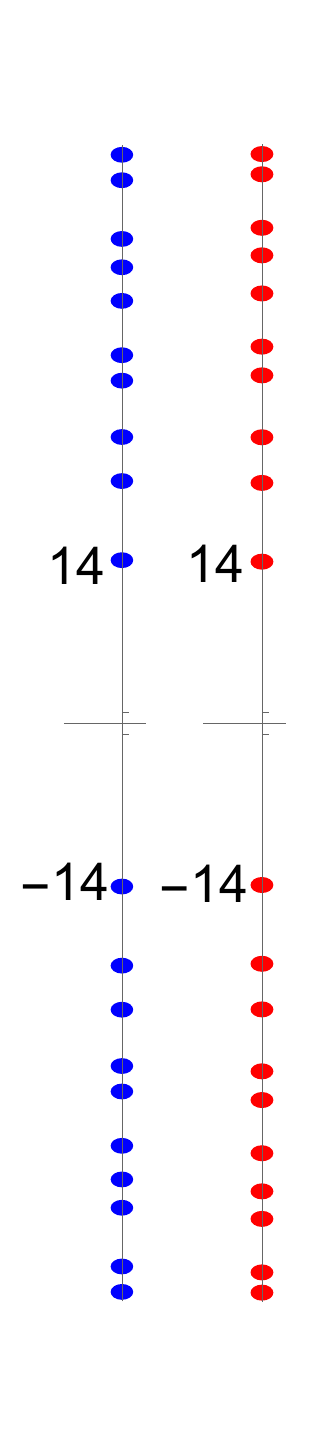}
    \caption{The low lying spectrum  of $iD(\lambda,k)$ for  $\lambda^2=10.5$, $k=18$, on  the left (in blue). On the right (in red) the low lying zeros of the Riemann zeta function  \label{dirac10p5}}
\end{figure}
 More precisely, the spectral triple $\Theta(\lambda,k)$    depends  on $\lambda$ and also on the choice of an integer $k<2\lambda^2$,  moreover for a fixed value of $k$ the positive non-zero eigenvalues $\lambda_n(D(\lambda,k))$ arranged in increasing order, vary continuously with $\lambda$. A striking fact (discovered numerically at first)
 is that for special values of $\lambda$ the dependence of $\lambda_n(D(\lambda,k))$ on the value of $k$ (close enough to $2\lambda^2$) disappears (see Figure \ref{figquant2} for the case $n=1$),  while  the \emph{common value} of these $\lambda_n(D(\lambda,k))$  coincides exactly with the imaginary part of the $n$-th zero of the Riemann zeta function! 
This means that the qualitative resemblance of spectra as in  Figure \ref{dirac10p5}  yields in fact  a sharp coincidence in some range: by varying  $\lambda$ in the interval $5\leq \lambda^2\leq 16.5$, and determining the coinciding eigenvalues up to $n=31$  one   produces $31$ numbers in amazing agreement with the  full collection of values of the first $31$  zeros of  the zeta function   
(see Figure \ref{figquant3}; incidentally  notice that the probability of obtaining such agreement from a random choice is of the order of $10^{-50}$ ).

\begin{figure}[H]	\begin{center}
\includegraphics[scale=0.50]{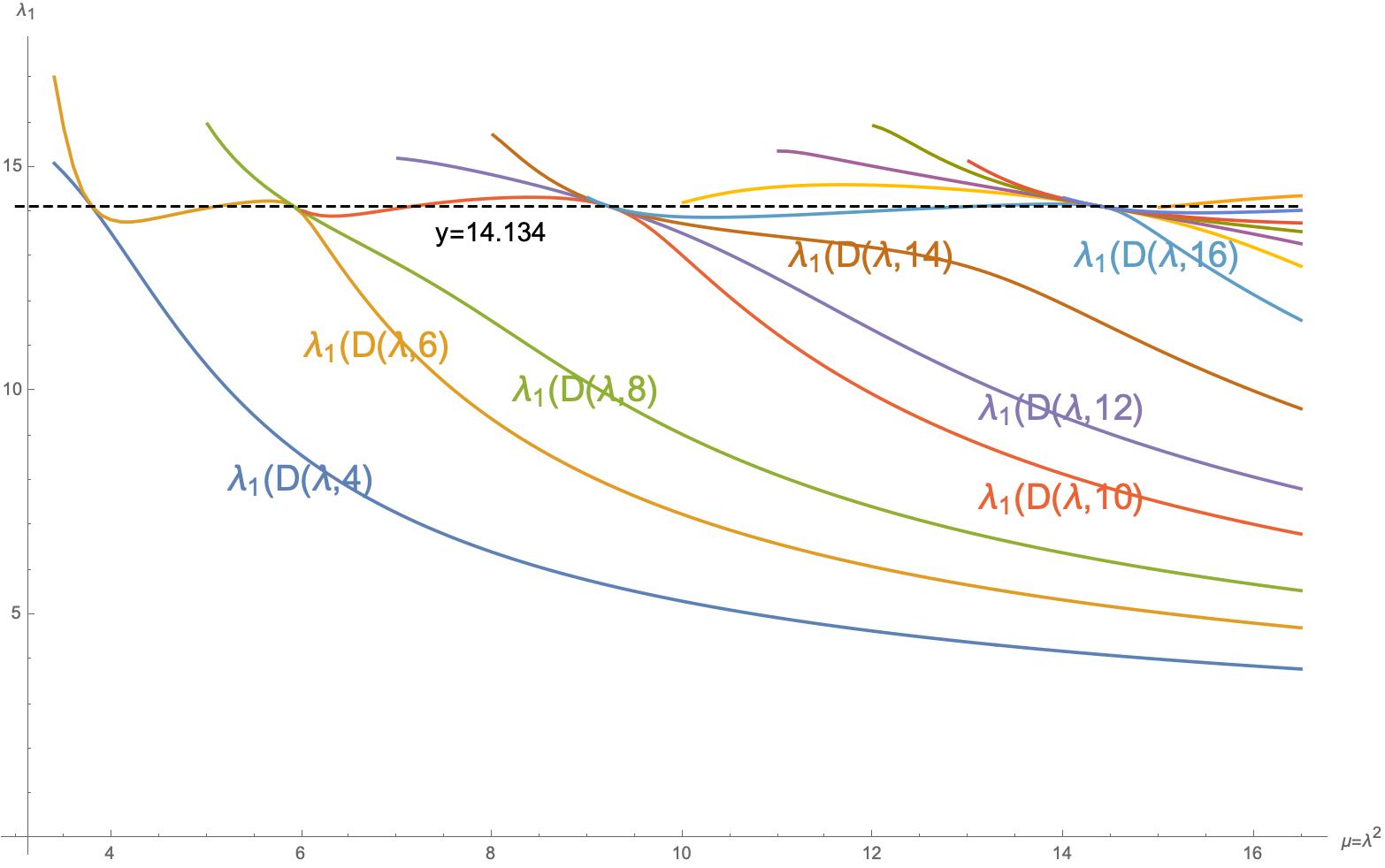}
\end{center}
\caption{The curves represent as a function of $\mu=\lambda^2$ the  first positive eigenvalue $\lambda_1(D(\lambda,2k))$ of $D(\lambda,2k)$. The ordinate of the points where the graphs touch each other is constant and coincides with the imaginary part $\zeta_1\sim 14.134$ of the first zero of zeta. The abscissas, \ie the values of $\mu$, are part of the geometric progression with ratio $\exp(\frac{2\pi }{\zeta_1})$ \label{figquant2}}
\end{figure}
\vspace*{-20pt}
\begin{figure}[H]	\begin{center}
\includegraphics[scale=0.40]{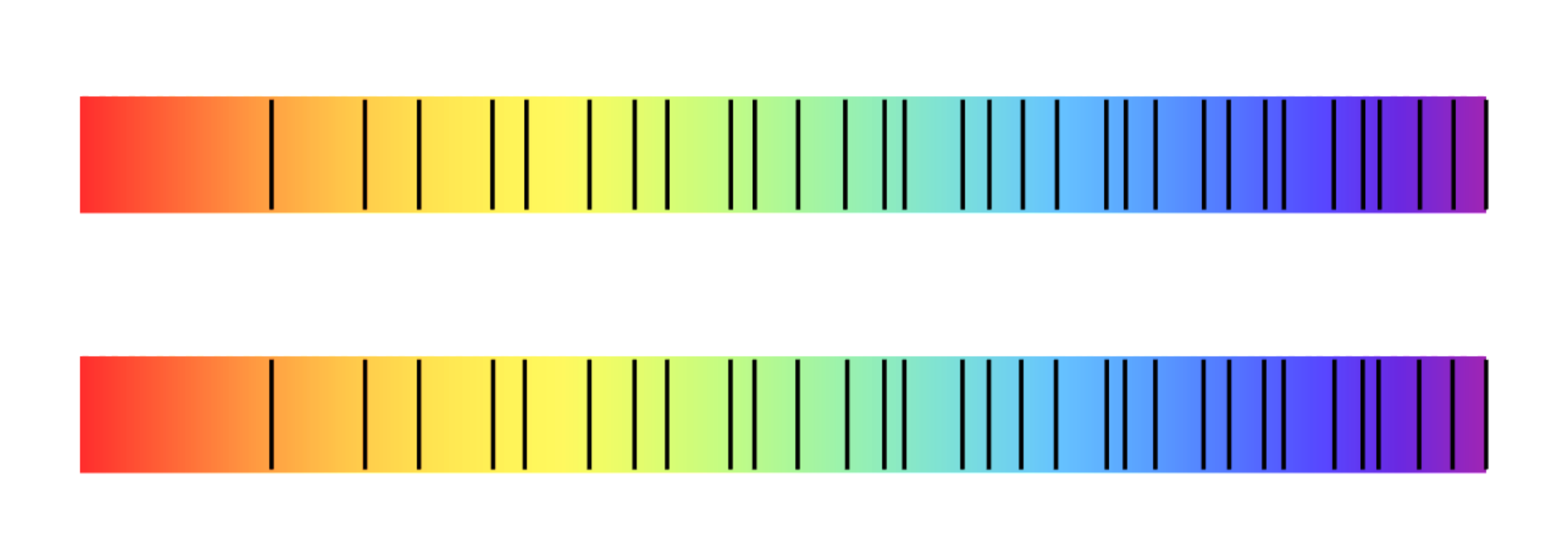}
\end{center}
\caption{ Computing the coinciding eigenvalues $\lambda_j(D(\lambda,k))$ one obtains a list (lower line) which one compares with the list (upper line) of imaginary parts $\zeta_j$ of zeros of zeta\label{figquant3}}
\end{figure}
The main goal of this paper is  to provide a theoretical explanation for this numerical ``coincidence'' and relate it to the spectral realization   of the zeros of zeta given in \cite{Co-zeta}. The new theoretical concept that emerges is that of a $\zeta$-cycle. 

In Section \ref{sectriemannsums} we explain how to define scale invariant Riemann sums for functions defined on $[0,\infty)$ with vanishing integral. This technique is implemented in the definition of the  linear map \[
\Sigma_\mu \cE: \sr0\to L^2(C)
\]
from the Schwartz space $\sr0$ of even functions,  $f,f(0)=0$, with vanishing  integral, to square integrable functions on the circle $C=\R_+^*/\mu^\Z$  of length $L=\log \mu$. The key notion is then provided by the following
\begin{definition} A {\bf $\zeta$-cycle} is  a circle $C$ of length $L=\log \mu$  such that the subspace $\Sigma_\mu \cE(\sr0)$ is not dense in the Hilbert space $L^2(C)$.
\end{definition}

It turns out that likewise for closed geodesics,  $\zeta$-cycles are stable under finite covers, and if $C$ is a $\zeta$-cycle of length $L$, then 
 the $n$-fold cover of $C$ is a $\zeta$-cycle of length $nL$, for any positive integer $n>0$.\newline
By construction, the subspace $\Sigma_\mu \cE(\sr0)\subset L^2(C)$ is invariant under the group of rotations of the circle which appears here from  the scaling action of the multiplicative group $\R_+^*$ on $C=\R_+^*/\mu^\Z$. The main result of this paper is the following 
\begin{theorem}\label{introspectralreal} $(i)$~The spectrum of the action of the multiplicative group $\R_+^*$ on the orthogonal  of $\Sigma_\mu \cE(\sr0)$ in $L^2(C)$ is formed by imaginary parts of zeros of the Riemann zeta function on the critical line.\newline
$(ii)$~Let $s>0$ be such that $\zeta(\frac 12+is)=0$, then any circle of length an integral multiple of $2\pi /s$ is a $\zeta$-cycle, and the spectrum of the action of $\R_+^*$ on  $(\Sigma_\mu \cE(\sr0))^\perp$ contains $s$.\end{theorem}

The ad-hoc Sobolev spaces  used in \cite{Co-zeta} to provide the spectral realization of zeros of  zeta   are here replaced by  the canonical Hilbert space $L^2(C)$ of square integrable functions. Moreover, Theorem \ref{introspectralreal} provides the theoretical explanation for  the above coincidence of spectral values. Indeed, the special values of $\lambda^2=\mu=\exp L$ at which the  $k$ dependence of the eigenvalue $\lambda_n(D(\lambda,k))$ disappear, signal  that the related circle of length $L$ is a $\zeta$-cycle and that $\lambda_n(D(\lambda,k))$ is in its spectrum. This  explains why the low lying part of the spectrum of the  spectral triple $\Theta(\lambda,k)$  possesses a tantalizing resemblance with   the low lying zeros of  the Riemann zeta function. Indeed, the special values of the length ($L$) of the circle for which the coinciding  $\lambda_n(D(\lambda,k))$ occur, form a part of the arithmetic progression  of multiples of $2\pi /\zeta_n$, where $\zeta_n$ is the imaginary part of the $n$-th zero of the zeta  function. This forces the graphs of the functions $\lambda_n(D(\mu^{1/2},k))$ to pass through points of the form $(\exp(2\pi m/\zeta_n),\zeta_n)$ (as in Figure \ref{figquant2}) which entails that the low lying spectrum of $D(\lambda,k)$ (for $k\sim 2\lambda^2$)  mimics the low lying zeros of the zeta function.
\newline
The  spectral triple $\Theta(\lambda,k)$ is  a finite rank perturbation of the Dirac operator on a circle of length $\log \mu=2\log \lambda$ and involves, as a key ingredient, classical prolate spheroidal wave functions  \cite{Slepian, Sl, Slepian0}. These functions are used to define a finite dimensional subspace (of dimension $k$) of the Hilbert space of square integrable functions on the circle  of length $2\log\lambda$, and the operator $D(\lambda,k)$ is then \emph{canonically} obtained from the operator of ordinary differentiation $D_0(\lambda)$ to insure that its kernel contains the above finite dimensional subspace.\newline
 A priori, there seems to be no relation between the construction of the spectral triple $\Theta(\lambda,k)$ and the Riemann zeta function:  in  Section \ref{riemweilexpl} we explain how we stumbled on $\Theta(\lambda,k)$ while continuing  our investigations of the Weil quadratic form restricted to test functions with support in a fixed interval. The Riemann-Weil explicit formulas give a concrete and finite expression of the semi-local Weil quadratic form  (see Section \ref{riemweilexpl0}) which is suitable for numerical exploration since it only involves primes less than, say, $\lambda^2$. By  semi-local Weil quadratic form we mean the restriction $QW_\lambda$ of the sesquilinear form 
\begin{equation}\label{weilQ}
QW(f,g):=\sum_{1/2+is\in Z} 	\overline{\widehat f(\bar s)}\widehat g(s)
\end{equation}
on test functions $f,g$ whose support is contained in the interval $[\lambda^{-1},\lambda]\subset \R_+^*$. In \eqref{weilQ},  $Z$ is the set of non-trivial zeros of the Riemann zeta function and Fourier transform is defined on $C_c^{\infty}(\R_+^*)$ by 
\begin{equation}\label{fouriermu}
\widehat f(s)=\fourier_\mu(f)(s):=\int_{\R_+^*}f(u)u^{-is}d^*u.
\end{equation}
 One knows that the positivity of the Weil quadratic form $QW_\lambda$ for all $\lambda$ implies the Riemann Hypothesis (RH), and in case RH holds,  $QW_\lambda$ is known to be strictly positive. In \cite{yoshida}, the positivity was shown to hold for $\lambda=\sqrt 2$ using numerical analysis.    In Section \ref{riemweilexpl0} we test numerically this positivity for larger values of $\lambda$, showing (\S \ref{sectsensitive})  that the contribution from the archimedean place alone  ceases  to be positive  in the upper part of the interval $\log(\lambda^2) \in [\log 2-0.2,\log 2+0.2]\sim [0.493,0.893]$, while  the positivity is  restored by adding the contribution  of the prime $2$. This latter contribution depends explicitly on $p=2$ in a form $W_p$ which in fact can be evaluated for any real number $p\sim 2$ (\ie  close to but not equal to $2$). We show  (\S \ref{sectsensitivep2})  that  by requiring positivity one restricts the allowed  values of $p$ to an interval of size $\sim 10^{-3}$ around $p=2$, and in \S \ref{sectchangesign} we display  that when $\lambda^2$ grows past a prime power and one ignores its contribution, the quadratic form $QW_\lambda$ fails to remain positive. This fact is   displayed   up to $\lambda^2 \sim 7$.   
One striking numerical result  is described in \S \ref{sectsmall}, where we report numerical evidence  that as $\lambda$ increases  the corresponding operator in $\cH(\lambda):=L^2([\lambda^{-1},\lambda],d^*u)$ admits a finite number of extremely small positive eigenvalues. For instance,  we find that when $\lambda^2=11$  the smallest positive eigenvalue is $2.389\times 10^{-48}$. 
 The corresponding eigenfunctions are graphically reported   in Figures \ref{eigen1}, \ref{eigen2}, \ref{eigen3}.
\newline
 Section \ref{riemweilexpl} explains conceptually the presence of these extremely small positive eigenvalues and  there we also give an excellent approximation of the related eigenfunctions. The theoretical  reason for the presence of these extremely small eigenvalues springs from the fact  that the  radical of the Weil quadratic form contains the range of the  map $\cE$ of  \cite{Co-zeta}, that is defined on the codimension two subspace $\sr0$ of even Schwartz functions fulfilling $f(0)=\widehat f(0)=0$ by 
\begin{equation}
\cE(f)(x):=x^{1/2}\sum_{n>0}f(nx) \qqq f\in \sr0.\label{mapeeintro}
\end{equation}
Even though  RH implies that $QW_\lambda$ is strictly positive, and thus that its radical is $\{0\}$, by making use of \eqref{mapeeintro}, one can nevertheless  construct   functions $g$ with  support in $[\lambda^{-1},\lambda]$ which are in the ``near radical'' of the Weil quadratic form \ie fulfill $QW_\lambda(g)\ll \Vert g\Vert^2$.  More precisely, if the support of the even function $f\in \sr0$ is contained in the interval $[-\lambda,\lambda]\subset \R$, the support of $\cE(f)$ is contained in $(0,\lambda]\subset \R_+^*$. On the other hand, the Poisson formula 
\begin{equation}
\cE(\widehat f)(x)=\cE(f)(x^{-1}) \qqq f\in \sr0\label{poissonintro}
\end{equation}
 shows that the support of $\cE(f)$ is contained in $[\lambda^{-1},\infty)$ provided the support of the even function $\widehat f$ is contained in the interval $[-\lambda,\lambda]\subset \R$. The obstruction to obtain an element $\cE(f)$ of the radical of $QW_\lambda$ is the equality $\cP_\lambda \cap \widehat\cP_\lambda=\{0\} $, where  $\cP_\lambda$ and $\widehat{\cP_\lambda}$ are the cutoff projections in the Hilbert space $L^2(\R)^{\rm ev}$ of square integrable even functions  (the projection $\cP_\lambda$ is given by the multiplication by the characteristic function of the interval  $[-\lambda,\lambda]\subset \R$, the projection $\widehat{\cP_\lambda}$ is its conjugate by the Fourier transform $\fourierer$). The seminal work of Slepian and Pollack \cite{Slepian, Sl, Slepian0} on band limited functions then shows that while $\cP_\lambda \cap \widehat\cP_\lambda=\{0\}$, the angle operator between these two projections admits a finite number $1+\nu(\lambda^2)\sim 2 \lambda^2$ of extremely small non zero eigenvalues and that  the corresponding eigenfunctions are the prolate spheroidal wave functions  
$$
\psi_{m,\lambda}(x):=\text{\textit{PS}}_{2m,0}\left(2 \pi  \lambda^2,\frac{x}{\lambda}\right), \  \ m\leq \nu(\lambda^2)\sim 2 \lambda^2.
$$
By construction, each $\psi_{m,\lambda}$ is a function on the interval $[-\lambda,\lambda]$ that one extends by $0$ outside that interval. Its Fourier transform $\fourier_{e_\R}(\psi_{m,\lambda})$ restricted to the interval $[-\lambda,\lambda]$, is equal to $\chi_m\psi_{m,\lambda}$  where the scalar $\chi_m$ is very close to $(-1)^m$ provided that $m$ is less than $\nu(\lambda^2)\sim 2 \lambda^2$. After taking care of the two conditions $f(0)=\widehat f(0)=0$,  the restriction of $\cE(f)$ to the interval $[\lambda^{-1},\lambda]$ gives rise to a function which we call a ``prolate vector'', and  on which $QW_\lambda$ takes non-zero, but extremely small values. This fact is verified concretely in Section \ref{riemweilexpl}  where we compare the eigenvectors of the Weil quadratic form $QW_\lambda$  associated to its smallest eigenvalues with the orthogonalisation of the prolate vectors  obtained using the technique outlined above, from the prolate spheroidal wave functions. 
\newline  
The construction of the spectral triple $\Theta(\lambda,k)$ is carried out in   Section \ref{spectrip}. Even though this construction is motivated by the results of Section \ref{riemweilexpl} on the near radical of the Weil quadratic form $QW_\lambda$, the technique involved  only uses the prolate vectors  without any reference to $QW_\lambda$. Using the first $k+2$ prolate functions, one obtains a $k$-dimensional subspace of $L^2( [\lambda^{-1},\lambda],d^*u)\simeq L^2(\R_+^*/\lambda^{2\Z},d^*u)$, then  one lets $\Pi(\lambda,k)$ be the associated orthogonal projection. By definition, the spectral triple $\Theta(\lambda,k)=(\cA(\lambda),\cH(\lambda),D(\lambda,k))$ is given by the action by multiplication of the algebra of smooth functions $\cA(\lambda):=C^{\infty}(\R_+^*/\lambda^{2\Z})$ on $\cH(\lambda):=L^2(\R_+^*/\lambda^{2\Z},d^*u)$, while the operator $D(\lambda,k)$ is the finite rank perturbation
\begin{equation}\label{iDintro}
D(\lambda,k):=(1-\Pi(\lambda,k))\circ D_0\circ (1-\Pi(\lambda,k)), \ \  D_0=-iu\partial_u
\end{equation}
of the standard Dirac operator $D_0=-iu\partial_u$ (with periodic boundary conditions when viewed in $L^2( [\lambda^{-1},\lambda],d^*u)\simeq L^2(\R_+^*/\lambda^{2\Z})$). We compute the low lying spectra of these spectral triples and find a neat resemblance with the low lying zeros of the Riemann zeta function, provided that $k$ is sufficiently close to the largest allowed value $\nu(\lambda^2)$.  On the other hand, since the eigenvalues $\lambda_n(D(\lambda,k))$ vary with $\lambda$, one cannot expect that they reproduce exactly the $n$-th zero of the zeta function. The subtlety of the relation is  explained in Section \ref{2zeros}, where we produce  several criterions  to recover the zeros of the Riemann zeta function using the  non-zero eigenvalues $\lambda_n(D(\lambda,k))$. First we show that for $k=2\ell$, the eigenvalues fulfill the inequality $\lambda_n(D(\lambda,k+1)) \leq \lambda_n(D(\lambda,k))$, then we prove (\S \ref{sectcriter})  that for certain  values of $\lambda$ one has $\lambda_n(D(\lambda,k+1)) \simeq \lambda_n(D(\lambda,k))$. When this happens and $k$ is close enough to the upper bound $\nu(\lambda^2)$, the common eigenvalue coincides with the imaginary part of the $n$-th zero of the zeta function. This result is strengthened in \S \ref{sectevoleigen}, where we plot the evolution of the eigenvalues $\lambda_n(D(\lambda,k))$, as functions of $\mu=\lambda^2$, for fixed $k$, and find that several graphs coincide at the above special values of $\lambda$ as displayed in Figure \ref{figquant2}. In \S \ref{sectquantiz} we find that the obtained special points in the $(\mu, \lambda_n)$ plane fulfill the quantization condition 
$$
\mu^{i \lambda_n}=1.
$$
This result suggests that for the above special values of $\lambda$ one has  an eigenvector which is already an eigenvector of the unperturbed Dirac operator $D_0$. In  \S \ref{sectcommon}  we apply this criterion to select the special values of $\lambda$, and compute the 31 first zeros of the Riemann zeta function with the precision shown in Figure \ref{figquant3}. 
The conceptual explanation of these experimental findings  is Theorem \ref{introspectralreal} whose proof is developed  in the final Section \ref{sectzetacycles}. 

\section{The  semi-local Weil quadratic form}\label{riemweilexpl0}
%\vspace{1in}

In this section we test numerically the positivity of the Weil quadratic form $QW(f,g)$, in the semi local case, namely for test functions $f,g$ with support in the interval $[\lambda^{-1},\lambda]$. This investigation breaks down in two independent cases so that $QW_\lambda=QW^+_\lambda\oplus QW^-_\lambda$, according  to the parity of  $f$ and $g$ with respect to the symmetry operator $u\mapsto u^{-1}$. Lemma \ref{polarize} shows that for real test functions, the even functions  do not interfere with the odd ones. Moreover by construction the positivity of  $QW_\lambda$ depends on the length $L=2\log \lambda$  of the support of the test functions. We define in \eqref{basis} an orthonormal basis $\{\eta_n$, $n\in \Z\}$ of the Hilbert space $L^2([\lambda^{-1},\lambda])$ formed by odd (for the symmetry $u\mapsto u^{-1}$) real functions for $n<0$, and by even  real functions for $n\geq 0$. The matrix $\sigma(n,m)=QW(\eta_n,\eta_m)$ is the direct sum $\sigma=\sigma^+\oplus \sigma^-$ of two infinite symmetric real matrices, each of which is expressed as a finite sum involving the archimedean contribution $-W_\R$, as well as the contribution $-W_p$    from primes $p$ less than $\mu=\lambda^2$. The numerical tests consist in evaluating the eigenvalues of the very large portion of these matrices
corresponding to indices  $n$ and $m$  whose absolute values are $\leq N$.  These computations give significant evidence that the increasing of large $N$ does not alter substantially the lower part of the spectrum   of  $\sigma(n,m)$.  In \S\ref{sectsensitive} we find that the  archimedean contribution $-W_\R$ to the Weil quadratic form  when  taken separately, ends to be positive if  computed    in an interval extending slightly  beyond the value $L=\log 2$ (Figure~\ref{testeven1}). However, the  positivity  is restored after that value, and precisely in the interval $\log 2\le L<\log 3$, by implementing also the contribution of the prime $p=2$, in terms of the related functional $-W_2$. In \S\ref{sectsensitivep2} we  report our numerical findings supplying evidence to the fact that the sign of  $QW_\lambda$ is  also sensitive  to the replacement of   $-W_2$ by a functional $-W_p$  whose definition uses the same formula as $-W_2$ but replaces $2$ with $p$,  taken  as a real variable in a small neighborhood of $p=2$. Indeed the  computations show  that the positivity of the quadratic form fails if one considers  real values of $p$ outside an interval of size $<10^{-3}$ around $2$. In \S\ref{sectchangesign} we  report graphical evidence indicating how important is the contribution of each  functional $-W_p$ to preserve the positivity of the quadratic form, if the support of the test function stretches beyond a prime power $p^n$. Finally, in \S\ref{sectsmall} we display numerical evidence of the key fact that by suitably increasing the support of the test functions, the ``even'' and `` odd" matrices  $\sigma^\pm$   admit a finite number of extremely small positive eigenvalues.  The theoretical discussion of this result is presented in section \ref{riemweilexpl}.

\subsection{The matrix $\sigma=\sigma^+\oplus \sigma^-$}\label{sectexpexpl}
This  subsection describes the choice of test functions used  in this paper while performing the numerical computations. When viewed in Hilbert theoretic terms the restriction $QW_\lambda$ of the Weil quadratic form to functions with support in the interval $[\lambda^{-1},\lambda]$ is a lower bounded, lower semi-continuous quadratic form defined on the Hilbert space $\cH:=L^2([\lambda^{-1},\lambda],d^*u)$ with values in $(-\infty,+\infty]$. We choose an orthonormal basis $\{\eta_n\}_{n\in \Z}$ of $\cH$ which is a core for $QW_\lambda$ and compute the eigenvalues of very large portions of the associated matrix $QW(\eta_n,\eta_m)=\sigma(n,m)$. 
\subsubsection{Explicit formula}\label{subsectexplicit}
~Following \cite{EB}, one considers the class $\cW$ of complex valued functions $f$ on $\R^*_+$ which are continuous and with continuous derivative except  at finitely many points where both $f(x)$ and $f'(x)$ have at most a discontinuity of the first kind, and at which the value of $f(x)$ and $f'(x)$ is defined as the average of the right and left limits. Moreover one assumes  that for some $\delta>0$ one has 
$$
f(x)=O(x^\delta), \ \text{for} \ x\to 0+, \ \ f(x)=O(x^{-1-\delta}), \ \text{for} \ x\to +\infty.
$$  
The Mellin transform of   $f \in \cW$ is defined as
\begin{equation}\label{mellin}
 \tilde f(s):=\int_0^\infty f(x)x^{s-1}dx
 \end{equation}
Let $f^\sharp(x):=x^{-1}f(x^{-1})$, then Weil's explicit formula takes the form 
 \begin{equation}\label{bombieriexplicit}
 \sum_{\rho}\tilde f(\rho)=\int_0^\infty f(x)dx+\int_0^\infty f^\sharp(x)dx-\sum_v {\mathcal W}_v(f),
 \end{equation}
 where the  sum on the left hand side is over all complex zeros $\rho$ of the Riemann zeta function, and the sum on the right hand side   runs over all rational places  $v$ of $\Q$. The non-archimedean distributions $\mathcal W_p$ are defined as 
 \begin{equation}\label{bombieriexplicit1}
 {\mathcal W}_p(f):=(\log p)\sum_{m=1}^\infty\left(f(p^m)+f^\sharp(p^m)\right)
 \end{equation}
while the archimedean distribution is given by 
 \begin{equation}\label{bombieriexplicit2}
 {\mathcal W}_\R(f):=(\log 4\pi +\gamma)f(1)+\int_{1}^\infty\left(f(x)+f^\sharp(x)-\frac 2x f(1)\right)\frac{dx}{x-x^{-1}}.
 \end{equation}
 The  translation to (equivalent) formulas using the Fourier transform (in place of the Mellin transform) is done  by implementing the  automorphism $\Delta$ 
 \begin{equation}\label{MF1}
f\mapsto \Delta^{1/2}f=F, \quad F(x)= x^{1/2}f(x)
\end{equation}
which respects the convolution product and  satisfies the equalities
$$
(\Delta^{1/2}f^\sharp)(x)=x^{1/2}f^\sharp(x)=x^{-1/2}f(x^{-1})=(\Delta^{1/2}f)(x^{-1}).
$$
After taking complex conjugates, $\Delta$ is compatible with the natural involutions. 
For a rational place $v$,  we set 
$W_v(F):={\mathcal W}_v(\Delta^{-1/2}F),$   then  the above distributions $\mathcal W_p$ take the following  form
\begin{equation}\label{bombieriexplicit1bis}
 W_p(F)=(\log p)\sum_{m=1}^\infty p^{-m/2}\left(F(p^m)+F(p^{-m})\right).
 \end{equation}
 Using the multiplicative version $d^*x=dx/x$ of the Haar measure, the archimedean distribution $\mathcal W_\R$ becomes
 \begin{equation}\label{bombieriexplicit2bis}
 W_\R(F):=(\log 4\pi +\gamma)F(1)+\int_{1}^\infty\left(F(x)+F(x^{-1})-2x^{-1/2} F(1)\right)\frac{x^{1/2}}{x-x^{-1}}d^*x.
 \end{equation}
 \subsubsection{The semi-local Weil quadratic form}
~The Weil quadratic form is now re-written as 
\begin{equation}\label{bombtest}
QW(f,g)=\psi(f^**g), \quad\psi_{}(F):={\widehat F}(i/2)+{\widehat F}(-i/2)- W_\R(F)-\sum_{p} W_p(F)
 \end{equation}
 where ${\widehat F}(s):=\int F(u)u^{-is}d^*u$ denotes the Fourier transform of the function $F$. Moreover the functional $W_\infty:=- W_\R$ fulfills the following formula 
 \begin{equation}\label{thetaprime}
 W_\infty(F)=\int 
{\widehat F}(t)\frac{2\partial_t\theta(t)}{2 \pi}dt
\end{equation}
in terms of the derivative of the angular Riemann-Siegel function $\theta(t)$
 \begin{equation}\label{riesie}
\theta(t) = - \frac{t}{2} \log \pi + \Im \log \Gamma \left(
\frac{1}{4} + i \frac{t}{2} \right),
\end{equation}
with $\log \Gamma(s)$, for $\Re (s)>0$, the branch of the $\log$
which is real for $s$ real. \newline
By a lower bounded, lower semi-continuous  (lsc) quadratic form $q$ on a Hilbert space $\cH$ we mean a lower semi-continuous map\footnote{\ie such that when $\xi_n\to \xi$ one has $q(\xi)\leq \liminf q(\xi_n)$} $q:\cH\to (-\infty,+\infty]$ which fulfills $q(\lambda \xi)=\vert \lambda\vert^2 q(\xi)$ for all $\lambda \in \C$,  the parallelogram law 
$$
q(\xi + \eta) +q(\xi-\eta)=2 q(\xi) +2 q(\eta)
$$
and also  an inequality of the form $q(\xi)\geq -c \Vert \xi\Vert^2$ for all $\xi \in \cH$,  reflecting the lower bound of $q$.
  The associated sesquilinear form (antilinear in the first variable) is given on the domain of $q$, 
  $
  {\rm Dom}(q):=\{\xi \in \cH\mid q(\xi)<\infty\}
  $  by 
  $$
  q(\xi,\eta):=\frac 14\left( q(\xi + \eta) -q(\xi-\eta)+iq(i\xi + \eta) -iq(i\xi-\eta)\right).
  $$
  By a result of Kato (see \cite{Simon}, Theorem 2) such lower bounded quadratic forms correspond to lower bounded densely defined self-adjoint operators $T\geq -c$ on $\cH$ by the formula 
  $$
  q(\xi )+c\Vert \xi\Vert^2= \langle (T+c)^{\frac 12}\xi\mid (T+c)^{\frac 12}\xi\rangle=
  \Vert (T+c)^{\frac 12}\xi\Vert^2 \qqq \xi \in \cH.
  $$
  At the informal level this means  that $q(\xi,\eta )=\langle \xi\mid T\eta\rangle$.
\begin{proposition}\label{Hilbert} Let $\lambda>1$. The following formula defines a lower bounded lower semi-continuous quadratic form $QW_\lambda:L^2([\lambda^{-1},\lambda],d^*u)\to (-\infty,+\infty]$
	 \begin{equation}\label{quadratsemi}
 QW_\lambda(f,f):=\int 
\vert{\widehat f}(t)\vert^2\frac{2\partial_t\theta(t)}{2 \pi}dt + 2 \Re\left({\widehat f}(\frac i2)\bar {\widehat f}(-\frac i2)\right)-\sum_{1<n\leq \lambda^2} \Lambda(n)\langle f\mid V(n)f\rangle
\end{equation}
where $\Lambda(n)$ is the von Mangoldt function and  $V(n)$ is the bounded self-adjoint operator in $L^2([\lambda^{-1},\lambda],d^*u)$ such that 
\begin{equation}\label{quadratsemi1}
 \langle f\mid V(n)g\rangle=n^{-1/2}\left((f^**g)(n)+(f^**g)(n^{-1})\right ).
\end{equation}
\end{proposition}
\proof The function $\partial_t\theta(t)$ is even, lower bounded and of the order of $O(\log \vert t\vert) $ for $\vert t\vert\to \infty$. This shows that the first term ($Q_\infty(f)=\int\vert{\widehat f}(t)\vert^2\frac{2\partial_t\theta(t)}{2 \pi}dt$)  in \eqref{quadratsemi} defines  a lower bounded, lower semi-continuous quadratic form $Q_\infty$ on the Hilbert space $L^2(\R_+^*,d^*u)$. We view $L^2([\lambda^{-1},\lambda],d^*u)$ as the closed subspace of functions which vanish outside $[\lambda^{-1},\lambda]$. The restriction of the  quadratic form $Q_\infty$ to $L^2([\lambda^{-1},\lambda],d^*u)$ is still lower semi-continuous and lower  bounded, moreover the domain $\{\xi\in L^2([\lambda^{-1},\lambda],d^*u)\mid Q_\infty(\xi,\xi)<\infty\}$ is dense since it contains all smooth functions with support in $(\lambda^{-1},\lambda)$. Thus it remains to show that each of  the remaining terms can be written in the form $\langle f\mid Tg\rangle$ with $T$ bounded and self-adjoint in $L^2([\lambda^{-1},\lambda],d^*u)$. One has 
$$
{\widehat f}(\frac i2)=\int_{\lambda^{-1}}^\lambda f(u)u^{1/2}d^*u=\langle h\mid f\rangle,  \quad h(u):=u^{1/2} \qqq u\in [\lambda^{-1},\lambda].
$$
Thus the term $2 \Re\left({\widehat f}(\frac i2)\bar {\widehat f}(-\frac i2)\right)$  in \eqref{quadratsemi} is of  the form $\langle f\mid Tf\rangle$, where $T$ is the sum of the rank one operators $T=\vert h\rangle \langle h^*\vert+\vert h^*\rangle \langle h\vert$. Let us show that  
  \eqref{quadratsemi1} defines a bounded self-adjoint operator $V(n)$ in $L^2([\lambda^{-1},\lambda],d^*u)$. One has 
  $$
  (f^**g)(v)=\int f^*(vu^{-1})g(u)d^*u= \int \overline{f(v^{-1}u)}g(u)d^*u
  $$
  so that by the Cauchy-Schwarz inequality one derives $\vert(f^**g)(v)\vert\leq \Vert f\Vert \Vert g\Vert$. This shows that the equality  $(f^**g)(v)=\vert f\rangle \langle V_vg\vert$ defines a bounded operator $V_v$ in $L^2([\lambda^{-1},\lambda],d^*u)$. Moreover its adjoint is $V_v^*=V_{v^{-1}}$ and thus $V(n)$ is bounded and self-adjoint. \endproof 
  \begin{lemma}\label{core}  Let $\lambda>1$, and $U\in L^2([\lambda^{-1},\lambda],d^*u)$ be the function $U(u):=u^{\frac{i\pi }{\log \lambda}}$. Then the space of Laurent polynomials $\C[U,U^{-1}]$ is a core for the quadratic form $QW_\lambda$.  	
  \end{lemma}
  \proof  Since all terms in \eqref{quadratsemi} are bounded except for $Q_\infty$, it is enough to show that $\C[U,U^{-1}]$ is a core for the quadratic form $Q_\infty$ restricted to $L^2([\lambda^{-1},\lambda],d^*u)$. First note that the powers $U^n$ belong to the domain of
  $Q_\infty$ since the Fourier transform of $U^n$ is given by
  $$
 \widehat{U^n}(s)= \int_{\lambda^{-1}}^\lambda U^n(u)u^{-is}d^*u=\Big[u^{\frac{i\pi n}{\log \lambda}-is}\left(\frac{i\pi n}{\log \lambda}-is\right)^{-1}\Big]_{\lambda^{-1}}^\lambda=
 $$ 
  \begin{equation}\label{quadratsemi3.5}
 =2 \log \lambda (-1)^n \sin(s\log \lambda)\left(\pi n -s\log \lambda\right)^{-1}=O(1/\vert s\vert).
  \end{equation}
  \newline
  Thus the integral of  $\vert\widehat{U^n}(s)\vert^2 \partial_t\theta(t)$ is absolutely convergent and $U^n$ belongs to the domain of
  $Q_\infty$. 
    Next, given $\xi\in L^2([\lambda^{-1},\lambda],d^*u)$ such that $Q_\infty(\xi,\xi)<\infty$, we want to show that for any $\epsilon>0$ there exists $\eta \in \C[U,U^{-1}]$ such that (with $-c$ the lower bound of $Q_\infty$)
  \begin{equation}\label{quadratsemi2}
  (1+c)\Vert \xi-\eta\Vert^2 +Q_\infty(\xi-\eta,\xi-\eta)<\epsilon.
\end{equation}
  We switch from $\R_+^*$ to the additive group $\R$ (using the logarithm) and let $L=2\log \lambda$ and $\cH= L^2([-L/2,L/2])\subset L^2(\R)$. Under this change of variable the function $U$ becomes $U(x)=\exp(\frac{2\pi i x}{L})$. Moreover using the Fourier transform on $\R\simeq \widehat \R$ one can write
    \begin{equation}\label{quadratsemi3}
  Q_\infty(f,f):=\int 
\vert{\widehat f}(t)\vert^2\frac{2\partial_t\theta(t)}{2 \pi}dt.
\end{equation}
 In view of the asymptotic expansion  for $\vert t\vert \to \infty$  
$$
\partial_t\theta(t)=\frac{1}{2} (\log (\vert t\vert)-\log (2)-\log (\pi ))-\frac{1}{48 t^2}+O\left(t^{-4}\right)
$$
one can replace \eqref{quadratsemi2} by an equivalent condition of finding, given $\epsilon>0$, a Laurent polynomial $\eta$ in $U(x)=\exp(\frac{2\pi i x}{L})$ for $x\in [-L/2,L/2]$, such that  
\begin{equation}\label{quadratsemi3}
  \int\vert {\widehat \xi}(s)-{\widehat \eta}(s)\vert^2 (1+\log (1+s^2))ds<\epsilon.
\end{equation}
 We first replace $\xi$ by $\xi_1$ with $\xi_1(x):=\rho\,\xi(\rho x)$ for $\rho>1$, while  
 \begin{equation}\label{quadratsemi4}
  \int\vert {\widehat \xi}(s)-{\widehat \xi_1}(s)\vert^2 (1+\log (1+s^2))ds<\epsilon/2
\end{equation}
This latter inequality  holds provided $\rho$ is close enough to $1$. Indeed one has 
$$
\int\vert {\widehat \xi}(s)\vert^2 (1+\log (1+s^2))ds<\infty, \quad{\widehat \xi_1}(s)={\widehat \xi}(\rho^{-1}s),
$$
while the scaling action $S$ of $\R_+^*$ on the Hilbert space of functions on $\R$ with the norm
$$
\Vert f\Vert_1^2:=\int\vert f(s)\vert^2 (1+\log (1+s^2))ds
$$
is pointwise norm continuous. Note first that $S(\rho)$ is bounded uniformly near $\rho=1$ since 
$$
\int\vert f(\rho^{-1}s)\vert^2 (1+\log (1+s^2))ds=\rho\int\vert f(t)\vert^2 (1+\log (1+\rho^2t^2))dt
$$
while $(1+\log (1+\rho^2t^2))\leq 2(1+\log (1+t^2))$ for all $t\in \R$  for $\rho\leq 2$. The pointwise norm continuity of $S(\rho)$ follows since this action is pointwise norm continuous on the dense subspace of continuous functions with compact support. Now the support of $\xi_1$, $\xi_1(x):=\rho\,\xi(\rho x)$ is contained in the interval $[-\rho^{-1}L/2, \rho^{-1}L/2]$. Thus the convolution $\xi_2=\phi*\xi_1$ with a smooth function $\phi\in C_c^{\infty}(\R)$ whose support is contained in a small enough neighborhood of $0$ is a smooth function with support in the interior of the interval 
$[-L/2,L/2]$. Fix such a $\phi$ positive with integral equal to $1$. Let $\phi_n(x)=n\phi(nx)$, and let $\eta_n=\phi_n*\xi_1$. One has $\widehat \eta_n(s)=\widehat \phi(s/n)\widehat\xi_1(s)$.  The functions $\widehat \phi(s/n)$ are bounded  $\vert\widehat \phi(s/n)\vert\leq 1$ and converge pointwise to $1$, thus the Lebesgue dominated convergence theorem shows that for $n$ large enough, one has
 \begin{equation}\label{quadratsemi5}
  \int\vert {\widehat \xi_1}(s)-{\widehat \eta_n}(s)\vert^2 (1+\log (1+s^2))ds<\epsilon/4.
\end{equation}
Finally since $\eta_n\in C_c^{\infty}((-L/2,L/2))$ it can be viewed as a smooth function on the circle obtained by identifying the end points of the interval $[-L/2,L/2]$ so that there exists a sequence $(a_k)_{k\in \Z}$ of rapid decay such that $\eta_n=\sum a_k U^k$. Equation \eqref{quadratsemi3.5} still holds after the change of variables and one gets the equality
$$
\Vert{\widehat U^k}\Vert_1^2=4 (\log \lambda)^2 \int \vert\sin(s\log \lambda)\left(\pi k -s\log \lambda\right)^{-1}\vert^2(1+\log (1+s^2))ds.
$$
One has $\sin^2(s)s^{-2}\leq 2(1+s^2)^{-1}$, and for any $a\in \R$, $1+(s-a)^2\leq 2(1+a^2)(1+s^2)$ so that $\log(1+(s-a)^2)\leq \log 2 +\log (1+s^2)+ \log (1+a^2)$. This shows that $\Vert {\widehat U^k}\Vert_1^2=O(\log\vert k\vert)$ and hence, since $\eta_n=\sum a_k U^k$ where the sequence $(a_k)_{k\in \Z}$ is of rapid decay, one can find $N$ such that $\eta=\sum_{-N}^N a_k U^k$ fulfills
 \begin{equation}\label{quadratsemi6}
  \int\vert {\widehat \eta}(s)-{\widehat \eta_n}(s)\vert^2 (1+\log (1+s^2))ds<\epsilon/4.
\end{equation}
Combining \eqref{quadratsemi4}, \eqref{quadratsemi5}, \eqref{quadratsemi6} one obtains the
required approximation. \endproof 
\begin{proposition}\label{Hilbert1} Let $\lambda>1$. The  quadratic form $QW_\lambda:L^2([\lambda^{-1},\lambda],d^*u)\to (-\infty,+\infty]$ of \eqref{quadratsemi} fulfills, for any $f\in L^2([\lambda^{-1},\lambda],d^*u)$,	 
\begin{equation}\label{quadratsemi7}
 QW_\lambda(f)=\liminf_{g_n\to f} QW_\lambda(g_n), \ g_n\in \C[U,U^{-1}]
\end{equation}
\end{proposition}
\proof By applying the lower semicontinuity of $QW_\lambda$ one sees that in \eqref{quadratsemi7} the lhs is smaller than the rhs. The density of $\C[U,U^{-1}]$ in the domain of  $QW_\lambda$ for the graph-norm shown in Lemma \ref{core}, proves  that if $f$ is in the domain of  $QW_\lambda$ there exists a sequence $g_n$ of elements of  $\C[U,U^{-1}]$ converging to $f$ in norm, and such that $QW_\lambda(f)=\lim_{g_n\to f} QW_\lambda(g_n)$.\endproof 
\begin{corollary}\label{Hilbert2} The lower bound of $QW_\lambda$ is the limit, when $N\to \infty$, of the smallest eigenvalue of the restriction of $QW_\lambda$  to the linear span $E_N$ of the functions $U^k$ for $\vert k\vert \leq N$.	
\end{corollary}

\subsubsection{Basis  of real functions in $\C[U,U^{-1}]$}\label{sectrealbasis}
~In order to compute explicitly the smallest eigenvalue of the restriction of $QW_\lambda$  to the linear span $E_N$ of the functions $U^k$, for $\vert k\vert \leq N$, as in Corollary \ref{Hilbert2}, we first find a convenient orthonormal basis formed of real valued functions. We first consider the Hilbert space $L^2([-L/2,L/2])\subset L^2(\R)$ with the inner product defined by the formula
$$
\langle \xi\mid \eta \rangle:= \int_{-L/2}^{L/2} \overline{\xi(x)}\eta(x)dx 
$$
 An orthonormal real basis is given, by the constant function $\xi_0(x)=L^{-1/2}$ together with  the functions $\xi_n(x)$ $n\in \Z$, $n\neq 0$, defined as follows %\begin{equation}
\begin{align}\label{basis}
 \ \xi_n(x)&:=(-1)^n\Big(\frac 2L\Big)^{1/2}\cos\Big(\frac{2 \pi n x}{L}\Big)\qqq n>0\\
 \xi_n(x)&:=(-1)^n\Big(\frac 2L\Big)^{1/2}\sin\Big(\frac{2 \pi n x}{L}\Big) \ \qqq  n<0.\notag
\end{align}
%\end{equation}
We note the following simple facts
\begin{lemma}\label{polarize}
Let $L>0$, $\phi_j \in L^2([-L/2,L/2])$ and $\theta=\phi_1 * \phi_2^*$. Then\newline
$(i)$~The support of $\theta$ is contained in the interval $[-L,L]$, and for $t\in [0,L]$ one has 
$$
\theta(t)=\int_{t-L/2}^{L/2} \phi_1(x)\overline{\phi_2(x-t)}dx, \quad 
\theta(-t)=\int_{t-L/2}^{L/2} \phi_1(x-t)\overline{\phi_2(x)}dx.
$$
$(ii)$~If the functions $\phi_j$ are real, then: $\phi_1 * \phi_2^*(-t)=\phi_2 * \phi_1^*(t)$.\newline
$(iii)$~If the functions $\phi_j$ are real, with $\phi_1$  even and $\phi_2$  odd, then for all $t\in \R$ one has: $$\phi_1 * \phi_2^*(t)+\phi_2 * \phi_1^*(t)=0.$$
\end{lemma}
\begin{proof} $(i)$~The  function $\theta$ is given, using $\phi_2^ *(x)=\overline{\phi_2(-x)}$, by
$$
\theta(t)=\int_\R \phi_1(x)\overline{\phi_2(x-t)}dx\qquad t\in\R.
$$
Since the integrand is not zero only if $x\in [-L/2,L/2]$ and $x-t\in [-L/2,L/2]$, one can restrict the integration  in the interval $[t-L/2, L/2]$. The second equality follows using the equality $\theta^*=\phi_2 * \phi_1^*$. \newline
$(ii)$~Follows from $\theta^*=\phi_2 * \phi_1^*$.\newline
$(iii)$~Notice that $\phi_2^*=-\phi_2$ since $\phi_2$ is real and odd, and $\phi_1^*=\phi_1$ since $\phi_1$ is real and even, thus one derives: $\phi_1*\phi_2^*+\phi_2*\phi_1^*=-\phi_1*\phi_2+\phi_2*\phi_1=0$.
\end{proof} 
It is convenient to rewrite the Weil sesquilinear form $QW(f,g)=\psi(f^**g)$  using the natural invariance of the functional $\psi$ under the symmetry $ h^\sigma(u):=h(u^{-1})$. Thus 
\begin{equation}\label{weilQexp}
 \psi(h)=\psi^\#(h+h^\sigma),\ h^\sigma(u):=h(u^{-1})
\end{equation}
where 
\begin{equation}\label{bombtestsum}
\psi^\#(F):=W_{0,2}^\#(F)- W_\R^\#(F)-\sum W_p^\#(F)
 \end{equation}
with
\begin{equation}\label{bombtest2}
W_{0,2}^\#(F)=\int_{1}^{\infty} F(x)(x^{1/2}+x^{-1/2})d^*x
 \end{equation}
\begin{equation}\label{bombtest1}
W_\R^\#(F)=\frac 12(\log 4\pi +\gamma)F(1)+\int_{1}^\infty\frac{x^{1/2}F(x)-F(1)}{x-x^{-1}}d^*x
 \end{equation}
 \begin{equation}\label{bombtest3}
  W_p^\#(F)=(\log p)\sum_{m=1}^\infty p^{-m/2}F(p^m).
 \end{equation}
 
 \begin{lemma}\label{polarize2}
 With $\eta_n(u):=\xi_n(\log u)$,  the $\eta_j$, $\vert j\vert\leq n$ form an orthonormal basis of $E_n$.\newline 
  $(i)$~The matrix of the Weil sesquilinear form is given by the following formula 
\begin{equation}\label{weilformbase}
QW_\lambda(\eta_n,\eta_m)=\sigma(n,m)=
\psi^\#(h), \quad h(u)=\left(\xi_n * \xi_m^*+\xi_m*\xi_n^*\right)(\log u),
\end{equation}
 where $\psi^\#(h)$ is defined in \eqref{bombtestsum}.\newline
$(ii)$~For $n\geq 0$, $m< 0$ one has  $
(\xi_n * \xi_m^*+\xi_m*\xi_n^*)(y)=0\qqq y\in \R
$.\newline
$(iii)$~For $nm>0$, or  $n=0$  and $m\geq 0$,  one has: $\xi_n * \xi_m^*=\xi_m*\xi_n^*$.\newline Furthermore, the convolution $\xi_n * \xi_m^*(y)$  is an even function of $y$  whose explicit description, for $y\in [0,L]$, is given in the following table whose general term gives the function $1/2(\xi_n * \xi_m^*+\xi_m*\xi_n^*)(y)$
$$
\begin{array}{ccccc}
& &n>0& n=0 & n<0 \\
\hline\\
m>0,n\neq m &\vline &\frac{n \,\sin \left(\frac{2 \pi  n y}{L}\right)-m \,\sin \left(\frac{2 \pi  m y}{L}\right)}{\pi  \left(m^2-n^2\right)} & -\frac{\sin \left(\frac{2 \pi  m y}{L}\right)}{\sqrt{2} \pi  m} & 0 \\
m=n>0 &\vline & \frac{(L-y) \cos \left(\frac{2 \pi  n y}{L}\right)}{L}-\frac{\sin \left(\frac{2 \pi  n y}{L}\right)}{2 \pi  n} & \emptyset  & \emptyset  \\
m=0&\vline &-\frac{\sin \left(\frac{2 \pi  n y}{L}\right)}{\sqrt{2} \pi  n} & \frac{L-y}{L} & 0 \\
m<0,n\neq m &\vline &0 & 0 & \frac{m\, \sin \left(\frac{2 \pi  n y}{L}\right)-n\, \sin \left(\frac{2 \pi  m y}{L}\right)}{\pi  \left(m^2-n^2\right)} \\
m=n<0 &\vline&\emptyset  & \emptyset  & \frac{\sin \left(\frac{2 \pi  n y}{L}\right)}{2 \pi  n}+\frac{(L-y) \cos \left(\frac{2 \pi  n y}{L}\right)}{L} \\
\end{array}
$$
\end{lemma}
\begin{proof} $(ii)$~Follows from Lemma \ref{polarize}.\newline
 $(iii)$~The table reported in $(iii)$ is obtained by direct computation. \newline
$(i)$~The formulas reported in $(iii)$ show that all functions involved are in the domain of applicability $\cW$ of the explicit formulas (see \S \ref{subsectexplicit}). Moreover the terms of the explicit formulas correspond to the terms which enter in the definition  \eqref{quadratsemi} of the quadratic form $QW_\lambda$ in Proposition \ref{Hilbert}. 
\end{proof}  
 Lemma \ref{polarize2} $(ii)$  shows that $\sigma$ is a symmetric matrix and that
 $$
 \sigma(n,m)=0\qqq n\geq 0,m< 0.
 $$
 Thus $\sigma$ splits in two blocks $\sigma=\sigma^+\oplus \sigma^-$  which we shall call informally as the ``even '' and  ``odd '' matrices. They correspond to the partition  $\Z=\{n\geq 0\}\cup \{n< 0\} $ and one has
 \begin{equation}\label{weilblocks}
QW_\lambda=QW_\lambda^+\oplus QW_\lambda^-, \ \ \sigma=\sigma^+\oplus \sigma^-.
\end{equation} 
  This decomposition shows that the  positivity of the Weil quadratic form can be tested working separately the cases of  even functions (using the matrix $\sigma^+$)  and odd functions (using $\sigma^-$). In the chosen basis $\eta_n$, the even case corresponds to considering elements of the basis indexed by $n\geq 0$, while  the odd case involves  the $\eta_n$'s indexed by $n< 0$.

\subsubsection{The matrix $w_{0,2}(n,m)$}

~Next, we shall describe  the contribution of the first two terms in \eqref{bombtest} to the matrix $\sigma(m,n)$. The following lemma  shows that  these terms contribute by  a rank one matrix to both  the odd and the even  matrices $\sigma^\pm$.

\begin{lemma}\label{w02} Let $n,m>0$ be positive integers, $\theta =\xi_m*\xi_n^*$,   $F(x)=\theta(\log x)$. The following equality holds
\begin{equation}\label{h02ev}
{\widehat F}(i/2)+{\widehat F}(-i/2)=\frac{8  e^{-\frac{L}{2}} \left(e^{L/2}-1\right)^2 L^3}{\left(L^2+16 \pi ^2 m^2\right) \left(L^2+16 \pi ^2 n^2\right)}.
\end{equation}
If $n,m<0$ are negative integers, then one has
	\begin{equation}\label{h02}
{\widehat F}(i/2)+{\widehat F}(-i/2)=-\frac{256 \pi ^2 L e^{-\frac{L}{2}} \left(e^{L/2}-1\right)^2 m n}{\left(L^2+16 \pi ^2 m^2\right) \left(L^2+16 \pi ^2 n^2\right)}.
\end{equation}
\end{lemma}
\begin{proof} We give the proof of \eqref{h02}; one proves \eqref{h02ev} in a similar manner. One has 
\begin{align*}
&{\widehat F}(i/2)+{\widehat F}(-i/2)=\int_{\R^*_+} F(x)(x^{1/2}+x^{-1/2})d^*x=\int_\R \theta(t)(e^{t/2}+e^{-t/2})dt=\\
&=\int_0^{\infty} (\theta(t)+\theta(-t))(e^{t/2}+e^{-t/2})dt
=\int_0^L (\theta(t)+\theta(-t))(e^{t/2}+e^{-t/2})dt.
\end{align*}
For $m\neq n$   Lemma \ref{polarize2} $(iii)$  with $t\in [0,L]$,  implies 
$$ 
\theta(t)+\theta(-t)=\frac{2}{(m^2-n^2)\pi} \left(m  \sin \Big(\frac{2 \pi n t}{L}\Big)-n  \sin\Big(\frac{2 \pi m t}{L}\Big)\right).
$$
Furthermore one also has 
$$
\int_0^L \sin\Big(\frac{2 \pi n t}{L}\Big)(e^{t/2}+e^{-t/2})dt=-\frac{8 \pi  e^{-\frac{L}{2}} \left(e^{L/2}-1\right)^2 L n}{L^2+16 \pi ^2 n^2}
$$
which gives 
$$
{\widehat F}(i/2)+{\widehat F}(-i/2)=\frac{2}{(m^2-n^2)\pi}\left(-\frac{128 \pi ^3 e^{-\frac{L}{2}} \left(e^{L/2}-1\right)^2 L m n \left(m^2-n^2\right)}{\left(L^2+16 \pi ^2 m^2\right) \left(L^2+16 \pi ^2 n^2\right)}\right)
$$
and this  proves \eqref{h02}. 
When $m=n$,   
 Lemma \ref{polarize2} $(iii)$ gives  with $t\in [0,L]$,
 $$
\theta(t)+\theta(-t)=\frac{1}{n\pi} \sin\Big(\frac{2 \pi n t}{L}\Big) +2\Big(1 -\frac{ t}{L}\Big) \cos\Big(\frac{2 \pi n t}{L}\Big).
$$
Then, one has 
$$
\int_0^L \Big(1 -\frac{ t}{L}\Big)\cos\Big(\frac{2 \pi n t}{L}\Big)(e^{t/2}+e^{-t/2})dt=\frac{4 e^{-\frac{L}{2}} \left(e^{L/2}-1\right)^2 L \left(L^2-16 \pi ^2 n^2\right)}{\left(L^2+16 \pi ^2 n^2\right)^2}$$
which gives 
$$
{\widehat F}(i/2)+{\widehat F}(-i/2)=-\frac{8  e^{-\frac{L}{2}} \left(e^{L/2}-1\right)^2 L }{L^2+16 \pi ^2 n^2}+\frac{8 e^{-\frac{L}{2}} \left(e^{L/2}-1\right)^2 L \left(L^2-16 \pi ^2 n^2\right)}{\left(L^2+16 \pi ^2 n^2\right)^2}.
$$
This argument shows  \eqref{h02} using the equality: $-(L^2+16 \pi ^2 n^2)+(L^2-16 \pi ^2 n^2)=-32 \pi ^2 n^2$.
\end{proof} 

\subsubsection{The sum $\sum W_p$}
~The contribution of the non archimedean primes is given by \eqref{bombieriexplicit1bis}, now written as  
\begin{equation}\label{bomp}
\sum W_p=\sum_{1<m\leq \exp(L)}\Lambda(m) m^{-1/2}\left(\xi_n * \xi_m^*+\xi_m*\xi_n^*\right)(\log m).
 \end{equation}
 
\subsubsection{The functional $W_\R$}
 ~Let $\theta_{\rm sym}(t)=\left(\xi_n * \xi_m^*+\xi_m*\xi_n^*\right)(t)$, then\eqref{bombieriexplicit2bis} reads as \begin{align*}
W_\R &=
\int_0^L\frac{\exp \left(\frac{x}{2}\right) \theta_{\rm sym}(x)- \theta_{\rm sym}(0)}{\exp \left(x\right)-\exp \left(-x\right)}dx- \theta_{\rm sym}(0)\int_L^{\infty}\frac{dx}{\exp \left(x\right)-\exp \left(-x\right)}\\ &+\frac{1}{2} (\gamma +\log (4 \pi )) \theta_{\rm sym}(0).
      \end{align*}
One has 
$$
\int_L^{\infty}\frac{dx}{\exp \left(x\right)-\exp \left(-x\right)}=\frac 12 \log\left(\frac{e^L+1}{e^L-1} \right) 
$$
so that one obtains 
\begin{equation}\label{weinfty}
W_\R =\frac{ \theta_{\rm sym}(0)}{2}\Bigg(\gamma +\log \Bigg(4 \pi \frac{e^L-1}{e^L+1}  \Bigg)\Bigg)+\int_0^L\frac{\exp \left(\frac{x}{2}\right) \theta_{\rm sym}(x)- \theta_{\rm sym}(0)}{\exp \left(x\right)-\exp \left(-x\right)}dx.
\end{equation}
\begin{figure}[H]	\begin{center}
\includegraphics[scale=0.5]{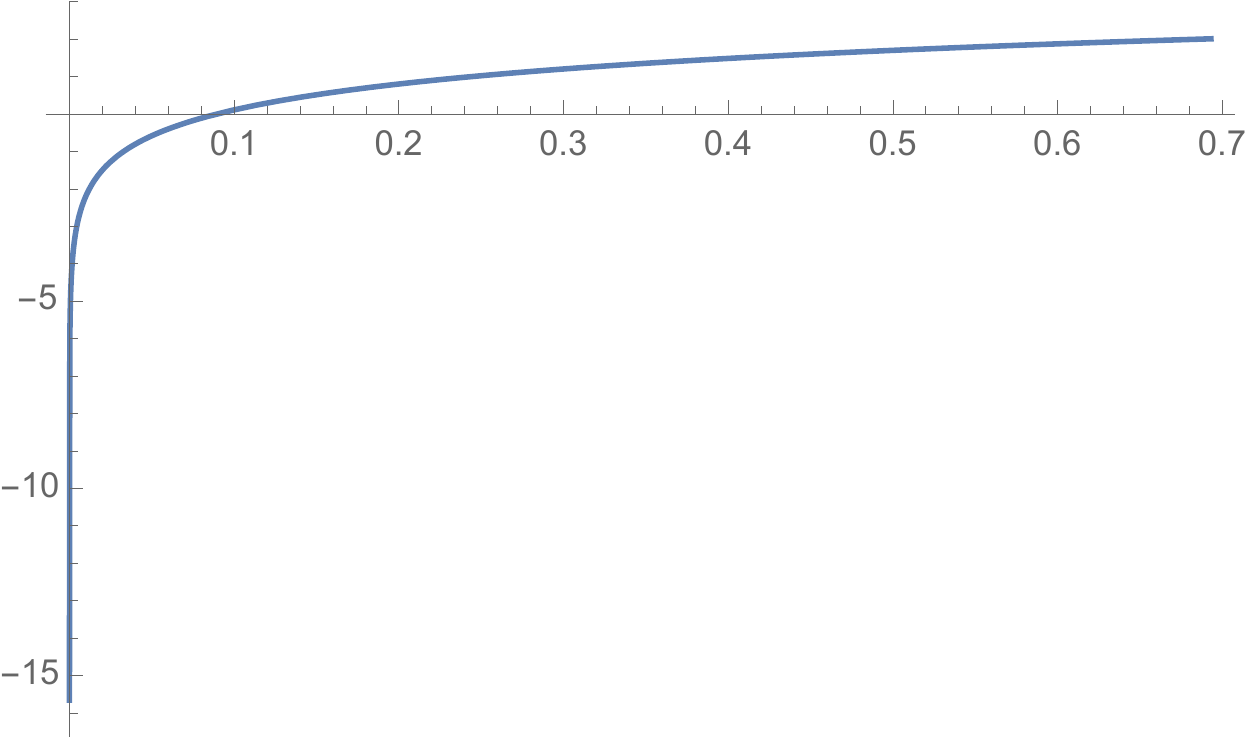}
\end{center}
\caption{Coefficient of $\frac{ \theta_{\rm sym}(0)}{2}$. Its value at $L=\log 2$ is $2.00963$\label{1testeven}}
\end{figure}
Figure \ref{1testeven}  shows that the coefficient of $\frac{ \theta_{\rm sym}(0)}{2}$ is negative near the origin  ($L=0$),   thus its contribution to the quadratic form $QW$ is a positive one for small values of $L$, due to the minus sign in front of $W_\R^+$  (in \eqref{bombtestsum}). This very same contribution becomes negative for larger values of $L$.
%\newpage

\subsection{Sensitivity of Weil positivity, archimedean place}\label{sectsensitive}
The first fact we report  from the numerical computations is that the archimedean contribution fails to remain positive when extended a bit beyond the value  $L=\log 2$. In the following two graphs (Figure \ref{testeven} and \ref{testeven1})  we report the variation of the smallest eigenvalue for the even matrix $\sigma^+$, as the value of $L$ approaches and then stretches a bit beyond $\log 2$. When one considers values of $L$ in the interval $\log 2\leq L < \log 3$, the  contribution of the primes to the Weil quadratic form is only by  $p=2$, and of the form
\begin{equation}\label{wpsym}
W_p(F)= p^{-1/2}\log p\left(\theta(\log p)+\theta(-\log p) \right). \end{equation}
Figure \ref{testeven2} shows that adding the contribution of  the prime $2$ to the archimedean contribution restores the  positivity of the even matrix $\sigma^+$. The graph is in terms of  $\mu:=\exp L$, and this choice of the variable is dictated by the fact that its integer prime power values play a crucial role in this study.

\begin{figure}[H]	\begin{center}
\includegraphics[scale=0.4]{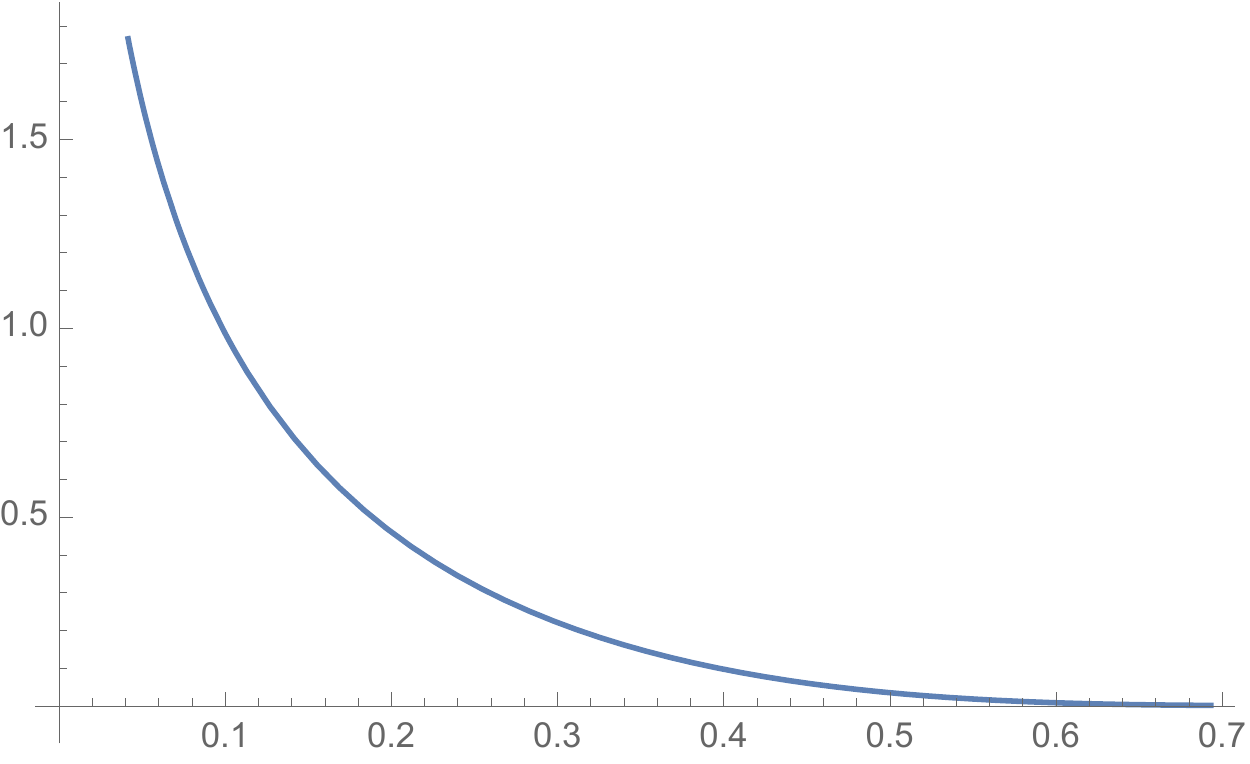}
\end{center}
\caption{Positivity of  the archimedean contribution to the even matrix for  $L\in [0,\log 2]$. The smallest eigenvalue  when $L=\log 2$ is $\sim 0.00133$\label{testeven}}
\end{figure}
\vspace*{-0.2in}
\begin{figure}[H]	\begin{center}
\includegraphics[scale=0.4]{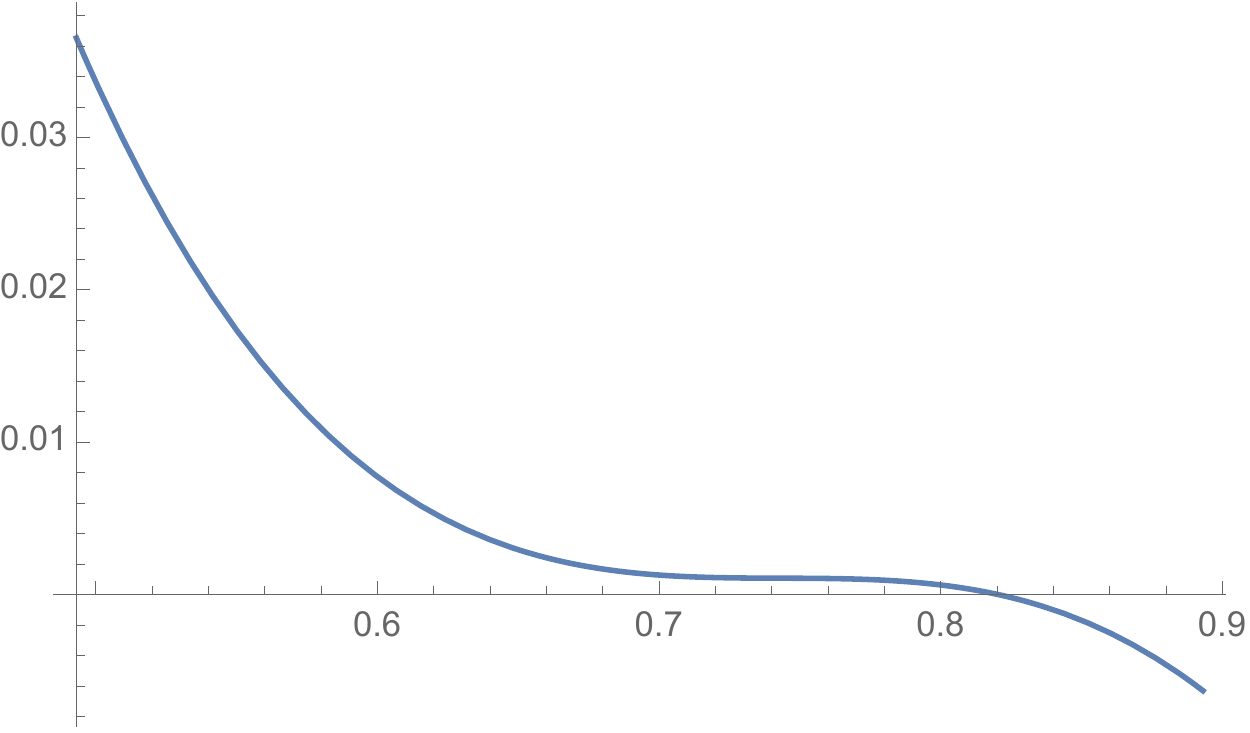}
\end{center}
\caption{Change of sign of the smallest eigenvalue of  the archimedean contribution to the even matrix for $L\in [\log 2-0.2,\log 2+0.2]\sim [0.493,0.893]$\label{testeven1}}
\end{figure}
\begin{figure}[H]	\begin{center}
\includegraphics[scale=0.5]{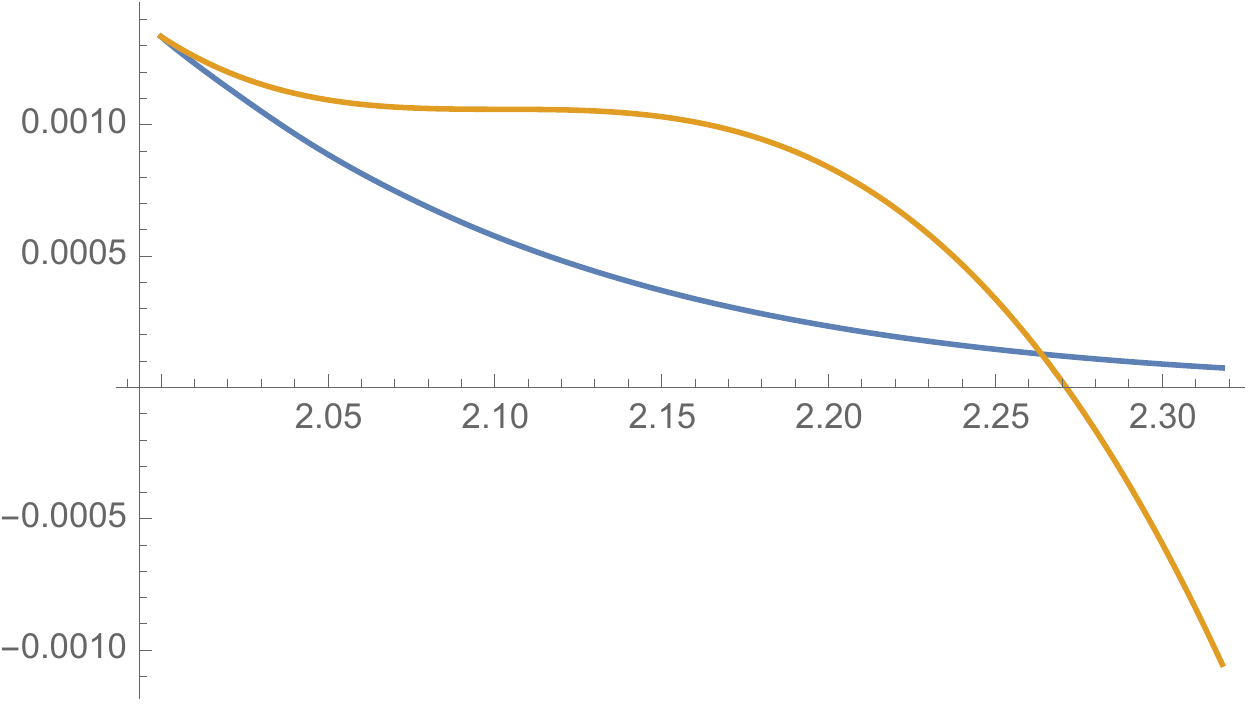}
\end{center}
\caption{Change of sign of the smallest eigenvalue  for the archimedean contribution alone, as a function of $\mu:=\exp L$, near $\mu=2$ (in yellow). After adding the contribution of the prime $2$ the  smallest eigenvalue  of the even matrix is  $>0$ (in blue)\label{testeven2}}
\end{figure}

\subsection{Sensitivity of Weil positivity to the precise value $p=2$}\label{sectsensitivep2}
Figure \ref{testeven2} shows that beyond $\mu=2$ the contribution \eqref{wpsym} of the prime $2$ first lowers the smallest eigenvalue in the interval $\exp L\in (2, 2.27)$ but then saves it from being negative. The value of the smallest eigenvalue of $\sigma^+$ for $\mu= 3$ is $< 6 \times 10^{-8}$. This suggests to use $p$ as a variable in \eqref{wpsym} and to test the sensitivity of Weil positivity to the precise value $p=2$. To this end one fixes $L=\log 3$ (\ie $\mu=3$) and replaces $2$ by a variable $p$ in \eqref{wpsym}.
\begin{figure}[H]	\begin{center}
\includegraphics[scale=0.45]{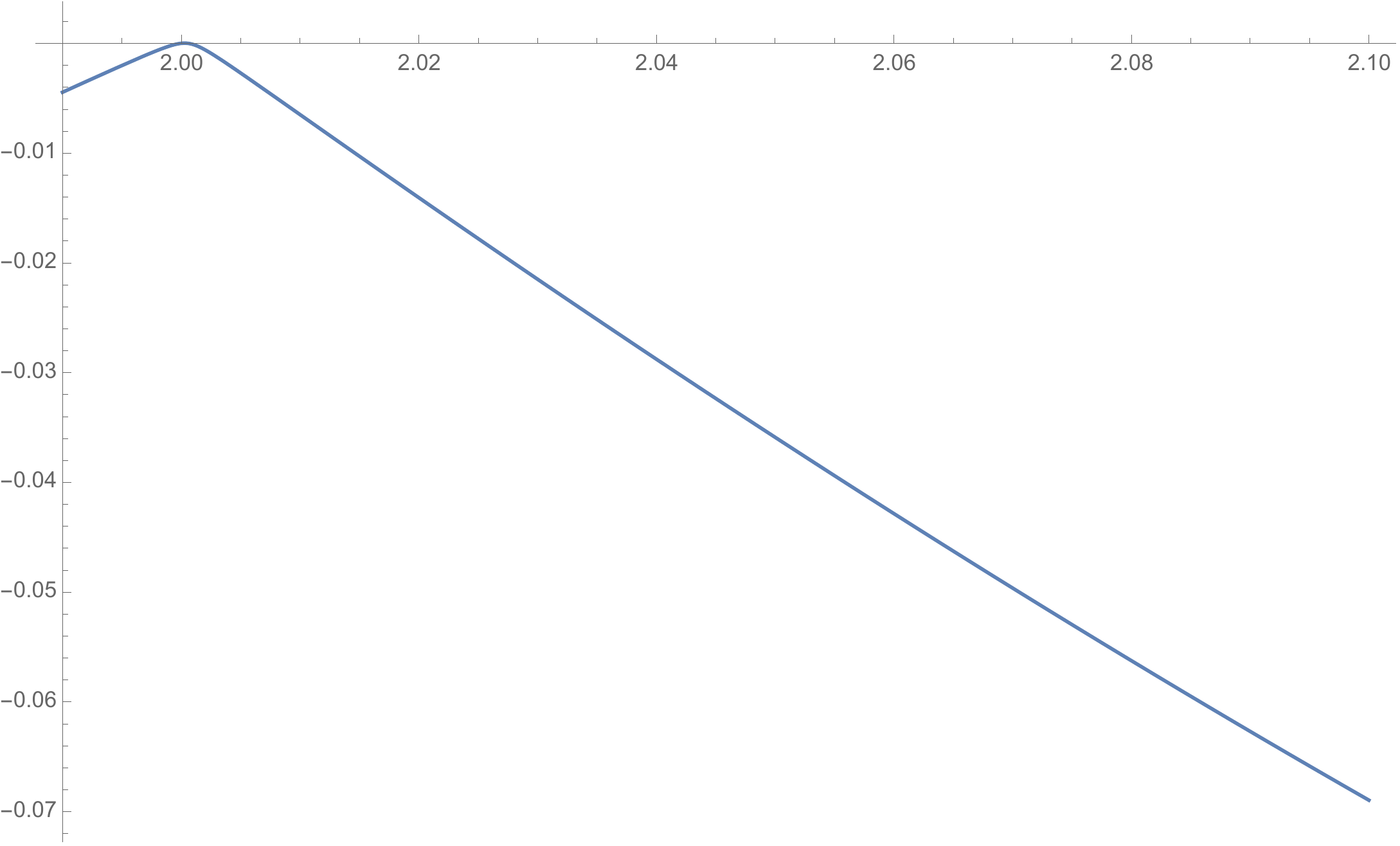}
\end{center}
\caption{Sensitivity  to the precise value $p=2$\label{testeven3}}
\end{figure}
As Figure \ref{testeven3} shows, one finds that the smallest eigenvalue $\lambda(p)$ for $L=\log 3$ is negative for $p=1.9999$ and also for $p=2.0005$, so that the positivity requirement restricts the choice of $p$ to an interval of size $< 10^{-3}$ around $p=2$.

\subsection{Change of sign of smallest eigenvalue}\label{sectchangesign}
%\subsubsection{even matrix} 

 Beyond $p=3$ the sign of the smallest eigenvalue of the sum of the  contributions of $\infty$ and $2$ to the even matrix $\sigma^+$  is reported in yellow in Figure \ref{testeven4}. Once again we notice that its negative behavior beyond $\mu=3$  is ``fixed'' and the output (in blue in the figure) switches to be positive by adding the  contribution of  the prime $3$.

\begin{figure}[H]	\begin{center}
\includegraphics[scale=0.6]{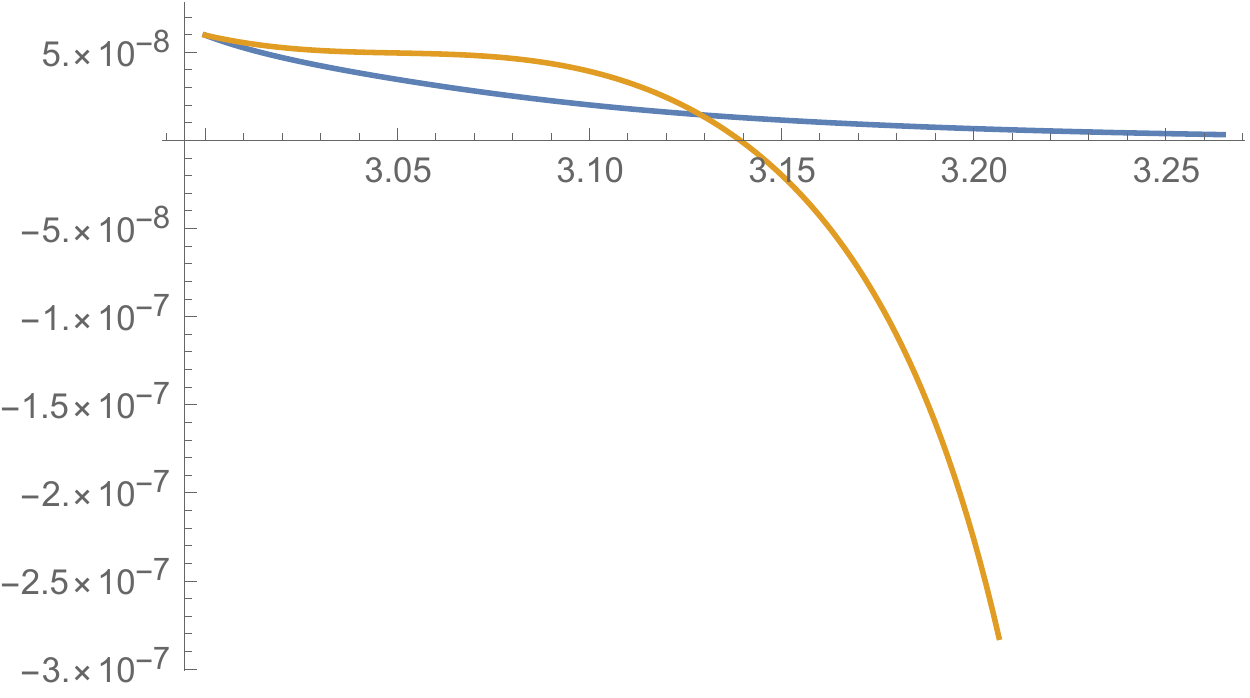}
\end{center}
\caption{Change of sign   of the smallest eigenvalue (in yellow) of the   contributions of $\infty$ and $2$ to the even matrix beyond $\mu=3$. In blue, after adding the contribution of the prime $3$: the total is $>0$\label{testeven4}}
\end{figure}

When $\mu$ goes beyond  the prime power $4=2^2$, the behavior of the smallest eigenvalue is similar to the earlier reported cases and is  shown in Figure \ref{testeven8}. 
\begin{figure}[H]	\begin{center}
\includegraphics[scale=0.6]{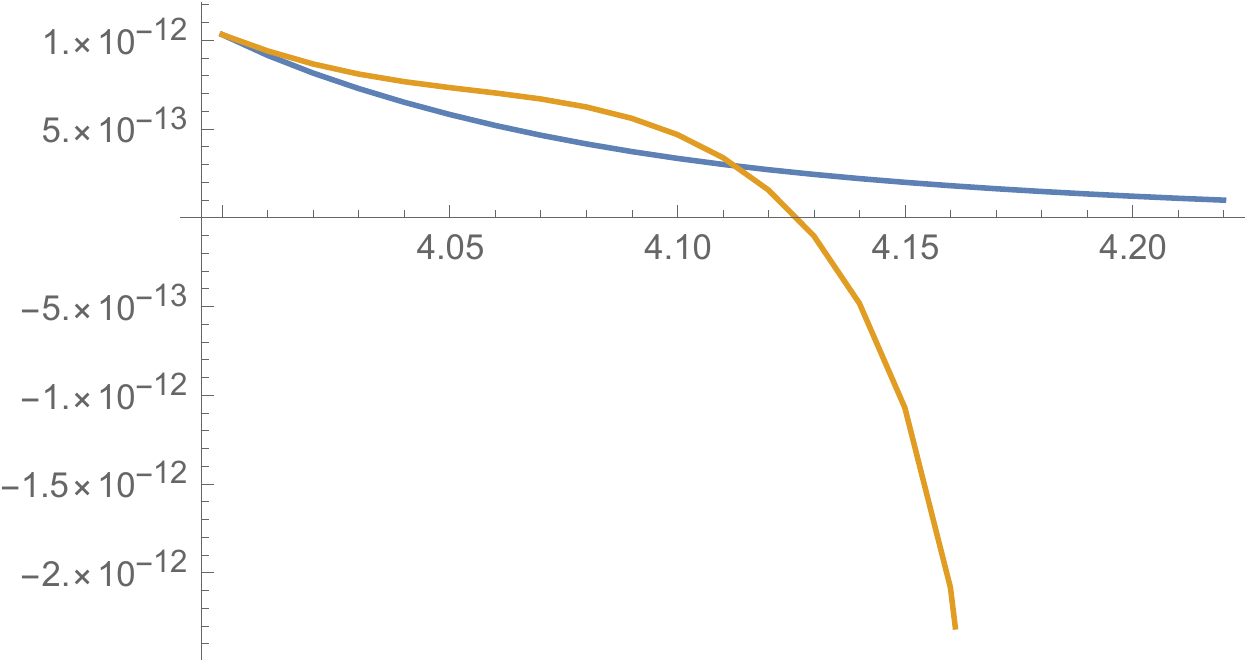}
\end{center}
\caption{Change of sign of the smallest eigenvalue of the even matrix beyond $4$: in yellow if one neglects the contribution of the prime power $4=2^2$, in blue if one does not. The smallest eigenvalue of the total contribution is $>0$\label{testeven8}}
\end{figure}
For $\mu\sim 5$, and $\mu\sim 7$ the behavior of the smallest eigenvalue for the even matrix $\sigma^+$ is similar to those shown in the earlier cases and is reported in Figures \ref{testeven9} and \ref{testevenp}. 
\begin{figure}[H]
\begin{minipage}[b]{0.45\linewidth}
\centering
\includegraphics[width=1.1\linewidth]{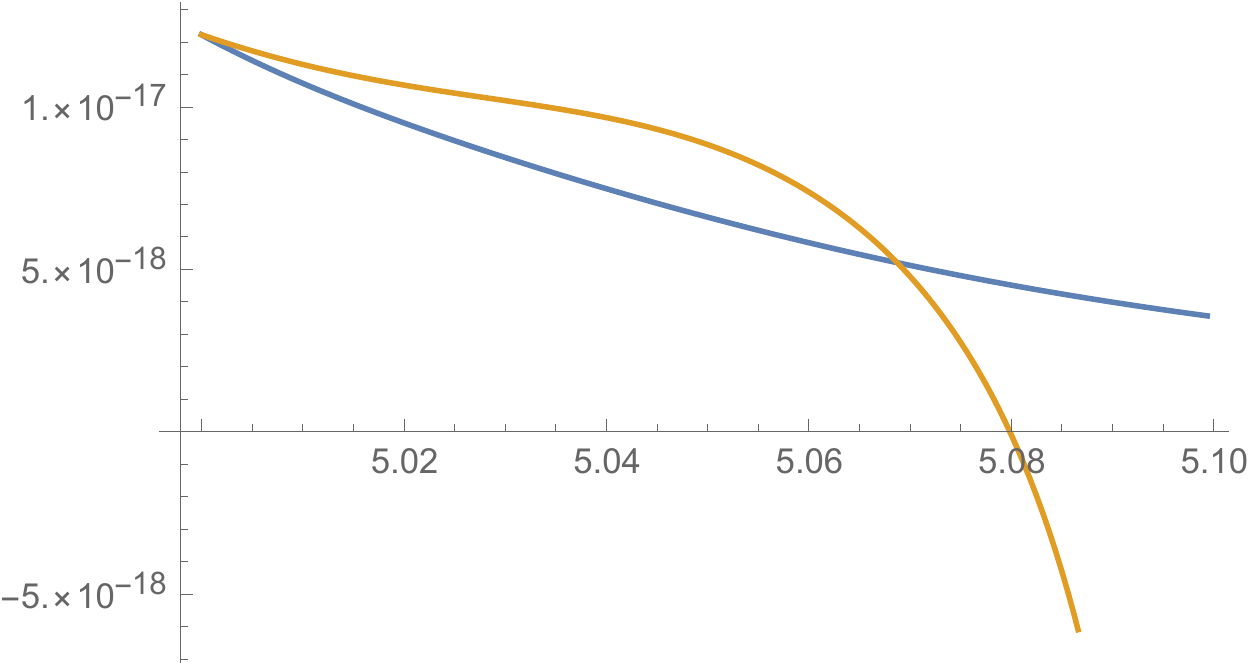}
\caption{Going beyond $\mu=5$ without (yellow) and with (blue) the contribution of the prime $5$\label{zetazeros} }
\label{testeven9}
\end{minipage}
\hspace{0.5cm}
\begin{minipage}[b]{0.45\linewidth}
\centering
\includegraphics[width=1.1\linewidth]{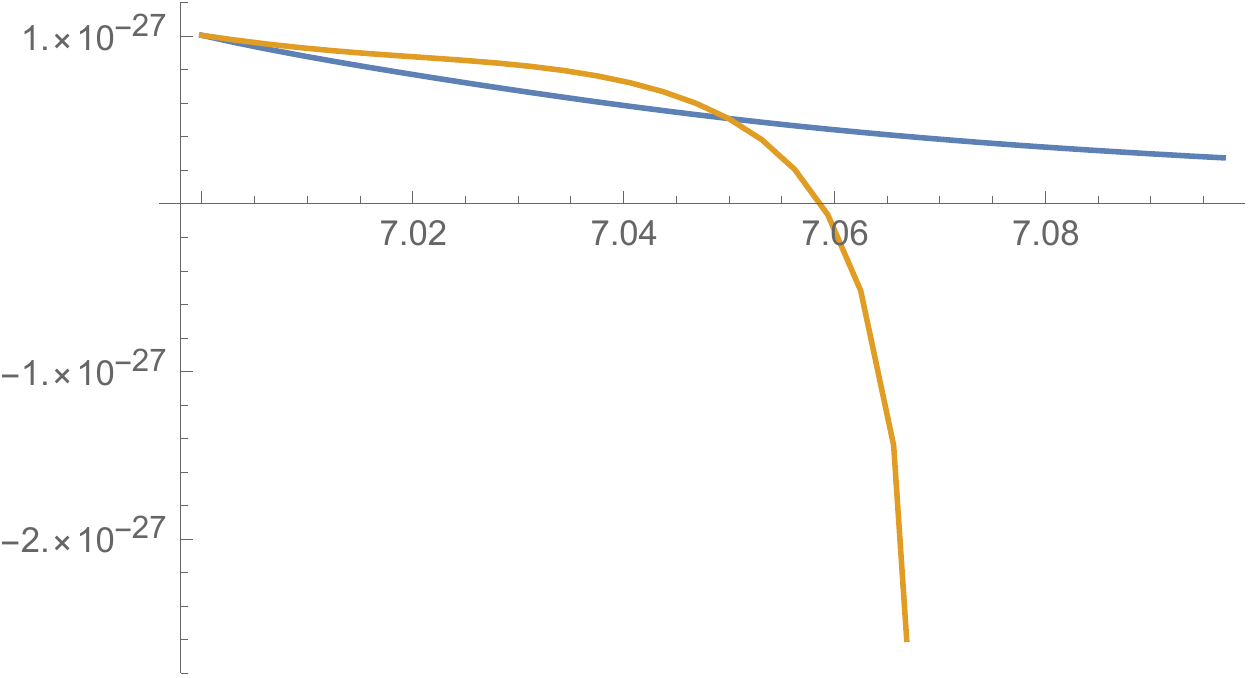}
\caption{Going beyond $\mu=7$ without (yellow) and with (blue) the contribution of the prime $7$
\label{testevenp} }
\end{minipage}
\end{figure}
The following graphs report the change of sign of the smallest eigenvalues for the odd matrices $\sigma^-$, and for the same choices of prime powers: namely near $2$, $3$, $4$, $5$ and $7$.
\begin{figure}[H]
\begin{minipage}[b]{0.45\linewidth}
\centering
\includegraphics[width=1\linewidth]{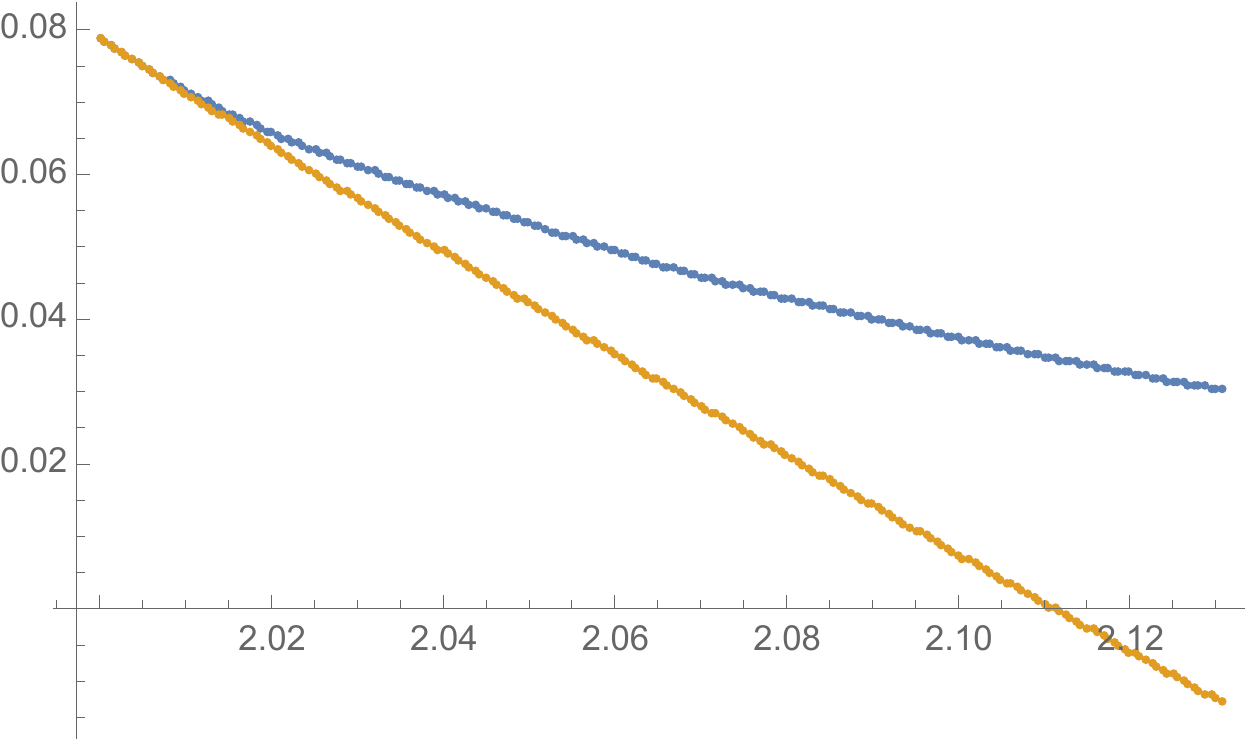}
\caption{Odd case. Going beyond $\mu=2$ without (yellow) and with (blue) the contribution of the prime $2$\label{testodd2} }
\end{minipage}
\hspace{0.5cm}
\begin{minipage}[b]{0.45\linewidth}
\centering
\includegraphics[width=1.1\linewidth]{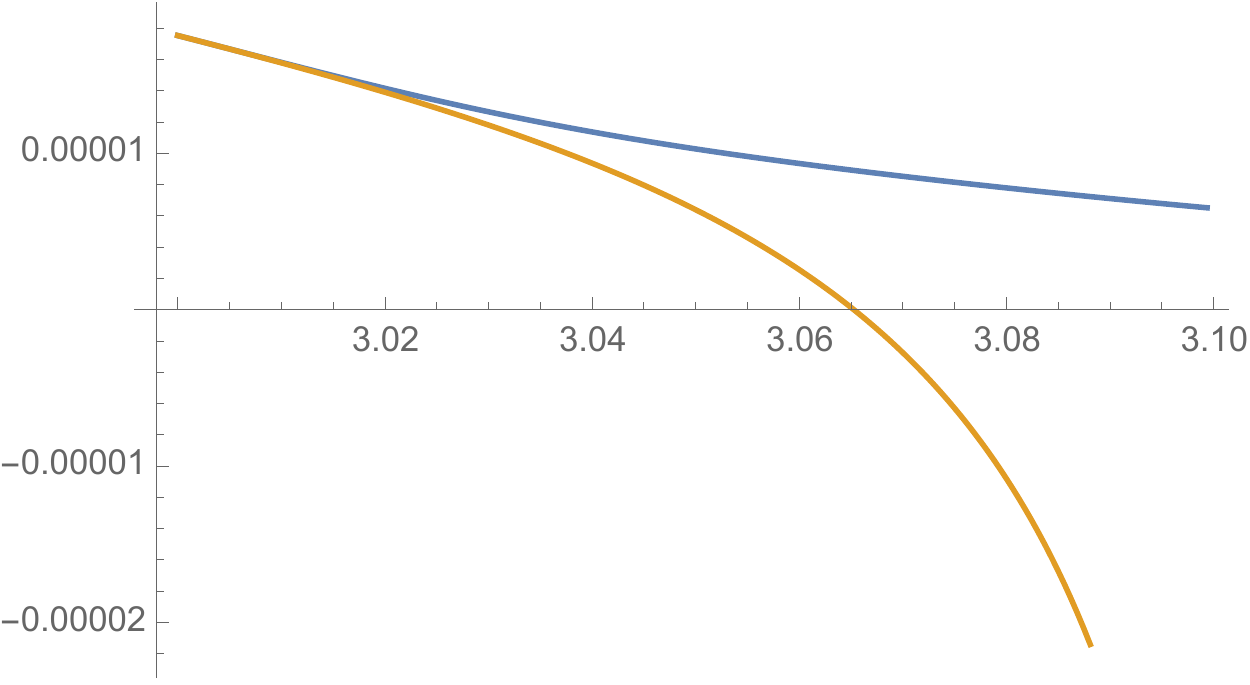}
\caption{Odd case. Going beyond $\mu=3$ without (yellow) and with (blue) the contribution of the prime $3$\label{testodd3} }
\end{minipage}
\end{figure}
\begin{figure}[H]
\begin{minipage}[b]{0.45\linewidth}
\centering
\includegraphics[width=1\linewidth]{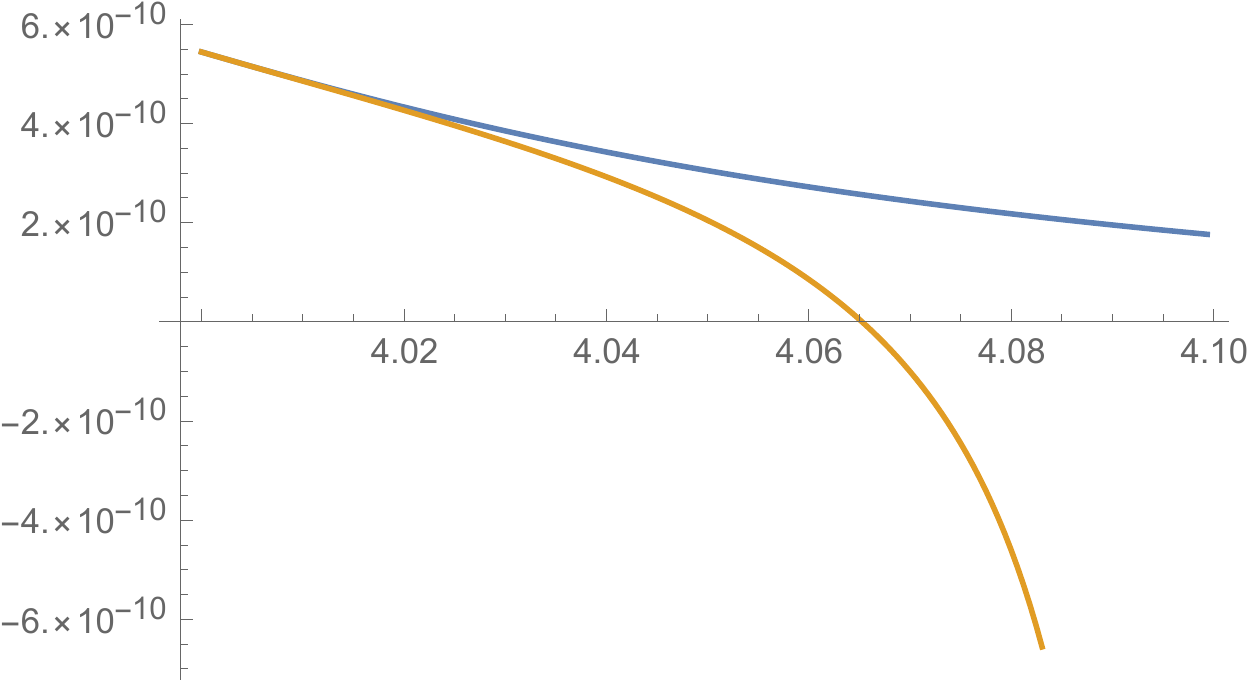}
\caption{Odd case. Going beyond $\mu=4$ without (yellow) and with (blue) the contribution of the prime power $4$\label{testodd4} }
\end{minipage}
\hspace{0.5cm}
\begin{minipage}[b]{0.45\linewidth}
\centering
\includegraphics[width=1.1\linewidth]{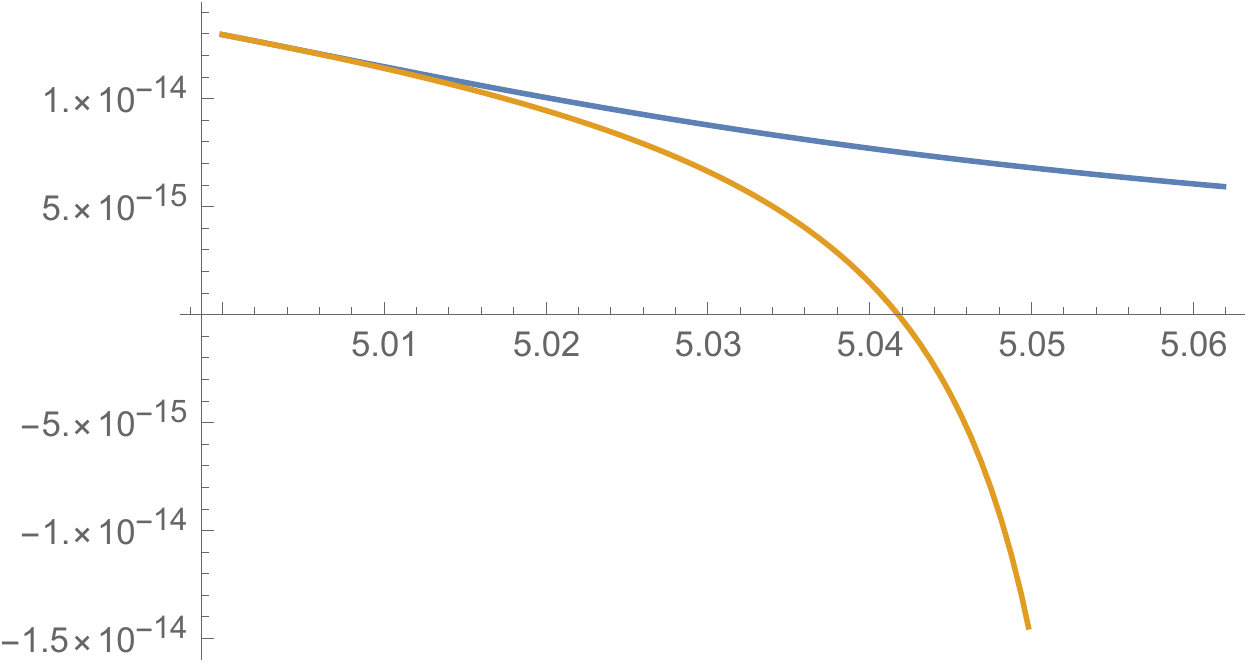}
\caption{Odd case. Going beyond $\mu=5$ without (yellow) and with (blue) the contribution of the prime $5$\label{testodd5} }
\end{minipage}
\end{figure}

\begin{figure}[H]	\begin{center}
\includegraphics[scale=0.5]{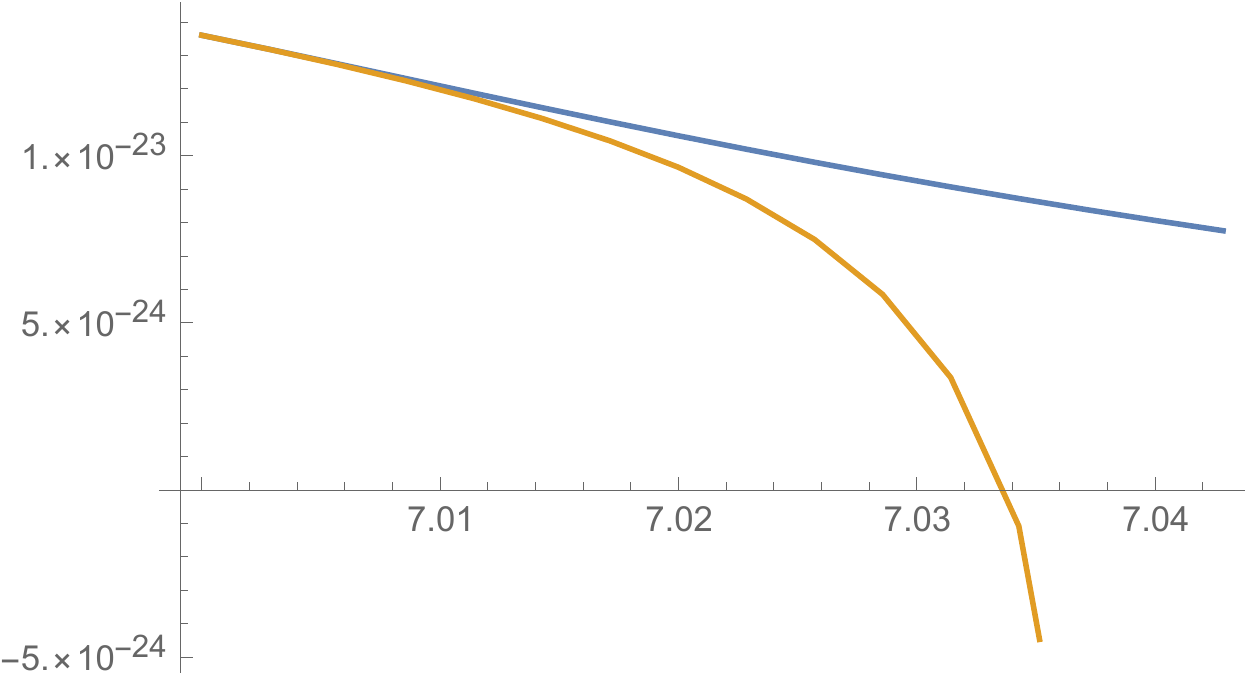}
\end{center}
\caption{Odd case. Going beyond $\mu=7$ without (yellow) and with (blue) the contribution of the prime $7$\label{testodd7}}
\end{figure}
\subsection{Semi-local Weil quadratic form,
small eigenvalues}\label{sectsmall}
Pushing the computations further and increasing the precision, one obtains an estimate of the size  of the smallest eigenvalue $s(L)$ of the even matrix,  as a function of $\mu=\exp L$. One finds an exponential behavior, as reported  in Figures \ref{testeven6} and \ref{testeven7}, where $\log s(L)$ is plotted in terms of $\mu=\exp L$.
\begin{figure}[H]
\begin{minipage}[b]{0.45\linewidth}
\centering
\includegraphics[width=1\linewidth]{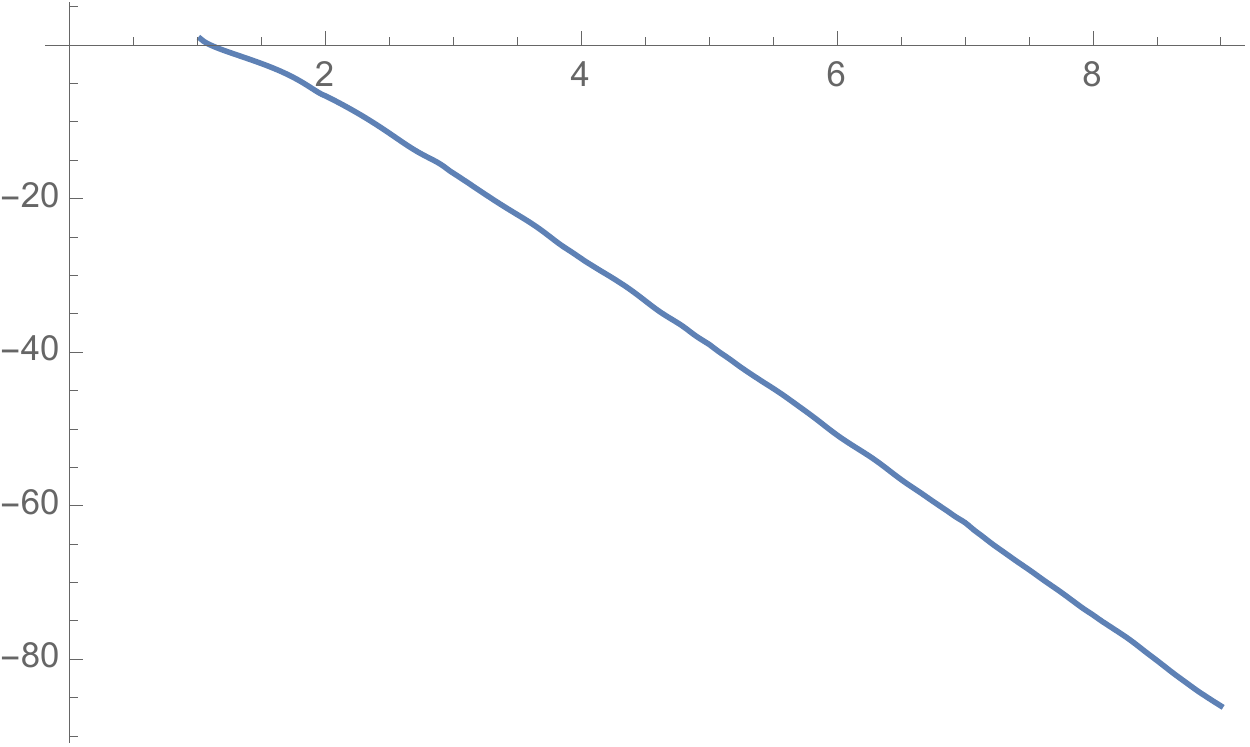}
\caption{Decay of the log of the smallest eigenvalue of the even matrix as a function of $\mu=\exp L$\label{testeven6} }
\end{minipage}
\hspace{0.5cm}
\begin{minipage}[b]{0.45\linewidth}
\centering
\includegraphics[width=1.1\linewidth]{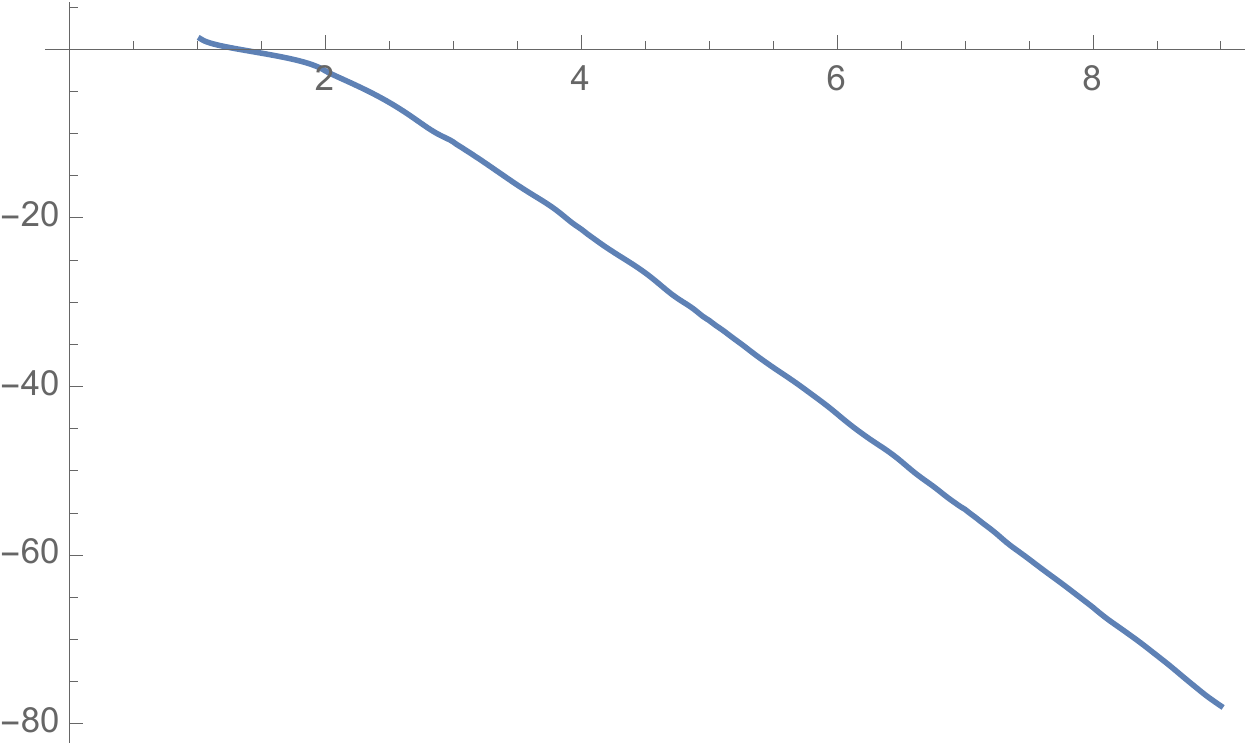}
\caption{Decay of the log of the smallest eigenvalue of the odd matrix as a function of $\exp L$\label{testeven7} }
\end{minipage}
\end{figure}

When one selects the small eigenvalues of the even matrix $\sigma^+$ and plots the graphs of the logarithm of their size, one finds (see Figure \ref{smallsmall}) that their number increases roughly like $\mu=\exp L$.
\begin{figure}[H]	\begin{center}
\includegraphics[scale=0.7]{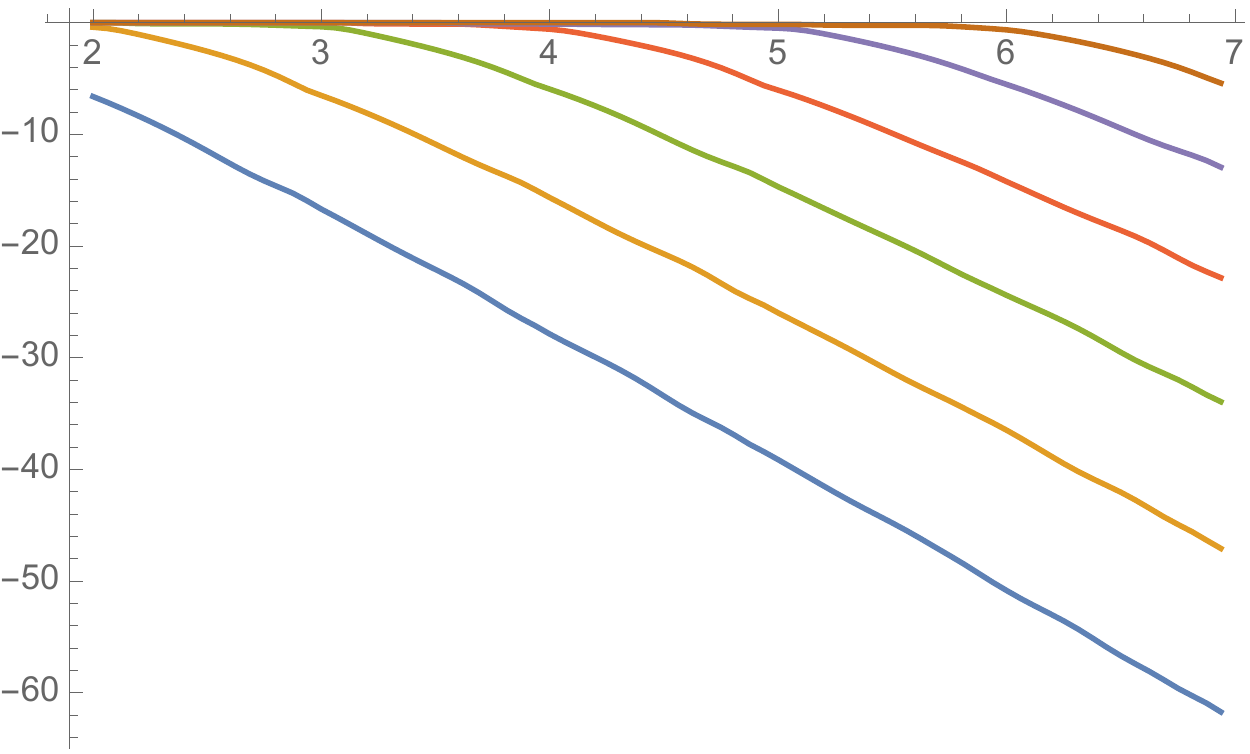}
\end{center}
\caption{Decay of the log of the smallest eigenvalues of the even matrix $\sigma^+$ as a function of $\mu=\exp L$\label{smallsmall}}
\end{figure}
For the odd matrix  $\sigma^-$, the behavior is similar  but with one less small eigenvalue, as shown in Figure \ref{smallsmall1}.
\begin{figure}[H]	\begin{center}
\includegraphics[scale=0.6]{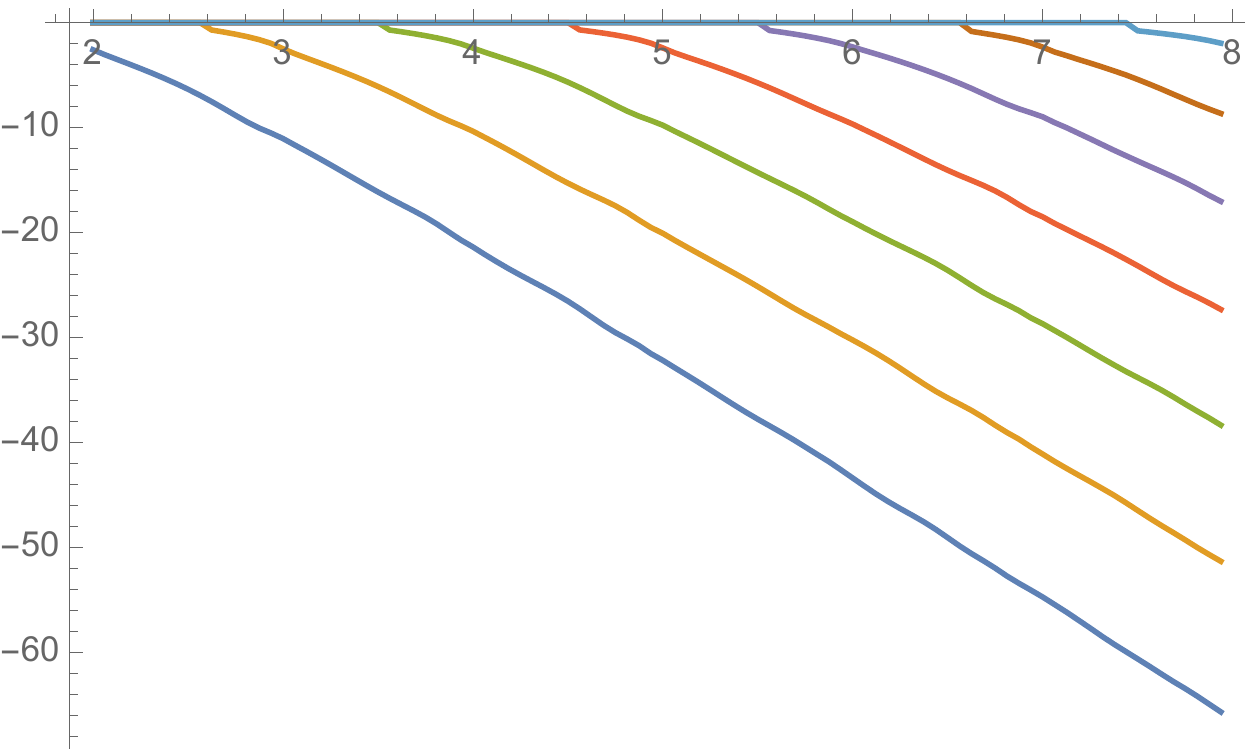}
\end{center}
\caption{Decay of the log of the smallest eigenvalues of the odd matrix $\sigma^-$ as a function of $\mu=\exp L$\label{smallsmall1}}
\end{figure}
%\newpage

\begin{figure}[H]	\begin{center}
\includegraphics[scale=0.45]{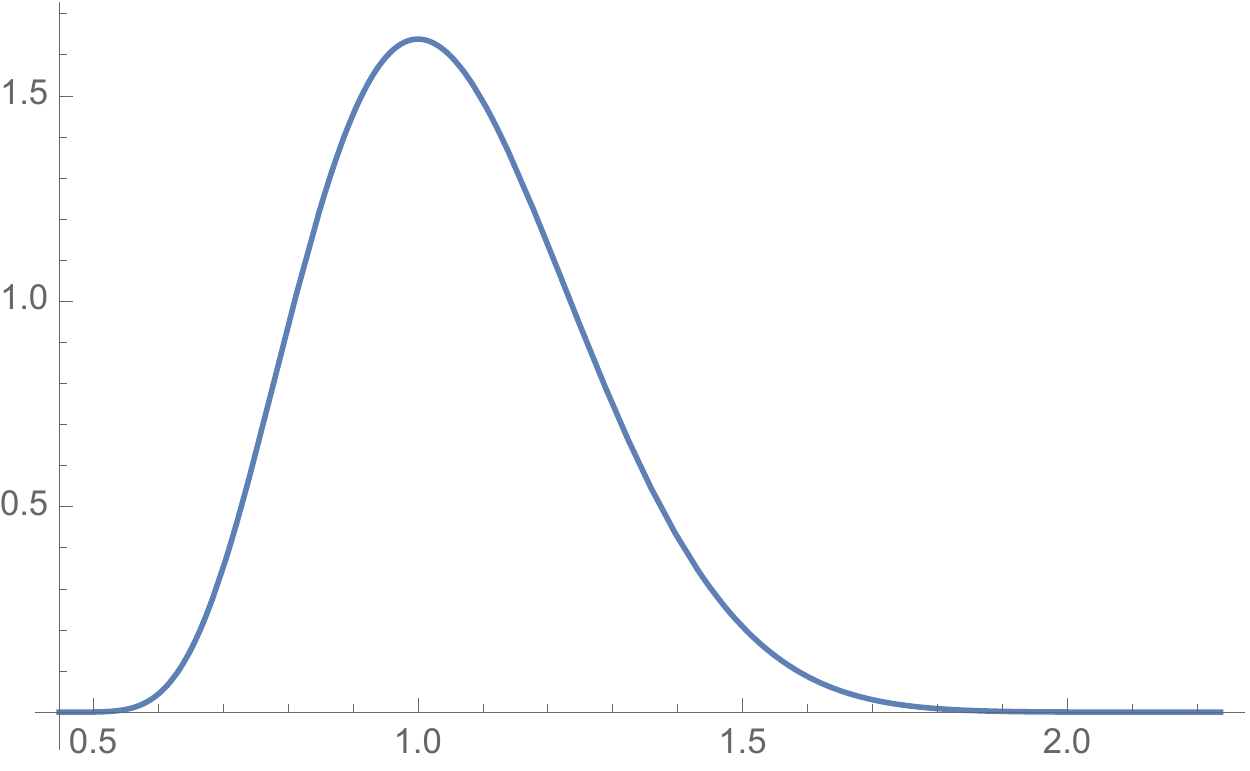}
\end{center}
\caption{Eigenvector for the smallest eigenvalue of $QW^+_\lambda$ as function on $\R_+^*$\label{eigen1}}
\end{figure}
 Figures \ref{eigen1}, \ref{eigen2} and \ref{eigen3}  report  the graphs of the eigenvectors of the quadratic form  $QW^+_\lambda$  respectively for the smallest, the second smallest and the third smallest eigenvalues.
\begin{figure}[H]
\begin{minipage}[b]{0.45\linewidth}
\centering
\includegraphics[width=1\linewidth]{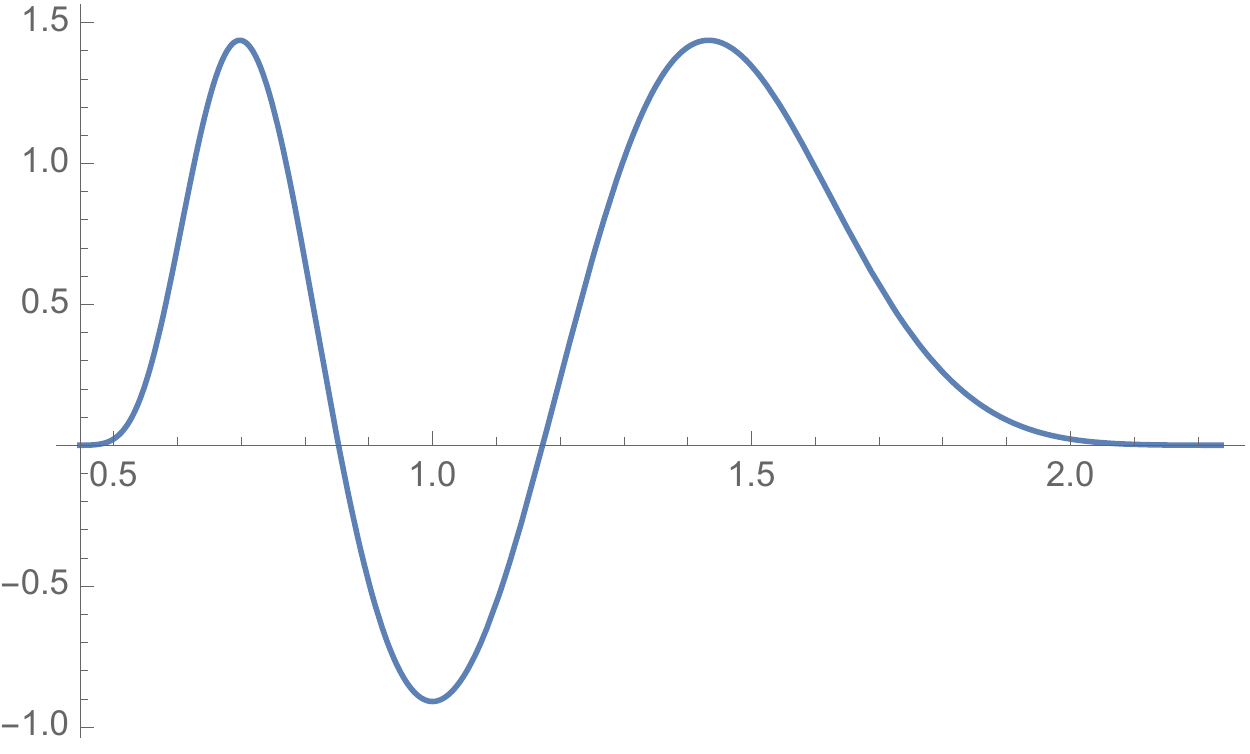}
\caption{Eigenvector for the second smallest eigenvalue of $QW^+_\lambda$ as function on $\R_+^*$\label{eigen2} }
\end{minipage}
\hspace{0.5cm}
\begin{minipage}[b]{0.45\linewidth}
\centering
\includegraphics[width=1.1\linewidth]{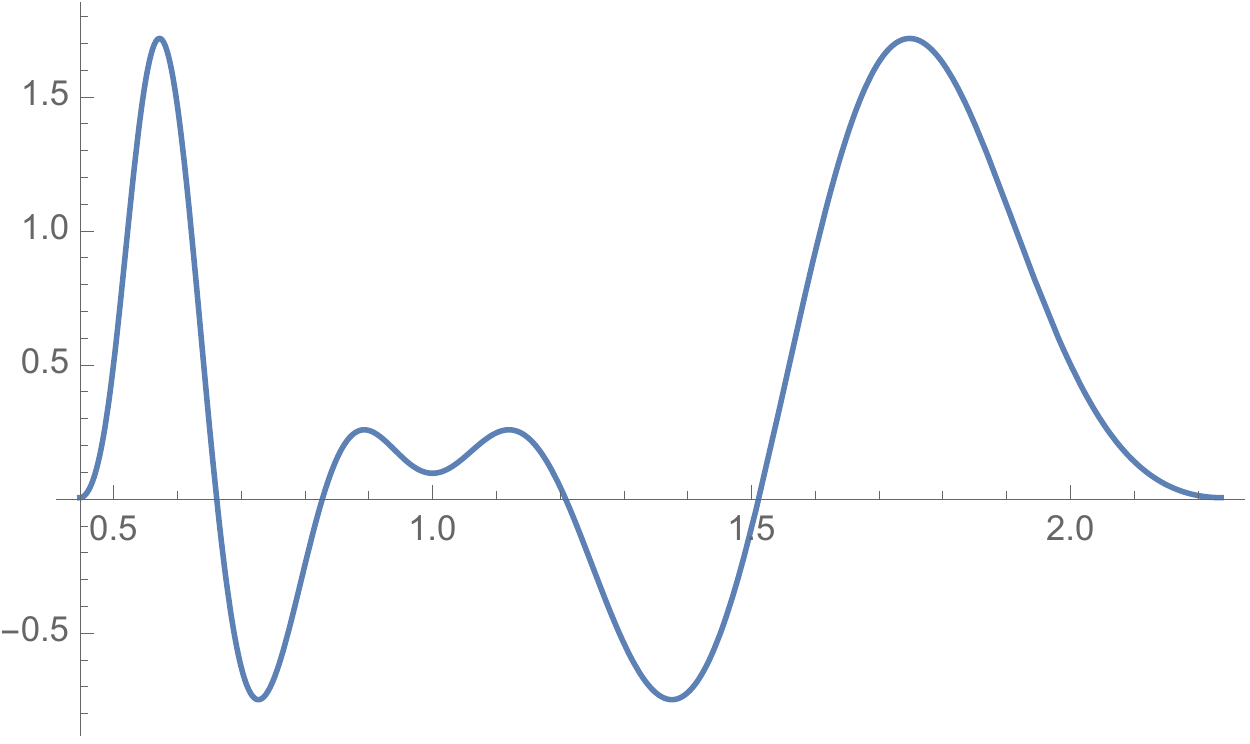}
\caption{Eigenvector for the third smallest eigenvalue of $QW^+_\lambda$ as function on $\R_+^*$\label{eigen3}}
\end{minipage}
\end{figure}

\section{Eigenfunctions and the prolate projection $\Pi(\lambda,k)$}\label{riemweilexpl}
In this section we explain the existence of the very small eigenvalues of the Weil quadratic form $QW_\lambda$ on test functions with support in an interval $[\lambda^{-1},\lambda]\subset \R_+^*$. We start by recalling that if RH holds, then the Weil quadratic form restricted to functions with support in a  finite interval has zero  radical, since the number $N(r)$ of zeros  of modulus  at most $r$   of the Fourier transform of a function $f$ with compact support is of the order $
N(r) =O(r)$
 (see \cite{Rudin} \S 15.20 (2)), while if $f$  belonged to the radical of $QW$ it would (assuming RH), vanish on all zeros of the Riemann zeta function whose number grows faster  than $O(r)$. On the other hand,  the radical of  $QW$ contains the range of the  map $\cE$ defined  on the codimension two subspace $\sr0\subset\cS(\R)$ of even Schwartz functions fulfilling $f(0)=\widehat f(0)=0$ by the formula (\cite{Co-zeta})
\begin{equation}
\cE(f)(x)=x^{1/2}\sum_{n>0}f(nx).
\label{mapee}
\end{equation}
It is thus natural to bring-in \eqref{mapee} for the construction of functions $g$ with  support in $[\lambda^{-1},\lambda]\subset \R_+^*$ which belong to the ``near radical'' of $QW_\lambda$ \ie fulfill $QW_\lambda(g)\ll \Vert g\Vert^2$. The definition of $\cE$  shows that if the support of the even function $f\in \sr0$ is contained in the interval $[-\lambda,\lambda]\subset \R$, then the support of $\cE(f)$ is contained in $(0,\lambda]\subset \R_+^*$. On the other hand, by applying the  Poisson formula (with $\widehat f$ the Fourier transform of $f$) one  has
\begin{equation}
\cE(\widehat f)(x)=\cE(f)(x^{-1}) \qqq f\in \sr0.\label{poisson}
\end{equation}
Thus  we see that $\lambda^{-1}$ would be a lower bound of the support of $\cE(f)$ if the support of the even function $\widehat f\in \sr0$ were contained in the interval $[-\lambda,\lambda]\subset \R$. However this latter inclusion is impossible  since  the Fourier transform of a function with compact support is analytic. In spite of this apparent obstacle in the construction, the work of Slepian and Pollack on band limited functions  \cite{Slepian0} provides a very useful approximate solution. The  conceptual way to formulate their result is in terms of   the pair of projections $\cP_\lambda$ and $\widehat{\cP_\lambda}$ in the Hilbert space $L^2(\R)^{\rm ev}$ of square integrable even functions. The operator $\cP_\lambda$ is  the   multiplication by the characteristic function of the interval  $[-\lambda,\lambda]\subset \R$, and the projection $\widehat{\cP_\lambda}$ is its conjugate by the (additive) Fourier transform $\fourierer$. These two projections   have zero intersection but their ``angle'', -- an operator with discrete spectrum-- admits approximatively $2 \lambda^2$  very small eigenvalues whose  associated eigenfunctions provide excellent candidates for the ``approximate intersection''  $\cP_\lambda \cap'\widehat{\cP_\lambda}$. In their work on signals transmission, Slepian and Pollack  discovered that these eigenfunctions are exactly the prolate spheroidal wave functions which were already known to be  solutions (by separation of variables) of the Helmoltz equation for prolate spheroids.

The basic result of Slepian and Pollack is the diagonalisation of the positive operator $\cP_\lambda \widehat{\cP_\lambda}\cP_\lambda$ in the Hilbert space $L^2([-\lambda,\lambda])$. They show  that this operator  commutes with the differential operator 
\begin{equation}
({\bf W}_{\lambda}\psi)(q) = \,- \partial (( \lambda^2- q^2) \partial )\,
\psi(q) + (2 \pi \lambda  q)^2 \, \psi(q)  \label{WLambdaq}
\end{equation}
(here $\partial$ is the ordinary differentiation in one
variable $q\in [-\lambda,\lambda]$ and the dense domain is that of smooth functions on $[-\lambda,\lambda]$). The operator ${\bf W}_{\lambda}$ (obtained by closing its domain in the graph norm) is selfadjoint and positive and its eigenfunctions are the prolate spheroidal wave functions. When considering  the Weil quadratic form $QW_\lambda$ evaluated on test functions with support in $[\lambda^{-1},\lambda]$ we shall compare the eigenvectors associated to the extremely small eigenvalues with the range of the map $\cE$ applied to linear combinations of  eigenfunctions $\psi$ of ${\bf W}_{\lambda}$   which belong to the approximate  intersection $\cP_\lambda \cap'\widehat{\cP_\lambda}$ and  vanish at zero.  For this process, we only take the eigenfunctions $\psi$   which are even functions of the variable $q$, and  distinguish two cases since the action of the Fourier transform  fulfills $\fourier_{e_\R}\psi\simeq \pm \psi$ on eigenfunctions $\psi$ of ${\bf W}_{\lambda}$. The corresponding  sign $\pm$  determines precisely the choice of  an eigenvector for the even or odd matrix. In standard notation one sets
 $$
\psi_{m,\lambda}(x):=\text{\textit{PS}}_{2m,0}\left(2 \pi  \lambda^2,\frac{x}{\lambda}\right)
$$
where $\psi_{m,\lambda}$ is a function on the interval $[-\lambda,\lambda]$ that one extends by $0$ outside that interval. Its Fourier transform $\fourier_{e_\R}(\psi_{m,\lambda})$ is equal to $\chi_m\psi_{m,\lambda}$ on  $[-\lambda,\lambda]$, where the scalar $\chi_m$ is very close to $(-1)^m$ provided that $m$ is less than $2 \lambda^2$. More precisely,  $\fourier_{e_\R}(\psi_{m,\lambda})$  is computed using  the equality
 $$
 \int_{-1}^1 \text{\textit{PS}}_{2m,0}(\gamma ,\eta ) \exp (i \gamma  \eta  \omega ) \, d\eta=(-1)^m 2 S_{2m,0}^{(1)}(\gamma ,1) \text{\textit{PS}}_{2m,0}(\gamma ,\omega )
 $$
 for $\gamma=2\pi \lambda^2$, $\omega=\frac{y}{\lambda}$. After changing variables to $\xi=\lambda\,\eta$, the equality above becomes
 $$
 \int_{-\lambda}^{\lambda}\psi_{m,\lambda}(\xi) \exp (i 2\pi \xi y ) \, d\xi=(-1)^m2 \lambda S_{2m,0}^{(1)}(2 \pi \lambda^2 ,1) \psi_{m,\lambda}(y).
 $$
 Given $\mu=\lambda^2$, one only retains the values of $m$ for which the characteristic value 
 $
 \chi(\mu,m)=2\lambda  S_{2m,0}^{(1)}(2 \pi \mu ,1)
 $
 is almost equal to $1$. This determines a collection $\{0, \ldots, \nu(\mu)\}$ of length  approximately equal to $2\mu$, such that  $\chi(\mu,m)\sim 1$ for $m\leq \nu(\mu)$. The formula $\nu(\mu)=2\mu-1$ works well when $\mu$ is a small half integer.

 \begin{figure}[H]	\begin{center}
\includegraphics[scale=0.5]{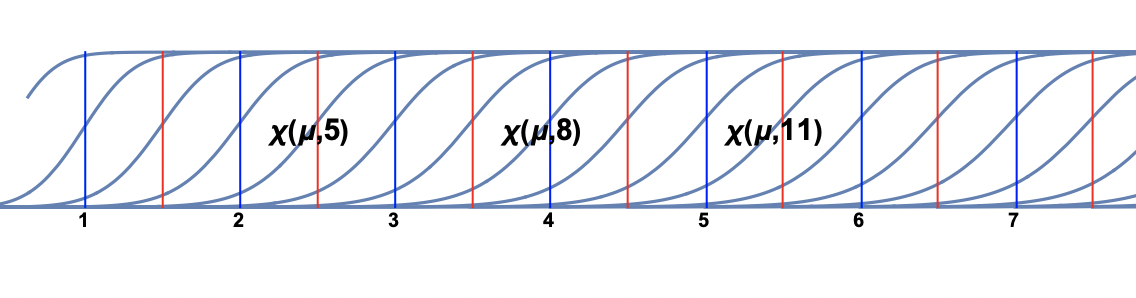}
\end{center}
\caption{Graphs of the functions $\chi(\mu,m)$ as functions of $\mu$ \label{chimum}}
 \end{figure}
 In order to define the prolate projection we consider   linear combinations of prolate functions 
 which vanish at $0$ and are given, for $n>0$, by  
 $$
 \phi_{2n}(x):=\psi_{2n}(x)\psi_{0}(0)-\psi_{0}(x)\psi_{2n}(0), \ \ \phi_{2n+1}(x):=\psi_{2n+1}(x)\psi_{1}(0)-\psi_{1}(x)\psi_{2n+1}(0).
 $$ 
  For $1< n\leq \nu(\mu)$ one may approximate $\fourierer(\phi_n)$ by $(-1)^n\phi_n$ and, using the Poisson formula, act as if $\cE(\phi_n)$ would fulfill  the equality 
 $
 \cE(\phi_n)(u^{-1})=(-1)^n\cE(\phi_n)(u)
 $.
 We can then compute the components of $\cE(\phi_n)$  in the orthogonal basis $\eta_j(u)=\xi_j(\log u)$  of $ \cH=L^2([\lambda^{-1},\lambda],d^*x)$ (Lemma \ref{polarize2}) which fulfill  $\eta_j(u^{-1})=- \eta_j(u)$ for $j< 0$ and $\eta_j(u^{-1})= \eta_j(u)$ for $j\geq 0$. For  $1<n\leq \nu(\mu)$, the component of $\cE(\phi_n)$ on  $\eta_j$ is non-zero only if $\eta_j$ has the same parity as $n$, \ie fulfills $\eta_j(u^{-1})=(-1)^n \eta_j(u)$,   and in this case is given by the  formula, 
 \begin{equation}\label{ephirough}
 \cE(\phi_n)_j\simeq 2 \sum_{1\leq r<\lambda} \int_1^{\lambda/r}\ u^{1/2} \phi_n(ru)\eta_j(u)d^*u.
  \end{equation}
   One computes all these components for $\vert j\vert \leq N$ with $N$ large, and  applies the Gram-Schmidt orthogonalization process  (separately for the even and odd cases) to the obtained vectors in $E_N$. This process determines orthonormal vectors  $\epsilon_n\in E_N\subset \cH=L^2([\lambda^{-1},\lambda],d^*x)$ for $1<n\leq \nu(\mu)$ which are, by construction, the natural candidate functions to be compared (up to sign) with the eigenfunctions of the semi-local Weil quadratic form $QW_\lambda$ on $E_N$.       \begin{definition}\label{projpi} Let $k< \nu(\lambda^2)$. We define $\Pi(\lambda,k)$ as the orthogonal projection on the linear span of the vectors $\epsilon_n$, for $n\in \{2,\ldots, k+1\}$.  	
  \end{definition}
Let $\gamma$ be the grading operator in $\cH=L^2([\lambda^{-1},\lambda],d^*x)$ which takes the values $\pm 1$ on functions satisfying the equality $f(u^{-1})=\pm f(u)$. By construction, the vectors $\epsilon_n$ are eigenvectors of $\gamma$ and  the following commutativity holds
 \begin{equation}\label{commutingPi}
 \gamma\ \Pi(\lambda,k) =\Pi(\lambda,k)\ \gamma \qqq \lambda, k
  \end{equation}
   The series of graphs reported here below display the coincidence of the $\epsilon_n$ with the actual eigenfunctions of the semi-local Weil quadratic form  $QW_\lambda$  for the smallest eigenvalues. When only one graph appears (in yellow the graph of the $\epsilon$'s) this means that the graphs of the two functions match with a  high precision, otherwise the  graph in blue  of the eigenfunction is no longer hidden behind the yellow graph. Notice that the coincidence of $\epsilon_{2m}$ with the eigenfunction of the even matrix for its $m$-th eigenvalue is expected to hold only when this eigenvalue is small and hence only when $\mu>m$. Similarly,  one expects the coincidence of $\epsilon_{2m+1}$  with the eigenfunction of the odd matrix for its $m$-th eigenvalue only when $\mu>m+1$ (since the number of small eigenvalues of the odd matrix is one less than for the even one). We have nevertheless plotted the graphs for all half integer values of $\mu$ between $3.5$ and $11$ to show the mismatch of the graphs when $\mu$ is too small.
\begin{figure}[H]	\begin{center}
\includegraphics[scale=0.5]{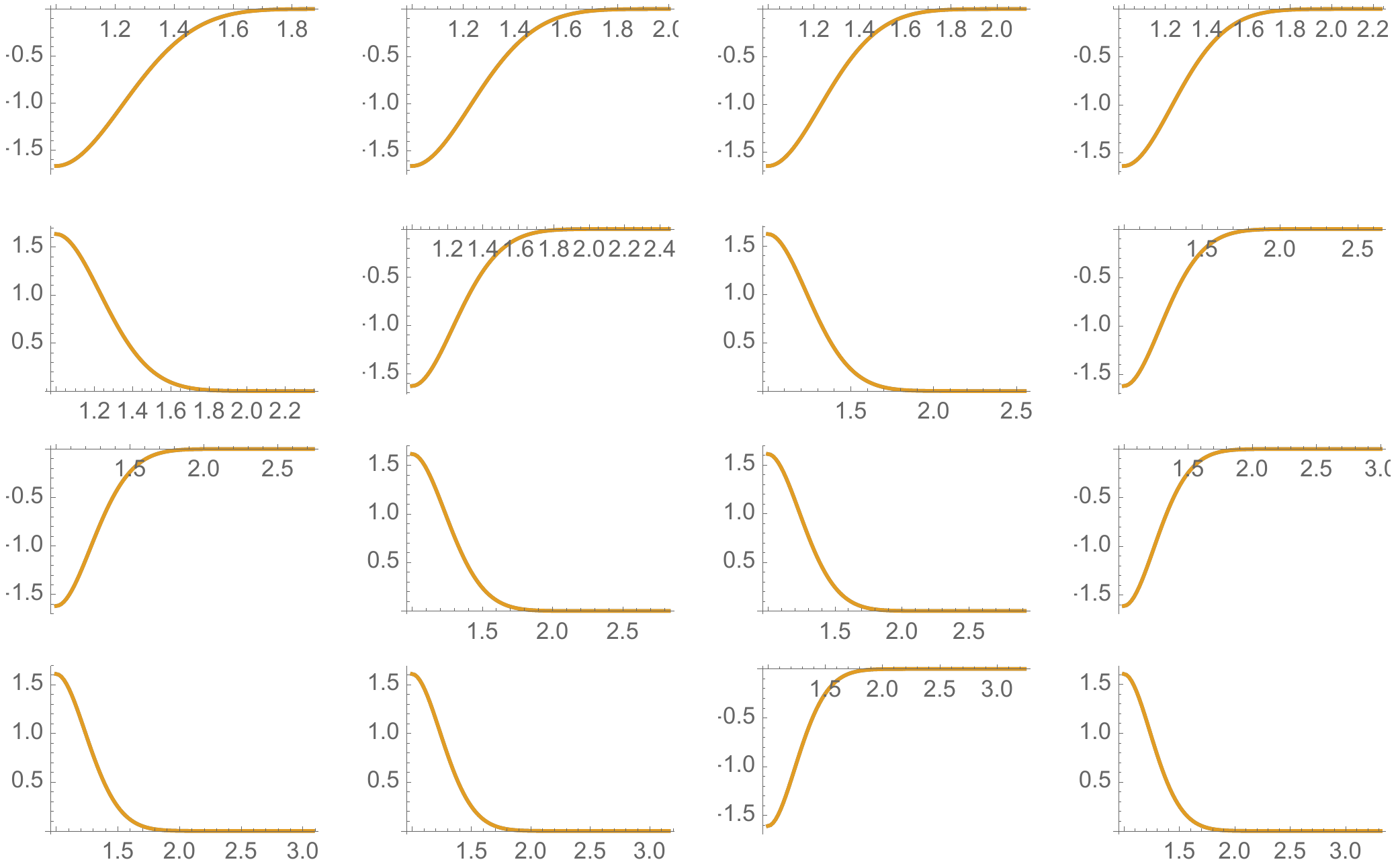}
\end{center}
\caption{agreement of eigenfunctions for the even matrix and the smallest eigenvalue,
for the $16$ values of $\mu$ between $3.5$ and $11$. For $\mu=11$  the eigenvalue is $2.389\times 10^{-48}$ \label{eigenev1}}
\end{figure}
\begin{figure}[H]	\begin{center}
\includegraphics[scale=0.6]{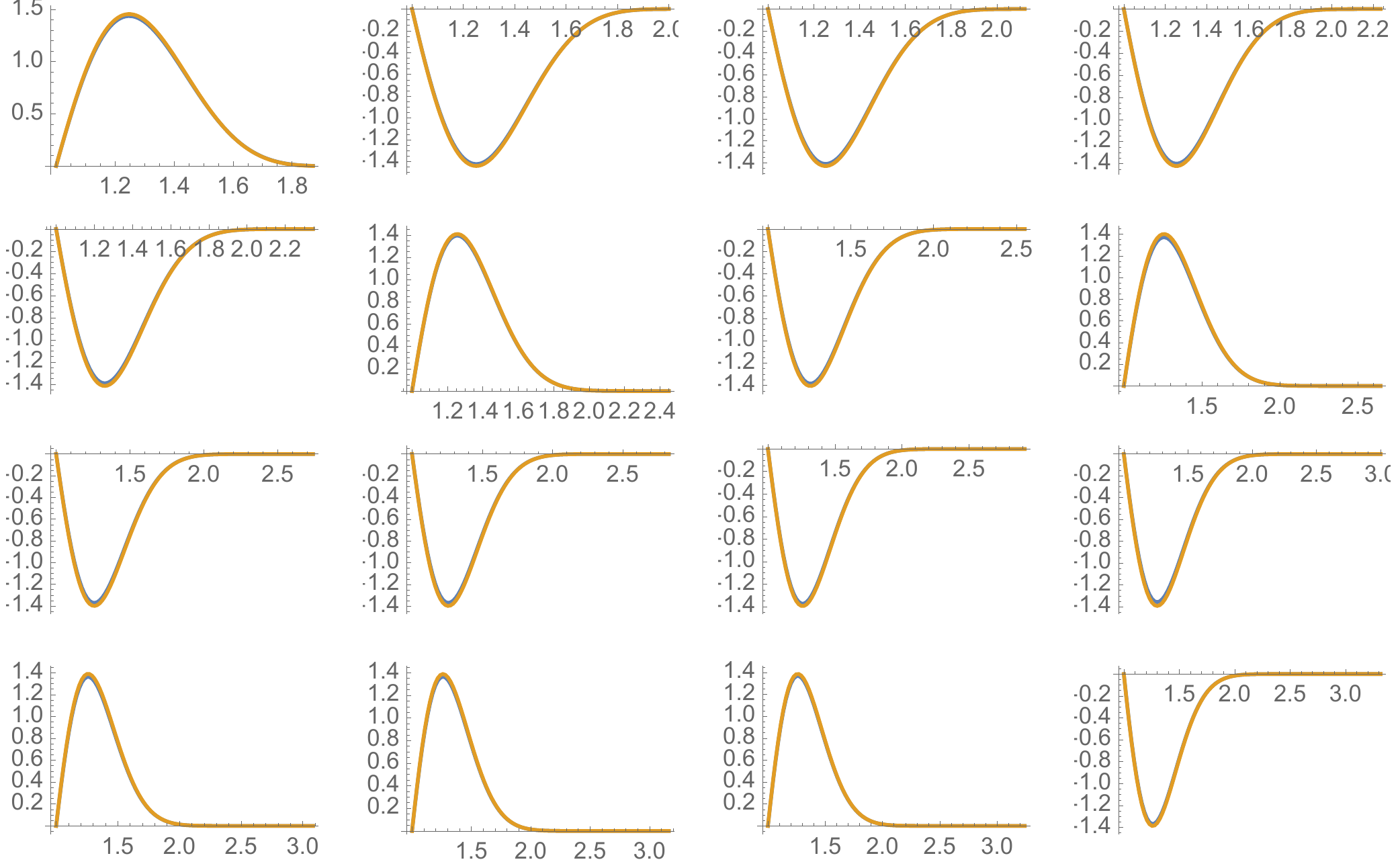}
\end{center}
\caption{agreement of eigenfunctions for the odd matrix and the smallest eigenvalue
for the $16$ values of $\mu$ between $3.5$ and $11$\label{eigenodd1}}
\end{figure}
%\vspace*{-1in}

\begin{figure}[H]	\begin{center}
\includegraphics[scale=0.6]{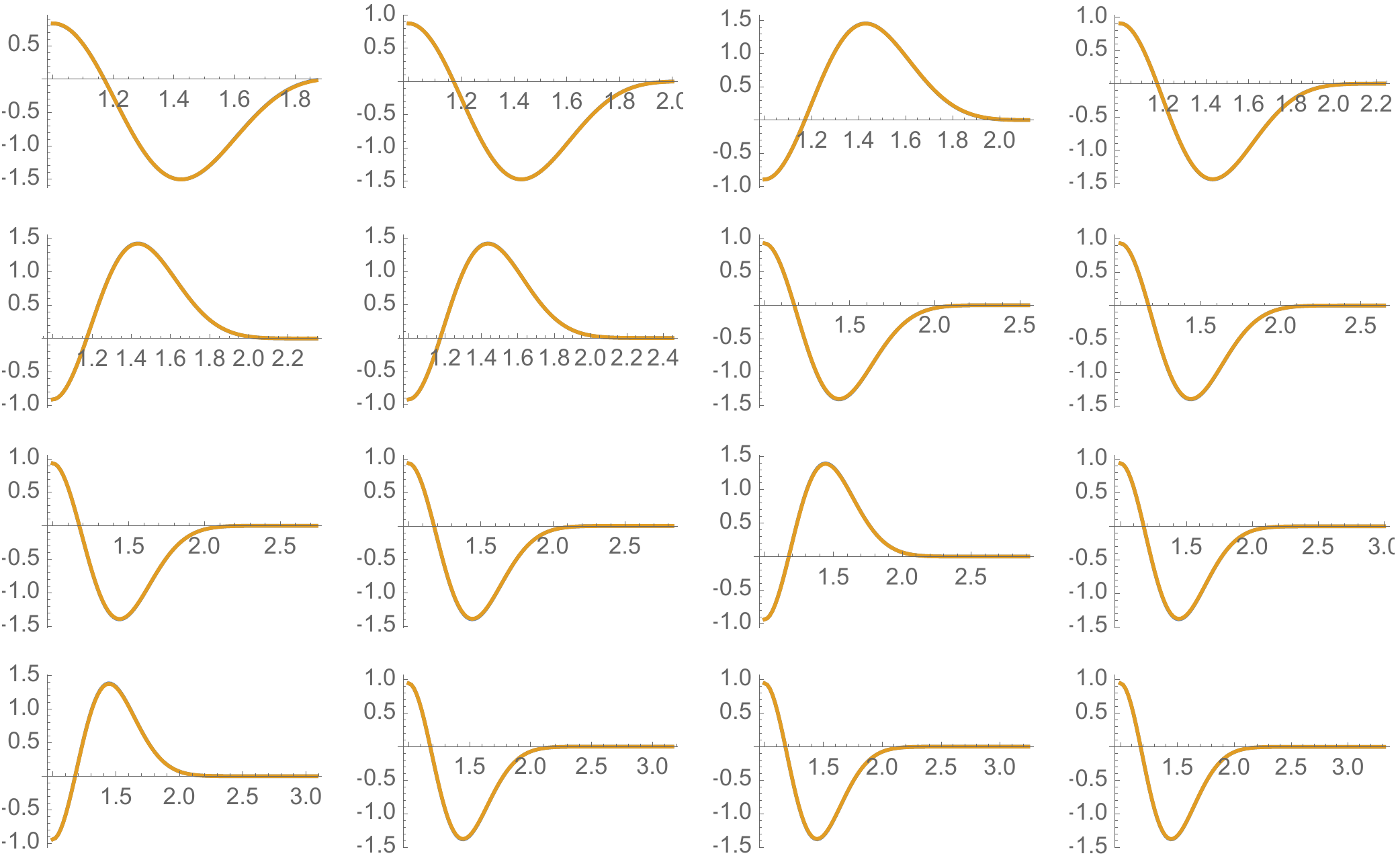}
\end{center}
\caption{agreement of eigenfunctions for the even matrix and the second smallest eigenvalue for the $16$ values of $\mu$ between $3.5$ and $11$\label{eigenev3}}
\end{figure}
\begin{figure}[H]	\begin{center}
\includegraphics[scale=0.6]{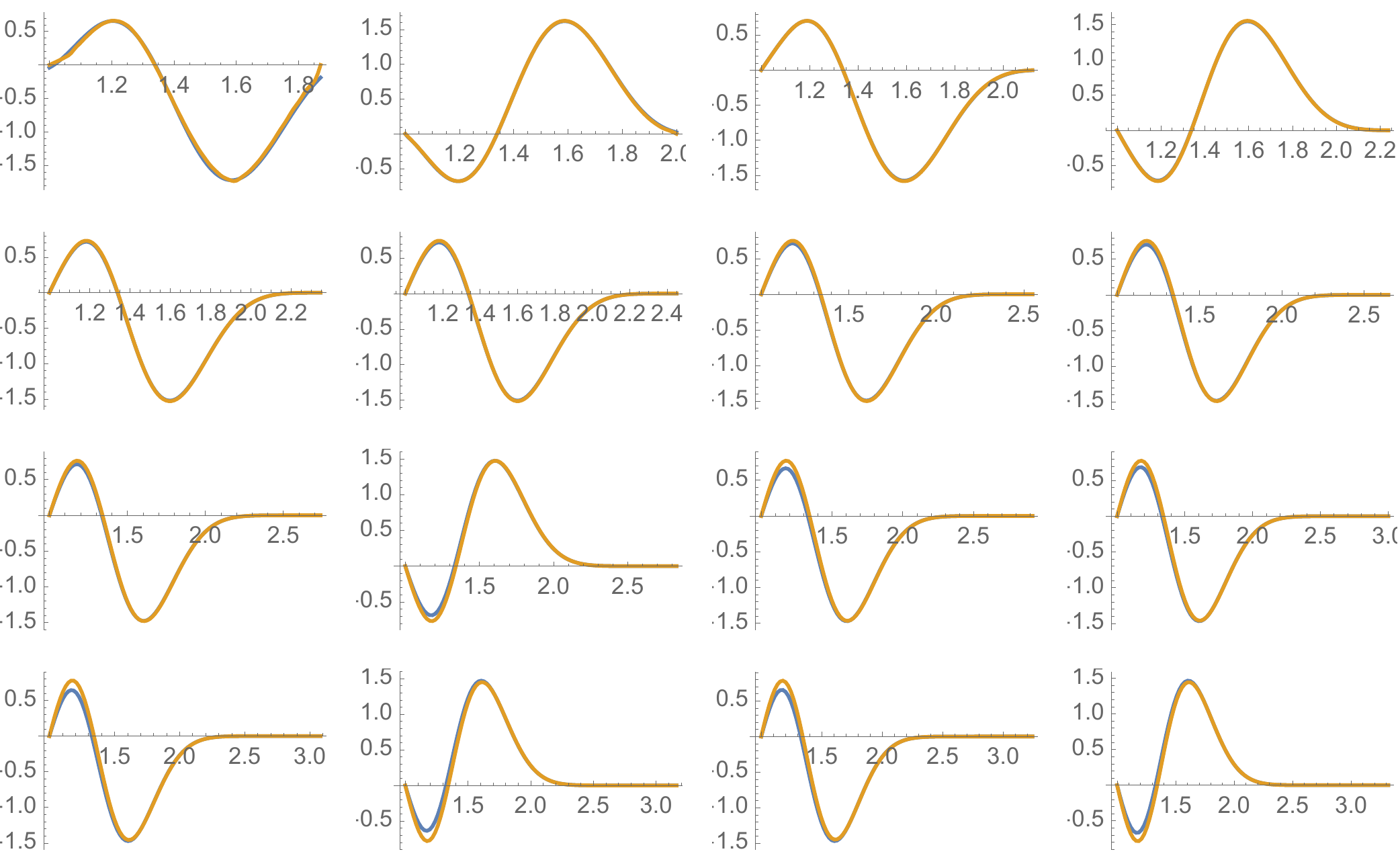}
\end{center}
\caption{agreement of eigenfunctions for the odd matrix and the second smallest eigenvalue for the $16$ values of $\mu$ between $3.5$ and $11$\label{eigenodd2}}
\end{figure}
\begin{figure}[H]	\begin{center}
\includegraphics[scale=0.6]{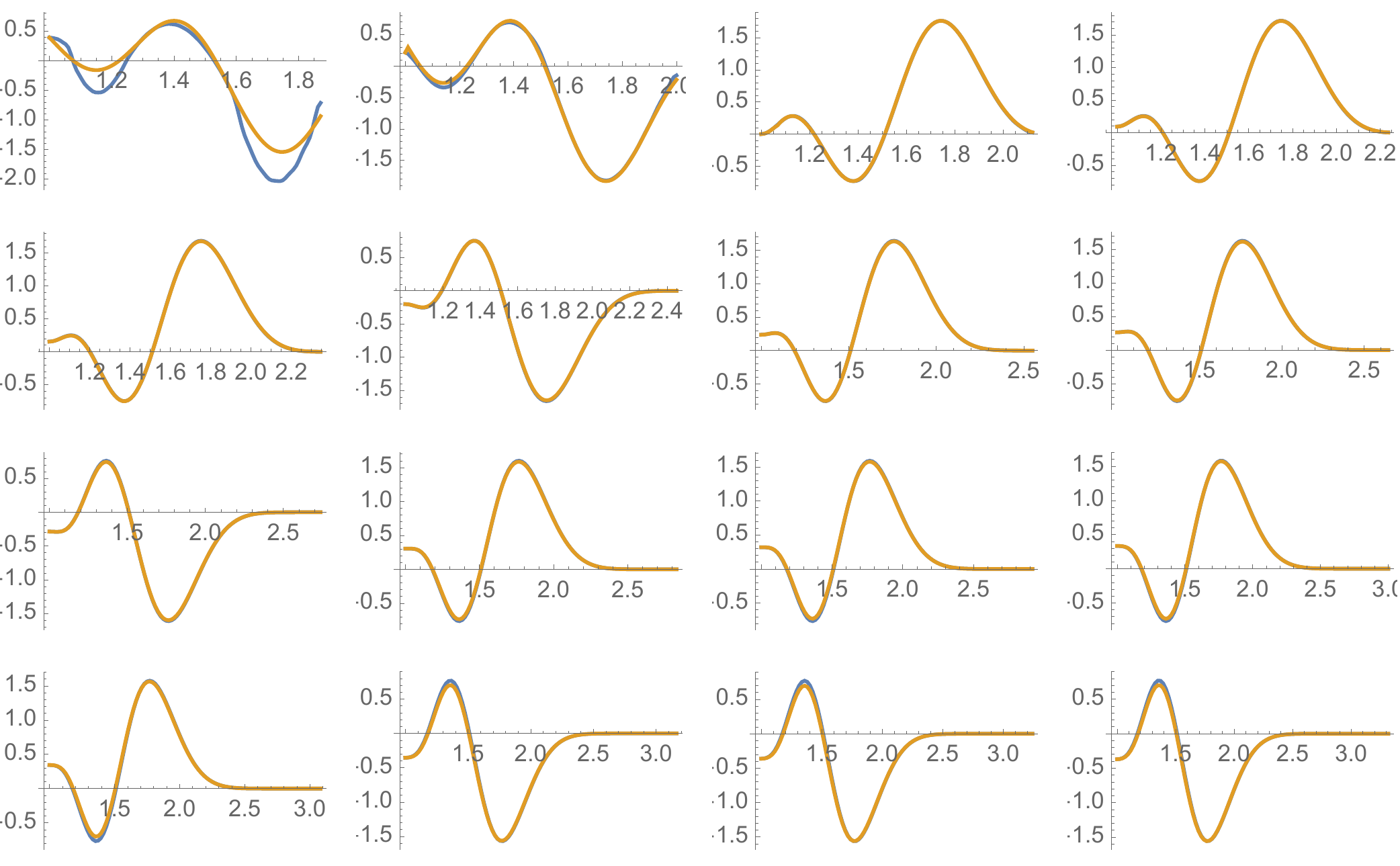}
\end{center}
\caption{agreement of eigenfunctions for the even matrix and the third smallest eigenvalue for the $16$ values of $\mu$ between $3.5$ and $11$. They begin to agree around $\mu=4$\label{even3}}
\end{figure}
\begin{figure}[H]	\begin{center}
\includegraphics[scale=0.55]{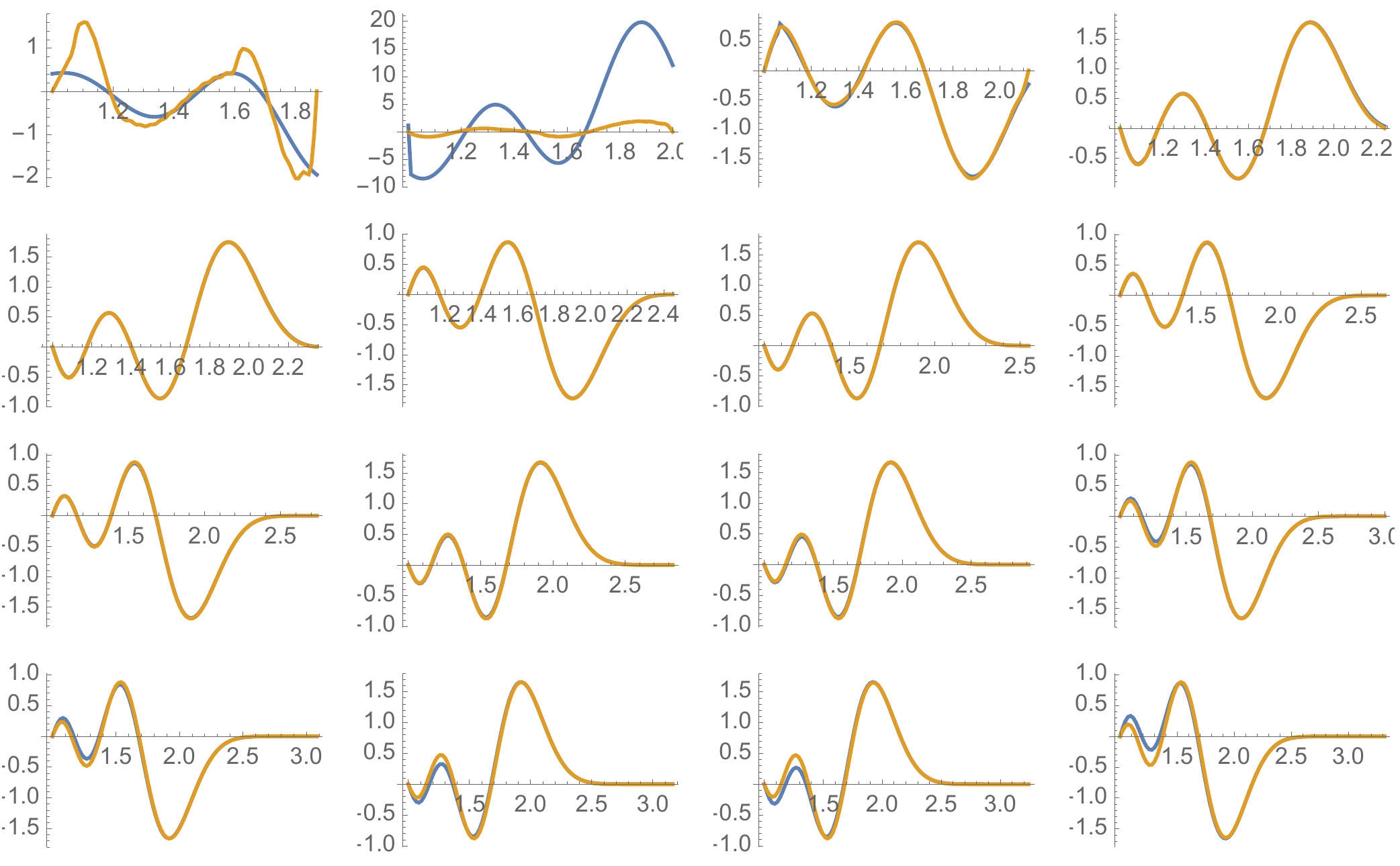}
\end{center}
\caption{agreement of eigenfunctions for the odd matrix and the third smallest eigenvalue for the $16$ values of $\mu$ between $3.5$ and $11$. They begin to agree around $\mu=4.5$\label{odd3}}
\end{figure}
 \begin{figure}[H]	\begin{center}
\includegraphics[scale=0.55]{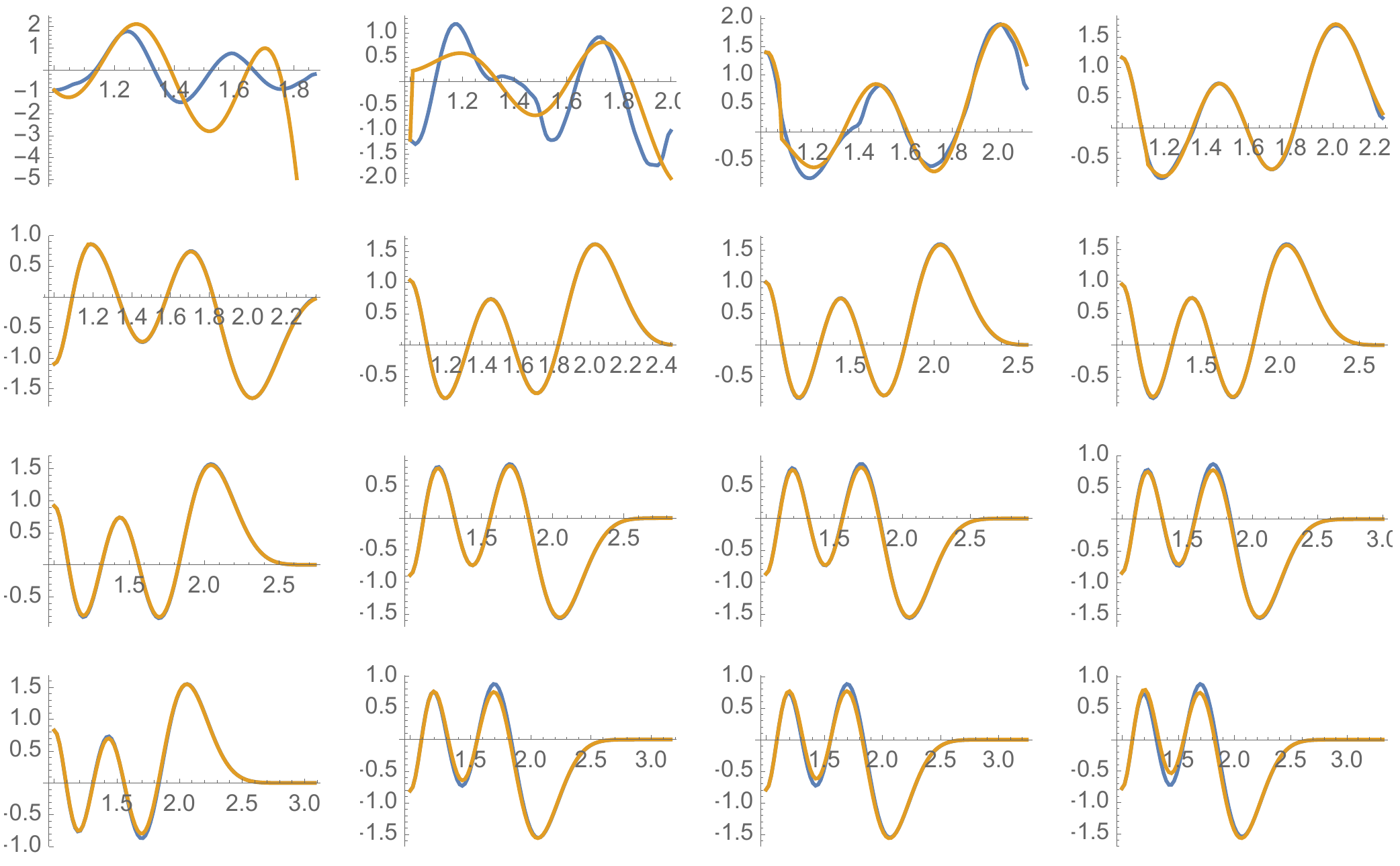}
\end{center}
\caption{agreement of eigenfunctions for the even matrix and the $4$-th smallest eigenvalue for the $16$ values of $\mu$ between $3.5$ and $11$. They begin to agree around $\mu=5$\label{even4}}
\end{figure}
\begin{figure}[H]	\begin{center}
\includegraphics[scale=0.55]{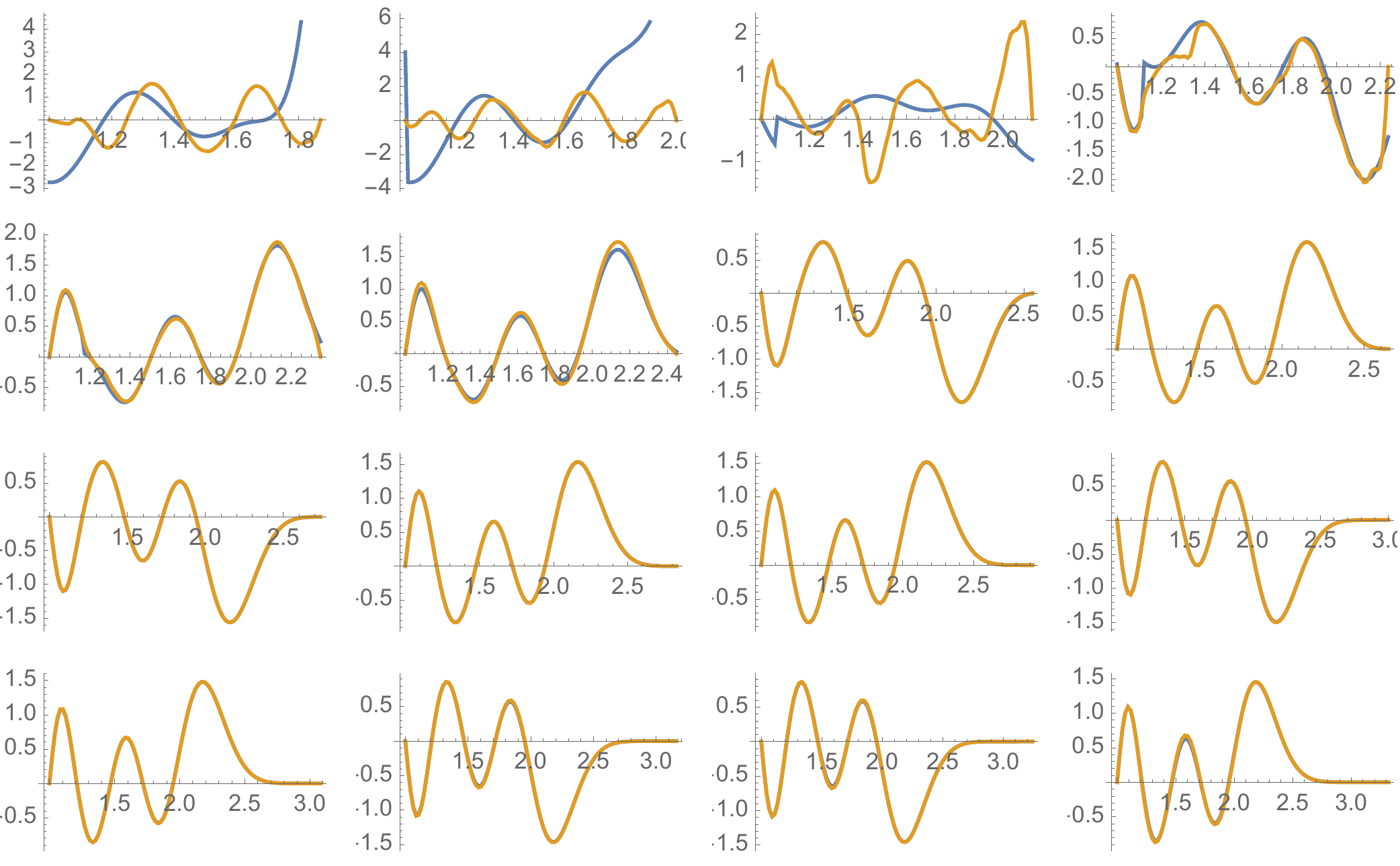}
\end{center}
\caption{agreement of eigenfunctions for the odd matrix and the $4$-th smallest eigenvalue for the $16$ values of $\mu$ between $3.5$ and $11$. They begin to agree around $\mu=5.5$\label{odd4}}
\end{figure}
\begin{figure}[H]	\begin{center}
\includegraphics[scale=0.6]{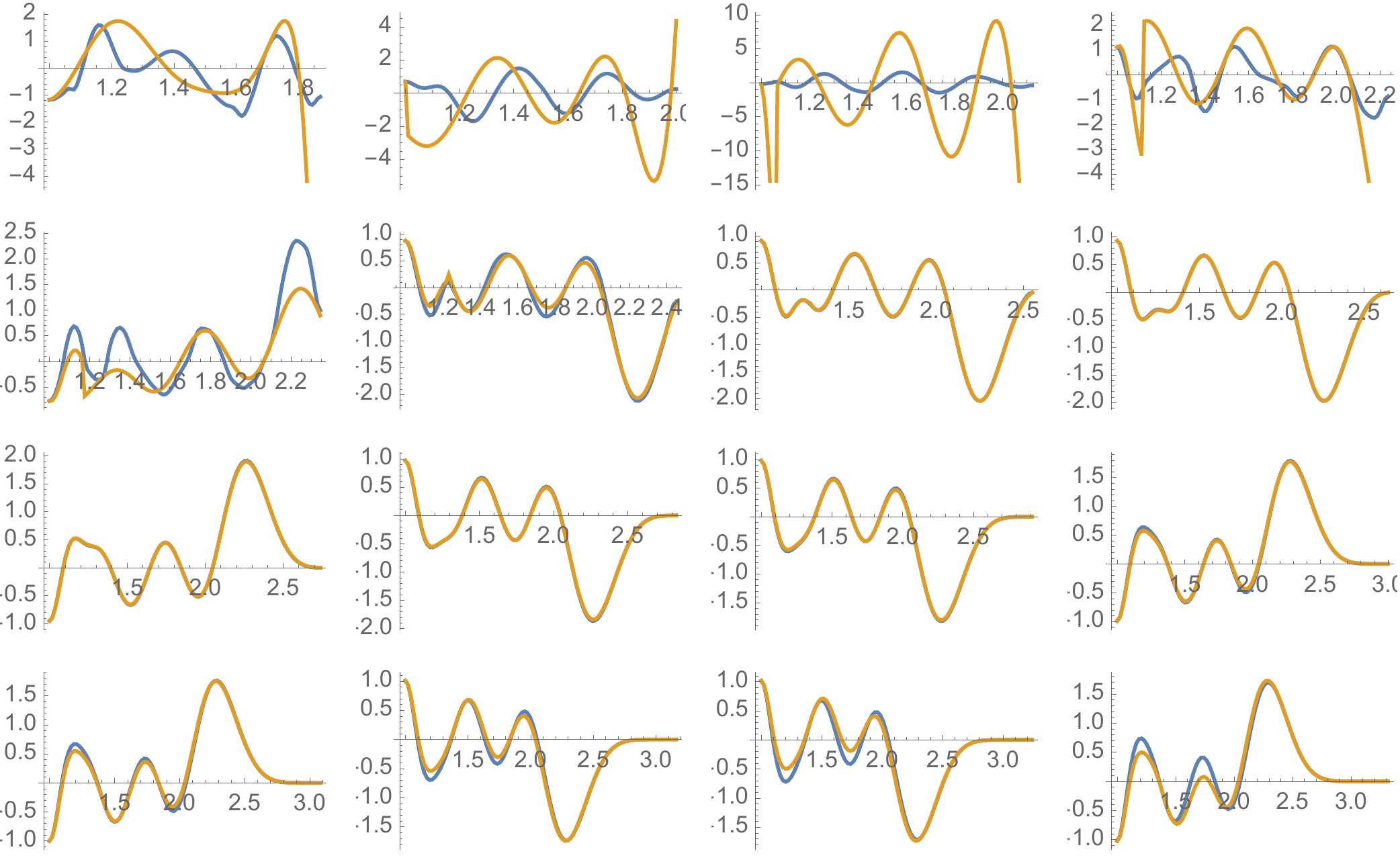}
\end{center}
\caption{agreement of eigenfunctions for the even matrix and the $5$-th smallest eigenvalue for the $16$ values of $\mu$ between $3.5$ and $11$. They begin to agree around $\mu=6$\label{even5}}
\end{figure}
\begin{figure}[H]	\begin{center}
\includegraphics[scale=0.6]{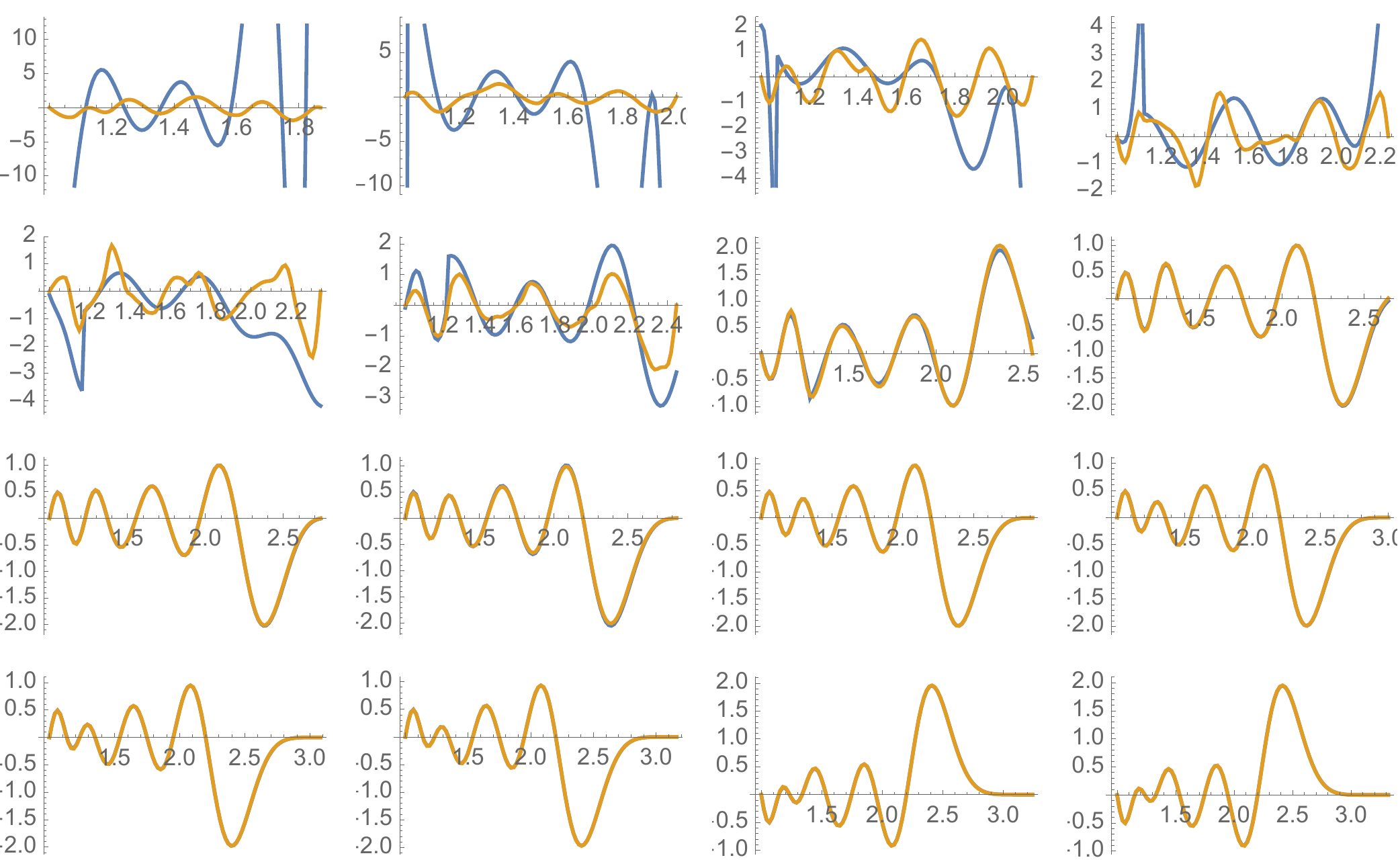}
\end{center}
\caption{agreement of eigenfunctions for the odd matrix and the $5$-th smallest eigenvalue for the $16$ values of $\mu$ between $3.5$ and $11$. They begin to agree around $\mu=6.5$\label{odd5}}
\end{figure}
\begin{figure}[H]	\begin{center}
\includegraphics[scale=0.6]{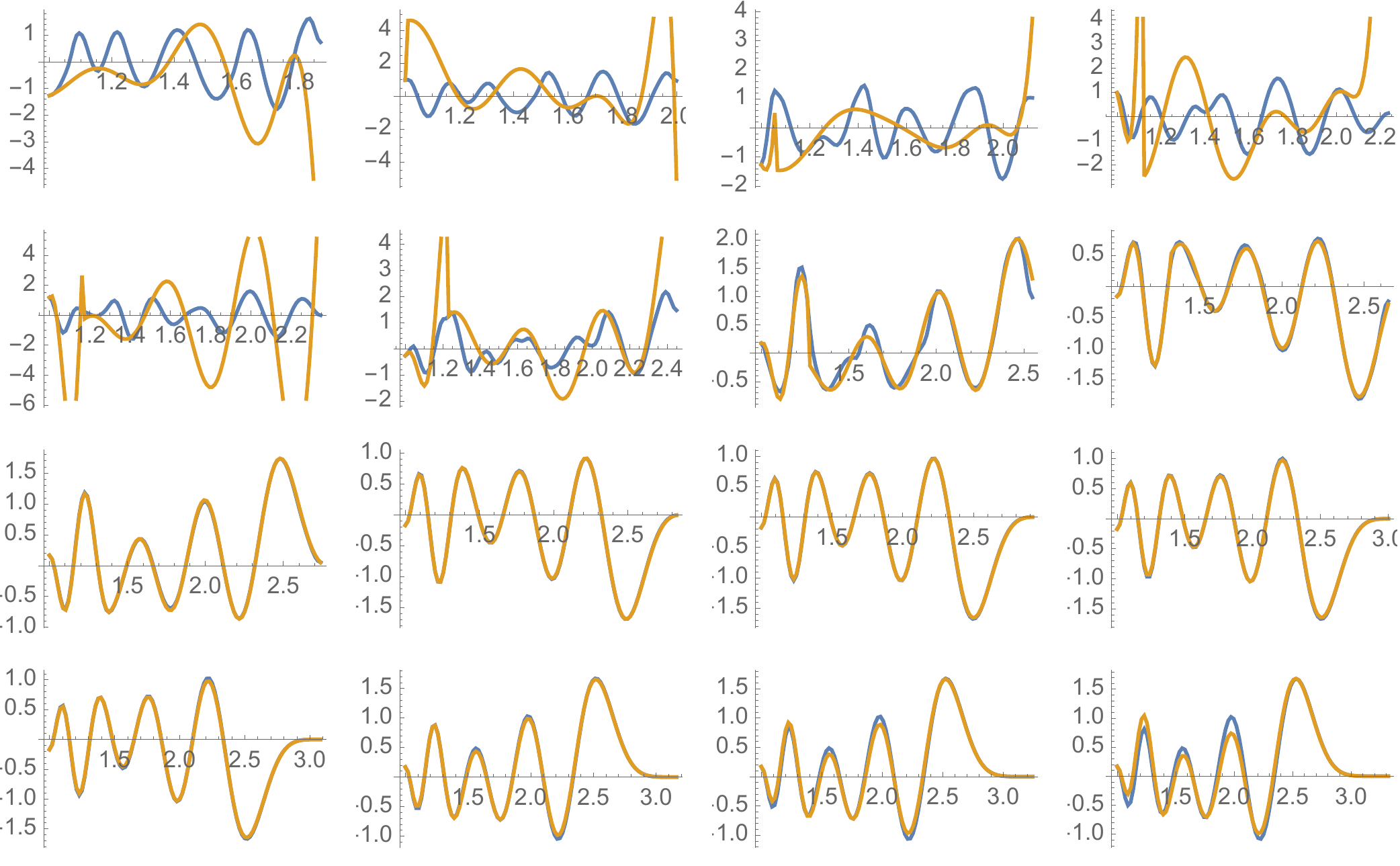}
\end{center}
\caption{agreement of eigenfunctions for the even matrix and the $6$-th smallest eigenvalue for the $16$ values of $\mu$ between $3.5$ and $11$. They begin to agree around $\mu=7$\label{even6}}
\end{figure}
\begin{figure}[H]	\begin{center}
\includegraphics[scale=0.6]{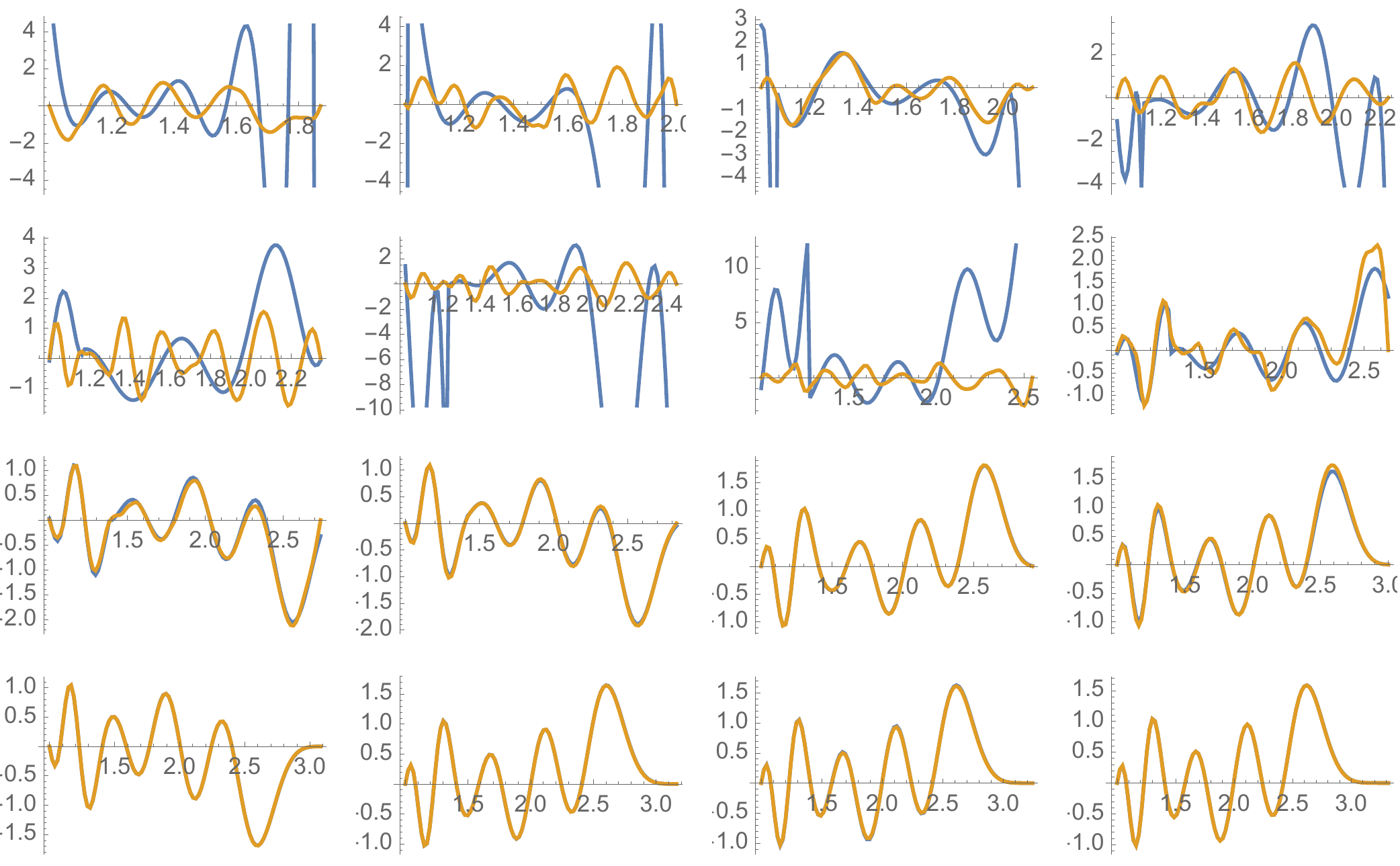}
\end{center}
\caption{agreement of eigenfunctions for the odd matrix and the $6$-th smallest eigenvalue for the $16$ values of $\mu$ between $3.5$ and $11$. They begin to agree around $\mu=7.5$\label{odd6}}
\end{figure}

\section{The spectral triple $\Theta(\lambda,k)=(\cA(\lambda),\cH(\lambda),D(\lambda,k))$}\label{spectrip}

 The spectral triple $\Theta(\lambda,k)=(\cA(\lambda),\cH(\lambda),D(\lambda,k))$ described in this section, whose spectrum has a remarkable similarity with the low lying zeros of the Riemann zeta function is defined through the action by multiplication of the algebra of smooth functions $\cA(\lambda):=C^{\infty}(\R_+^*/\lambda^{2\Z})$ on the Hilbert space $\cH(\lambda):=L^2(\R_+^*/\lambda^{2\Z},d^*u)$. The operator $D(\lambda,k)$ is defined by the following formula
\begin{equation}\label{iD}
D(\lambda,k):=(1-\Pi(\lambda,k))\circ D_0(\lambda)\circ (1-\Pi(\lambda,k)), \ D_0(\lambda):=(-iu\partial_u).
\end{equation}
This is a finite rank perturbation of the standard Dirac operator $D_0(\lambda)$, since by construction the range of the prolate projection $\Pi(\lambda,k)$ is contained in  the domain of $D_0(\lambda)$,  so that one derives
$$
D(\lambda,k)=D_0(\lambda) -\Pi(\lambda,k)D_0(\lambda)-D_0(\lambda)\Pi(\lambda,k)+\Pi(\lambda,k)D_0(\lambda)\Pi(\lambda,k).
$$
\begin{proposition}\label{propDpert}The operator $D(\lambda,k)$, combined with the action of periodic functions by multiplication in $L^2([-L/2,L/2])$ defines a spectral triple.
\end{proposition}
\begin{proof} The operator $D(\lambda,k)$ is a finite rank perturbation of $D_0(\lambda)$,  thus by the Kato-Rellich theorem (see \cite{schmudgen} Proposition 8.6) it   is essentially self-adjoint on any core of $D_0(\lambda)$. The domain of $D(\lambda,k)$ is the same as the domain of $D_0(\lambda)$ and the boundedness of the commutator $[D(\lambda,k),f]$ follows from the boundedness of the perturbation. 
\end{proof} 
To compare the spectrum of $D(\lambda,k)$, for $k$ just below the upper bound $\nu(\lambda^2)\sim 2\lambda^2$ (discussed in section \ref{riemweilexpl}), with the zeros of the Riemann zeta function, one needs to select an appropriate range  of eigenvalues for which the comparison is meaningful. By construction the number of eigenvalues of $D(\lambda,k)$ in the interval $[0,E]$ has the same asymptotic behavior as for $D_0(\lambda)$, and thus differs from the asymptotic behavior of the number $N(E)$   of zeros of the Riemann zeta function with imaginary part in the interval $[0,E]$, namely
\begin{equation}
N(E) =\# \{ \rho \, | \,  \zeta(\rho)=0, \ \text{ and } \ 0<
\Im(\rho)\leq E \}. \label{ri1}
\end{equation}
This number is the sum of two contributions: $N(E)=\langle N(E)\rangle + N_{\rm osc}(E)$. The oscillatory term $N_{\rm osc}(E)$ is of  order  $\log E$ and,  more importantly in this context, one knows that
\begin{equation}\label{nerough}
\langle N(E)\rangle =
 \frac{E}{2\pi}\log \frac{E}{2\pi} - \frac{E}{2\pi}\,.
 \end{equation}
 When considering the operator $D(\lambda,k)$, with $k$ smaller and close to the upper bound $\nu(\lambda^2)$, we let $\mu=\lambda^2$ and we obtain the following 
 \begin{proposition}\label{propDpert1} For  $E=2 \pi \mu$, the number $N'(E)$ of non-zero eigenvalues of the operator $D(\lambda,k)$  in the interval $(0,E]$ fulfills $N'(E)\sim \langle N(E)\rangle$.
\end{proposition}
\begin{proof} It follows from \eqref{commutingPi} that $D(\lambda,k)\gamma=-\gamma D(\lambda,k)$ so that the number of eigenvalues  of $D(\lambda,k)$ of absolute value less than $E$ is $2 N'(E)$ plus the dimension of the kernel of $D(\lambda,k)$. The spectrum of $D(\lambda,k)$   is a perturbation of the spectrum of $D_0(\lambda)$ \ie of  $\{\frac{2 \pi  k}{L}\mid k\in \Z\}$. The perturbation increases the dimension of the kernel of $D(\lambda,k)$  by the dimension of the projection $\Pi(\lambda,k)$ \ie by $k\sim 2 \mu$, up to a $\log \mu$ term.  Thus, the number of non-zero eigenvalues of $D(\lambda,k)$ with absolute value less than $E$ has an approximated size equal to
$$
2 N'(E)\sim \#\left(\Big\{\frac{2 \pi  j}{L}\mid j\in \Z{\Big\}}\cap [-E,E]\right)-2 \mu\sim 2\frac{E L}{2 \pi }-2\mu=2\left(\frac{E}{2\pi}\log \frac{E}{2\pi} - \frac{E}{2\pi}\right).
$$
using $L=\log \mu$ and $\mu= \frac{E}{2\pi}$, which gives the expected estimate.
\end{proof} 

 \subsection{Examples $\mu=5.5, 6.5,7.5,8.5, 9.5$}
In this part we report some numerical evidence showing the close resemblance of the spectrum of $D(\lambda,k)$ with the low lying zeros of the Riemann zeta function, for a sample of small values of $\mu$.
 \subsubsection{$\mu=5.5$}
 ~For $\mu=5.5$, the cosine eigenvalues $\chi(5.5,n)$  are extremely close to $1$ when $n=0,1,2,3,4$ and given for the next values of $n$ in the following table 
 \[
 \begin{tabular}{ l| c r }
 $n$ & $\chi(5.5,n)$ &\\
 \hline 
  5  & 0.99999999999647719857 & \\
  6  & 0.99999999894391115741 &\\
  7  & 0.99999980631702676769 &\\
  8  & 0.99997809227622865324 &\\
  9  & 0.99852183576050441685 &\\
  10  & 0.95065832620623051607 &\\
  11  & 0.57197061534624863399 &\\
  12  & 0.139174533954574303539 &
\end{tabular}
 \]
 Thus one derives that $\nu(5.5)=10$, since the next eigenvalue $0.5719706153$ is far from $1$. One has $2\pi 5.5\sim 34.5575$.  The following table compares the positive  eigenvalues $\lambda_j=\lambda_j(D(\lambda,k))$ of $D(\lambda,k)$ (reported on the left column)  with the imaginary part $\zeta_j$ of the first zeros of the Riemann zeta function  (right column)  
 \[
  \begin{tabular}{ l| c r }
 $\lambda_j$ & $\zeta_j$ &\\
 \hline 
  14.781  & 14.1347 & \\
  21.701  & 21.022 &\\
  25.547  & 25.0109 &\\
  29.345  & 30.4249 &\\
  33.168  & 32.9351 &
\end{tabular}
 \]
 The spectral visualization is shown in Figure \ref{compare5p5}, with the zeta zeros at the bottom  
 \begin{figure}[H]	\begin{center}
\includegraphics[scale=0.7]{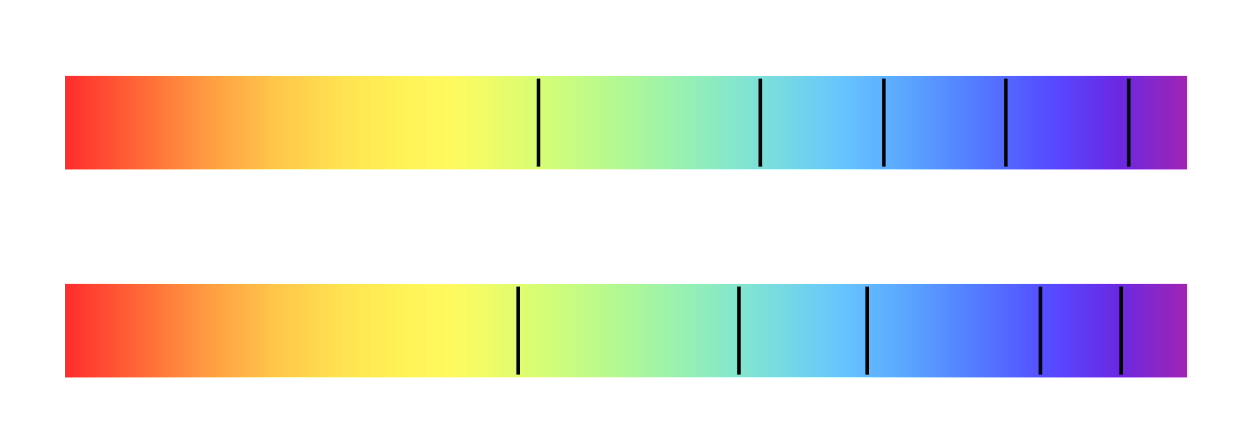}
\end{center}
\caption{First $5$ non-zero eigenvalues for the Dirac in upper line and imaginary parts of zeros of zeta in lower line\label{compare5p5}}
\end{figure}

  \subsubsection{$\mu=6.5$}
 ~For $\mu=6.5$ the cosine eigenvalues $\chi(6.5,n)$  are extremely close to $1$   when $n=0,1,2,3,4,5,6$; for $7\le n\le 14$ the  values are reported in the following table
 \[ 
 \begin{tabular}{ l| c r }
 $n$ & $\chi(6.5,n)$ &\\
 \hline 
  7  & 0.99999999998668315975 & \\
  8  & 0.99999999731589077585 &\\
  9  & 0.99999963978717981581 &\\
  10  & 0.99996808936687677767 &\\
  11 & 0.99821407841789989100 &\\
  12  & 0.94788066237037484836 &\\
  13  & 0.57534099083086049406 &\\
  14  & 0.14710511279564130503 &
\end{tabular}
 \]
 Thus one has $\nu(6.5)=12$, since the next eigenvalue $0.5753409908$ is far from $1$. One has  $2\pi 6.5\sim 40.8407$.  Once again,  the following  table reports the   eigenvalues  $\lambda_j=\lambda_j(D(\lambda,k))$ compared with the imaginary part $\zeta_j$ of 
 the first zeros of the zeta function. 
  \[
 \begin{tabular}{ l| c r }
 $\lambda_j$ & $\zeta_j$ &\\
 \hline 
  13.936 & 14.1347 & \\
  20.580 & 21.022 &\\
  24.690 & 25.0109 &\\
  30.194 & 30.4249 &\\
  33.454 & 32.9351 &\\
  36.826 & 37.5862 &\\
  40.259 & 40.9187 &
\end{tabular}
 \]
 The spectral visualization is shown in Figure \ref{compare6p5} with the zero of the zeta function in the second line  
 \begin{figure}[H]	\begin{center}
\includegraphics[scale=0.7]{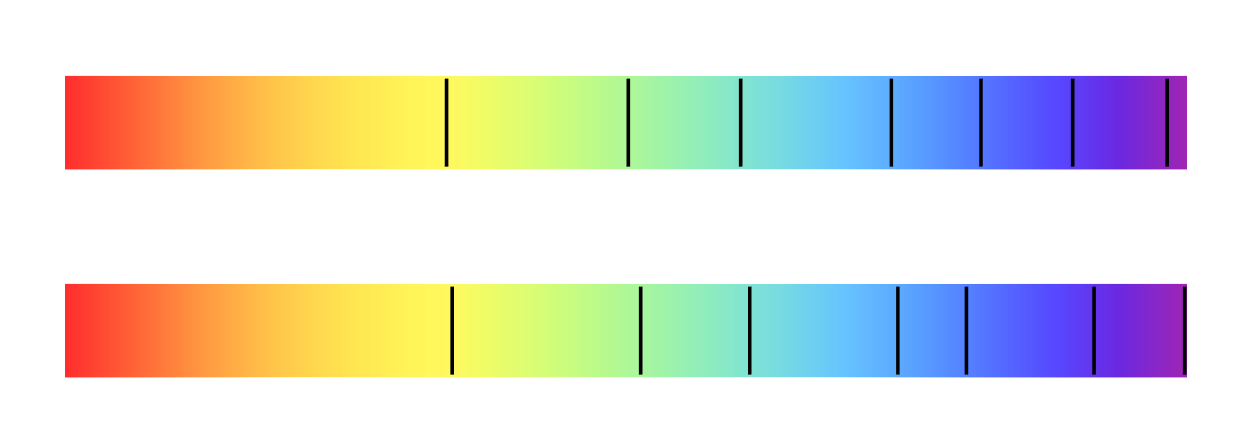}
\end{center}
\caption{First $7$ non-zero eigenvalues for the Dirac in upper line and imaginary parts of zeros of zeta in lower line.\label{compare6p5}}
\end{figure} 
 \subsubsection{$\mu=7.5$}
 ~The cosine eigenvalues $\chi(7.5,n)$  are extremely close to $1$ for $n=0,1,2,3,4,5,6,7,8$, and then given by
 \[
 \begin{tabular}{ l| c r }
 $n$ & $\chi(7.5,n)$ &\\
 \hline 
  9 & 0.99999999996397226733 & \\
  10 & 0.99999999453062631606 &\\
  11 & 0.99999941709770526957 &\\
  12 & 0.99995709581648305854 &\\
  13 & 0.99792322303841470726 &\\
  14 & 0.94552083061302325507 &\\
  15 & 0.57809629788957190907 &\\
  16 & 0.15383636015962926720 &
\end{tabular}
 \]
 Thus one has $\nu(7.5)=14$ since the next eigenvalue $0.5780962979$ is far from $1$. One has $2\pi 7.5\sim 47.1239$.
 Next table compares    the eigenvalues  $\lambda_j=\lambda_j(D(\lambda,k))$ with the imaginary part $\zeta_j$ of 
 the first zeros of the zeta function. 
 \[
 \begin{tabular}{ l| c r }
 $\lambda_j$ & $\zeta_j$ &\\
 \hline 
  15.060 & 14.1347 & \\
  21.683 & 21.022 &\\
  24.948 & 25.0109 &\\
  30.979 & 30.4249 &\\
  33.243 & 32.9351 &\\
  37.406 & 37.5862 &\\
  40.514 & 40.9187 &\\
  43.643 & 43.3271 &\\
 46.658 & 48.0052 &
\end{tabular}
 \]
 The spectral visualization is shown in Figure \ref{compare7p5}, with zeta zeros in the second line  
 \begin{figure}[H]	\begin{center}
\includegraphics[scale=0.7]{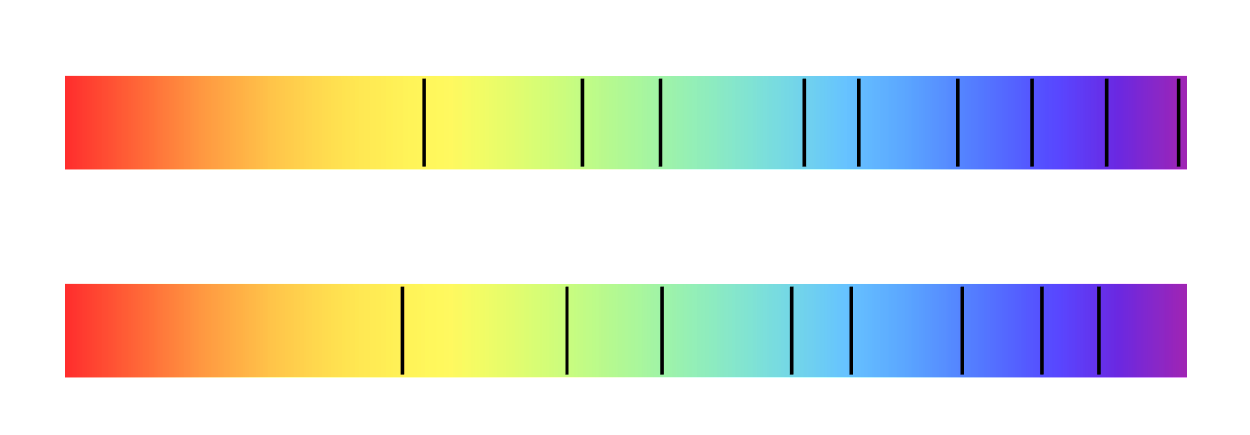}
\end{center}
\caption{First $9$ non-zero eigenvalues for the Dirac in upper line and imaginary parts of zeros of zeta in lower line\label{compare7p5}}
\end{figure}

  \subsubsection{$\mu=8.5$}
 The $\chi(8.5,n)$  are extremely close to $1$ for $n\leq 10$, and the next ones are given by
 \[
 \begin{tabular}{ l| c r }
 $n$ & $\chi(8.5,n)$ &\\
 \hline 
  11 & 0.99999999992101000288 & \\
  12 & 0.99999999034148375362 &\\
  13 & 0.99999913999089362040 &\\
  14 & 0.99994536408530411219 &\\
  15 & 0.99764801726717553636 &\\
  16 & 0.94347292951033144975 &\\
  17 & 0.58041289343441020661 &\\
  18 & 0.15967051202562674536 &
\end{tabular}
 \]
 Thus one has $\nu(8.5)=16$, (the next eigenvalue $0.5804128934$ is far from $1$) and  $2\pi 8.5\sim 53.4071$. The following table reports the eigenvalues  $\lambda_j=\lambda_j(D(\lambda,k))$ compared to the imaginary part $\zeta_j$ of 
 the first zeros of the zeta function.
  \[
 \begin{tabular}{ l| c r }
 $\lambda_j$ & $\zeta_j$ &\\
 \hline 
  14.887  & 14.1347 & \\
  20.778  & 21.022 &\\
  25.535  & 25.0109 &\\
  29.928  & 30.4249 &\\
  32.473  & 32.9351 &\\
  37.965  & 37.5862 &\\
  41.088  & 40.9187 &\\
  43.741  & 43.3271 &\\
 46.685  & 48.0052 &\\
 49.910  & 49.7738 &\\
 52.845  & 52.9703 &
\end{tabular}
 \]
  The spectral visualization is reported in Figure \ref{compare8p5}, with the zeta zeros in the second line. 
   \begin{figure}[H]	\begin{center}
\includegraphics[scale=0.7]{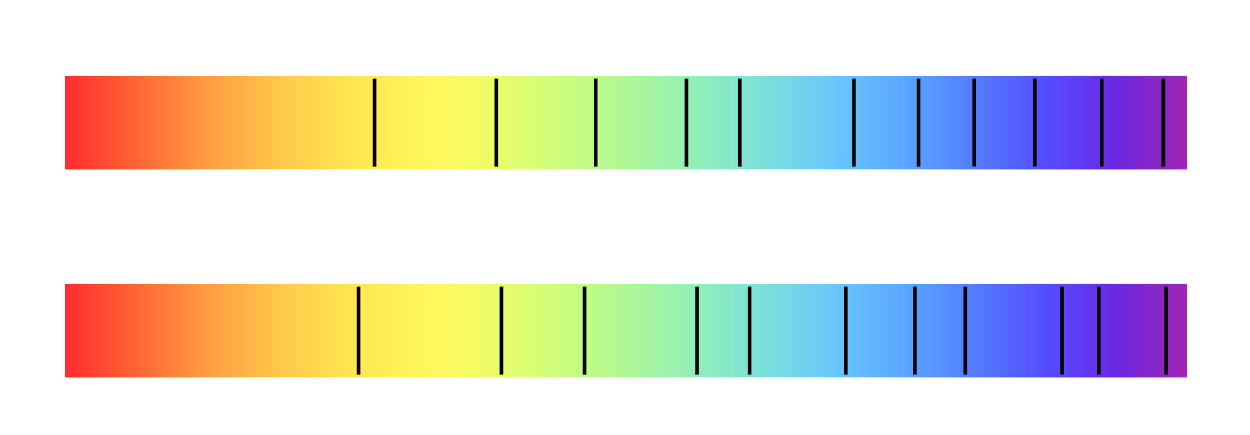}
\end{center}
\caption{First $11$ non-zero eigenvalues for the Dirac in upper line and imaginary parts of zeros of zeta in lower line\label{compare8p5}}
\end{figure}

  \subsubsection{$\mu=9.5$}
~For  $\mu=9.5$ the cosine eigenvalues $\chi(9.5,n)$  are extremely close to $1$ when $ 0\le n\le 12$, and  for $13\le n\le 20$ they are reported in the table
 \[
 \begin{tabular}{ l| c r }
 $n$ & $\chi(9.5,n)$ &\\
 \hline 
  13 & 0.99999999984990646525 & \\
  14 & 0.99999998455736228573 &\\
  15 & 0.99999881131048713492 &\\
  16 & 0.99993308190344158164 &\\
  17 & 0.99738707752987412262 &\\
  18 & 0.94166650390462098514 &\\
  19 & 0.58240244869697875785 &\\
  20 & 0.16480962032526478957 &
\end{tabular}
 \]
 Thus one has $\nu(9.5)=18$, since the next eigenvalue $0.5824024487$ is far from $1$. One has $2\pi 9.5\sim 59.6903$  and the following table reports the eigenvalues $\lambda_j=\lambda_j(D(\lambda,k))$ compared to the imaginary part $\zeta_j$ of the first zeros of the zeta function 
 \[
 \begin{tabular}{ l| c r }
 $\lambda_j$ & $\zeta_j$ &\\
 \hline 
  13.998 & 14.1347 & \\
  21.501 & 21.022 &\\
  25.121 & 25.0109 &\\
  30.689 & 30.4249 &\\
  33.583 & 32.9351 &\\
  37.813 & 37.5862 &\\
  41.272 & 40.9187 &\\
  43.050 & 43.3271 &\\
 47.319 & 48.0052 &\\
 50.190 & 49.7738 &\\
 53.026 & 52.9703 &\\
 55.731 & 56.4462 &\\
 58.581 & 59.347 &
\end{tabular}
 \]
 The spectral visualization is shown in Figure \ref{compare9p5}, with the zeta zeros in the second line   
 \begin{figure}[H]	\begin{center}
\includegraphics[scale=0.7]{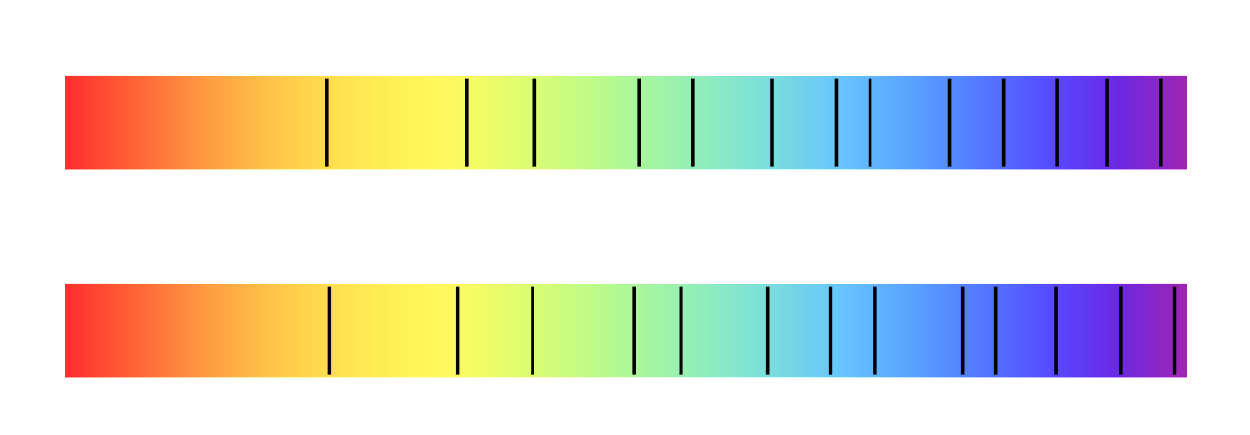}
\end{center}
\caption{First $13$ non-zero eigenvalues for the Dirac in upper line and imaginary parts of zeros of zeta in lower line\label{compare9p5}}
\end{figure}

  \subsubsection{$\mu=10.5$}
  ~For  $\mu=10.5$ the cosine eigenvalues $\chi(10.5,n)$  are extremely close to $1$ when $0\le n\le 14$, and  for $15\le n\le 22$ they are reported in the table
  \[
 \begin{tabular}{ l| c r }
 $n$ & $\chi(10.5,n)$ &\\
 \hline 
  15 & 0.99999999974270022369 & \\
  16 & 0.99999997703659571104 &\\
  17 & 0.99999843436641476606 &\\
  18 & 0.99992039045021729410 &\\
  19 & 0.99713907784499135361 &\\
  20 & 0.94005235637340584775 &\\
  21 & 0.58413979804862029634 &\\
  22 & 0.16939519615152177689 &
\end{tabular}
 \]
  Thus one has $\nu(10.5)=20$, since the next eigenvalue $0.5841397980$ is far from $1$. One also has $2\pi 10.5\sim 65.9734$.  The table of eigenvalues (left column) compared to the first zeta zeros (right column) is
\[
 \begin{tabular}{ l| c r }
 $\lambda_j$ & $\zeta_j$ &\\
 \hline 
  14.450  & 14.1347 & \\
  21.455  & 21.022 &\\
  25.356  & 25.0109 &\\
  30.345  & 30.4249 &\\
  32.600  & 32.9351 &\\
  37.410  & 37.5862 &\\
  40.387  & 40.9187 &\\
  42.895  & 43.3271 &\\
 48.095  & 48.0052 &\\
 50.346  & 49.7738 &\\
 53.272  & 52.9703 &\\
 56.050  & 56.4462 &\\
 58.737  & 59.347 &\\
 61.386  & 60.8318 &\\
 63.949  & 65.1125&
\end{tabular}
 \]
The spectral visualization is shown in Figure \ref{compare10p5}, with zeta zeros in the second line   
 \begin{figure}[H]	\begin{center}
\includegraphics[scale=0.7]{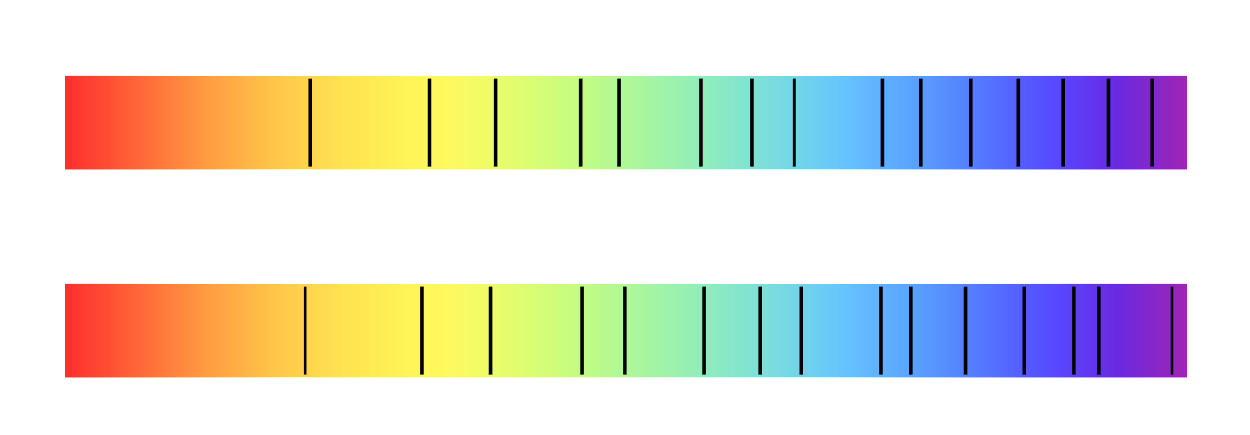}
\end{center}
\caption{First $15$ non-zero eigenvalues for the Dirac in upper line and imaginary parts of zeros of zeta in lower line.\label{compare10p5} }
\end{figure}

   \subsection{Average discrepancy}
  
  For an objective comparison of  the $N'(2\pi \mu)$ eigenvalues $\lambda_j$ of size up to  $2\pi \mu$, with the imaginary parts $\zeta_j$ of the zeros of the Riemann zeta function, one has at disposal  the following three possible measures of the discrepancy 
  \begin{enumerate}
  \item Mean absolute error:
  $$
  A(\mu):=\frac{1}{N'(2\pi \mu)}\sum \vert \lambda_j-\zeta_j \vert
  $$ 
  When this error is computed for the values of $\mu$ used in the previous pages it gives the following list of values
  $$
  A(5.5)=0.635176, \ A(6.5)=0.44693, \ A(7.5)=0.528827, $$ $$\ A(8.5)=0.456739, \ A(9.5)=0.395068
  $$
  \item Root-mean-square deviation. It is defined as the square root of the average value of the square deviation
  $$
  R(\mu):=\sqrt{\frac{1}{N'(2\pi \mu)}\sum (\lambda_j-\zeta_j )^2}
  $$
  This gives the following list of values
$$
  R(5.5)=0.691088, \ R(6.5)=0.48858, \ R(7.5)=0.650648, $$ $$\ R(8.5)=0.562489, \ R(9.5)=0.459776
  $$
  \item Normalized root-mean-square deviation. This deviation  is obtained by dividing the root-mean-square deviation by the diameter of the range of the variables. It is invariant under affine transformations and is thus a good measure of the discrepancy, usually expressed as a percentage. The diameter of the range of the variables is here equal to  $2\pi \mu -14$, and this gives the list, 
    $$
  NR(5.5)=0.0375848, \  NR(6.5)=0.0185609, \ NR(7.5)=0.0205914, $$ $$ NR(8.5)=0.0148189, \ NR(9.5)=0.0103126, \ NR(10.5)=0.00995148
  $$ 
  \end{enumerate}
  These numbers  show that the normalized root-mean-square deviation is steadily improving and reaches $1\%$ (one percent) for $\mu=9.5$ and then  drops to less than one percent for $\mu=10.5$.

  \section{Zeta zeros from eigenvalues of spectral triples}\label{2zeros}
  
  In the previous section we explored the low lying eigenvalues of the spectral triples $\Theta(\lambda,k)=(\cA(\lambda),\cH(\lambda),D(\lambda,k))$ for $k=2\ell$ an even number as close as possible to the boundary $\nu(\lambda^2)\sim 2 \lambda^2$ of the allowed interval. These numerical results give evidence of a deep relation between the low lying spectrum  $\lambda_n(D(\lambda,k))$ of these spectral triples and the low lying zeros of the Riemann zeta function. The dependence on the parameters $(\lambda,k)$, and the difference between the  growth of the eigenvalues and that of the  zeros of zeta,   show that the relation is certainly  more subtle than a simple equality between the eigenvalues $\lambda_n(D(\lambda,k))$ and the imaginary part $\zeta_n$ of the zeros.

 The main observation of this section is that, for any $n\in \N$ there are special values of the parameter $\lambda$ at which the dependence of  $\lambda_n(D(\lambda,k))$ on $k$ disappears. For these special values of $\lambda$ the {\emph common value} of the $\lambda_n(D(\lambda,k))$ {\emph coincides} with the imaginary part $\zeta_n$ of the $n$-th zero of the Riemann zeta function. Moreover, these special values of $\lambda$ form a geometric progression whose scale factor is the exponential of $\pi/\zeta_n$.
   
   This observation was first experimentally tested and it will be fully and conceptually justified  in  section \ref{sectzetacycles}.
   
   We shall pursue $4$ different criterions to detect these special values of $\lambda$. They are
   \begin{itemize}
   \item Comparison of $\lambda_n(D(\lambda,2\ell))$ with $\lambda_n(D(\lambda,2\ell+1))$ (\S \ref{sectcriter})
   \item Evolution of $\lambda_n(D(\lambda,k))$ as a function of $\lambda$ (\S \ref{sectevoleigen})
   \item Quantization criterion $x^{2iy}=1$ applied to the point $(\lambda,\lambda_n(D(\lambda,k)))$  (\S \ref{sectquantiz})
   \item How far is the eigenvector $\xi_n(D(\lambda,k))$ for $D(\lambda,k)$  from being an eigenvector of $D_0(\lambda)$   	
   \end{itemize}

   The numerical tests of these criterions show their agreement, but the precision becomes very sharp when one applies the last criterion. Applying the last method for the small range of $\lambda$ in the interval $(2,4)$ one obtains the agreement with the first $31$ zeros $\zeta_n$ ($n\leq 31$) of zeta with  sufficient accuracy to assess the probability of a fortuitous coincidence at $10^{-50}$.
    
    \subsection{The criterion $\lambda_n(D(\lambda,2 \ell))\sim \lambda_n(D(\lambda,2 \ell+1))$}\label{sectcriter}
    The first step in order to detect the special values of $\lambda$ is to see what happens if  one replaces $k=2\ell$ by the odd number $k+1=2\ell+1$. One sees that the positive eigenvalues $\lambda_n(D(\lambda,*))$ decrease and actually agree for special values of $\lambda$. We first briefly explain why  $\lambda_n(D(\lambda,2 \ell))\geq  \lambda_n(D(\lambda,2 \ell+1))$ and then display some numerical results showing the coincidence for special values of $\lambda$.  
   By construction, the kernel of $D(\lambda,k)$ contains the range of $\Pi(\lambda,k)$ and is thus at least of dimension $k$. Moreover by \eqref{commutingPi} one has, for the grading $\gamma$ of  $\cH(\lambda)$,
  \begin{equation}\label{diracpm}
 \gamma \ D(\lambda,k)=- D(\lambda,k) \ \gamma
  \end{equation}
  The kernel of the operator $D_0(\lambda)$ is one dimensional and given by the constant function $1_\lambda$ which is even (\ie $\gamma(1_\lambda)=1_\lambda$).
 This implies that the graded index of the operator $D_0(\lambda)$  is equal to $1$. Then by stability of the index it follows that the graded index of the operator $D(\lambda,k)$  is also equal to $1$. This means that the signature of the restriction of $\gamma$ to the  kernel of $D(\lambda,k)$ is $1$ and hence that the dimension of $\ker (D(\lambda,k))$ is an odd number. 
   Thus for $k=2\ell$ even it is natural to expect this kernel to be of dimension $k+1$. This entices one to compare the two non-zero eigenvalues $\lambda_n(D(\lambda,k))$ and $\lambda_n(D(\lambda,k+1))$. By construction one has $\Pi(\lambda,k)<\Pi(\lambda,k+1)$, and we now explain why the positive eigenvalues of these operators, arranged in increasing order, fulfill the inequality
   \begin{equation}\label{diracpm1}
  	\lambda_n(D(\lambda,k+1))\leq \lambda_n(D(\lambda,k))\qqq n,\lambda 
  \end{equation}
 
  \begin{lemma}\label{diracpm2} Let $A$ be a self-adjoint matrix of dimension $N$, and $E\subset \Qer A$ a subspace of its kernel. Then the positive eigenvalues $\mu_n(A)$ arranged in decreasing order fulfill
 \begin{equation}\label{diracpmeigen}
  \mu_n(A)=\max_{F\mid \dim F=n\atop F\perp E} \ \min_{\xi \in F\atop \Vert \xi\Vert=1}\langle \xi \mid A \xi\rangle
\end{equation} 	
  \end{lemma}
  \begin{proof} By the mini-max theorem of Courant-Fisher one has 
  $$
  \mu_n(A)=\max_{F\mid \dim F=n} \ \min_{\xi \in F\atop \Vert \xi\Vert=1}\langle \xi \mid A \xi\rangle
  $$ 
and we need to show that the added condition that $F$ is perpendicular to $E$ does not change the maximum. It can only lower it and it is enough to check that the choice of $F$ which reaches the maximum in the Courant-Fisher formula does fulfill $F\perp E$. Indeed this $F$ is the linear span of the eigenvectors for eigenvalues $\mu_k(A)$ for $k\leq n$, and all these eigenvectors are orthogonal to the kernel of $A$ since $\mu_k(A)\geq \mu_n(A)>0$ for $k\leq n$. \end{proof}

\begin{proposition}\label{diracpm3}Let $D\in M_N(\C)$ be a self-adjoint matrix.\newline
$(i)$~Let $P\in M_N(\C)$ be a projection (self-adjoint idempotent) and $Q=1-P$, $D_P:=QDQ$. Then the positive eigenvalues of $D_P$ arranged in decreasing order fulfill the equality
 \begin{equation}\label{diracpmeigen1}
  \mu_n(D_P)=\max_{F\mid \dim F=n \atop F\perp P} \ \min_{\xi \in F \atop \Vert \xi\Vert=1}\langle \xi \mid D\xi\rangle
  \end{equation}
$(ii)$~Let $P_j\in M_N(\C)$	be projections such that $P_1\leq P_2$. Then, with the notations of $(i)$ the positive eigenvalues of  $D_{P_j}$ fulfill the inequality
\begin{equation}\label{diracpmeigen2}
\mu_n(D_{P_2})\leq \mu_n(D_{P_1})
 \end{equation}
\end{proposition}
\begin{proof}$(i)$~By \eqref{diracpmeigen} applied for $A=D_P$ and $E=P(\C^N)$ one has 
$$
\mu_n(D_P)=\max_{F\mid \dim F=n,F\perp E} \ \min_{\xi \in F\mid \Vert \xi\Vert=1}\langle \xi \mid D_P \xi\rangle
$$
and for $\xi\perp E$ one has $Q\xi=\xi$ so that
$$
\langle \xi \mid D_P \xi\rangle=\langle \xi \mid Q D Q \xi\rangle=\langle Q\xi \mid D Q \xi\rangle=\langle \xi \mid D\xi\rangle
$$
 which gives \eqref{diracpmeigen1}.\newline
 $(ii)$~We apply \eqref{diracpmeigen1} to  $\mu_n(D_{P_j})$. The condition $F\perp P_2$ is more restrictive than $F\perp P_1$ so one obtains \eqref{diracpmeigen2}.\end{proof}

 %\newpage

 \begin{figure}[H]	\begin{center}
\includegraphics[scale=0.5]{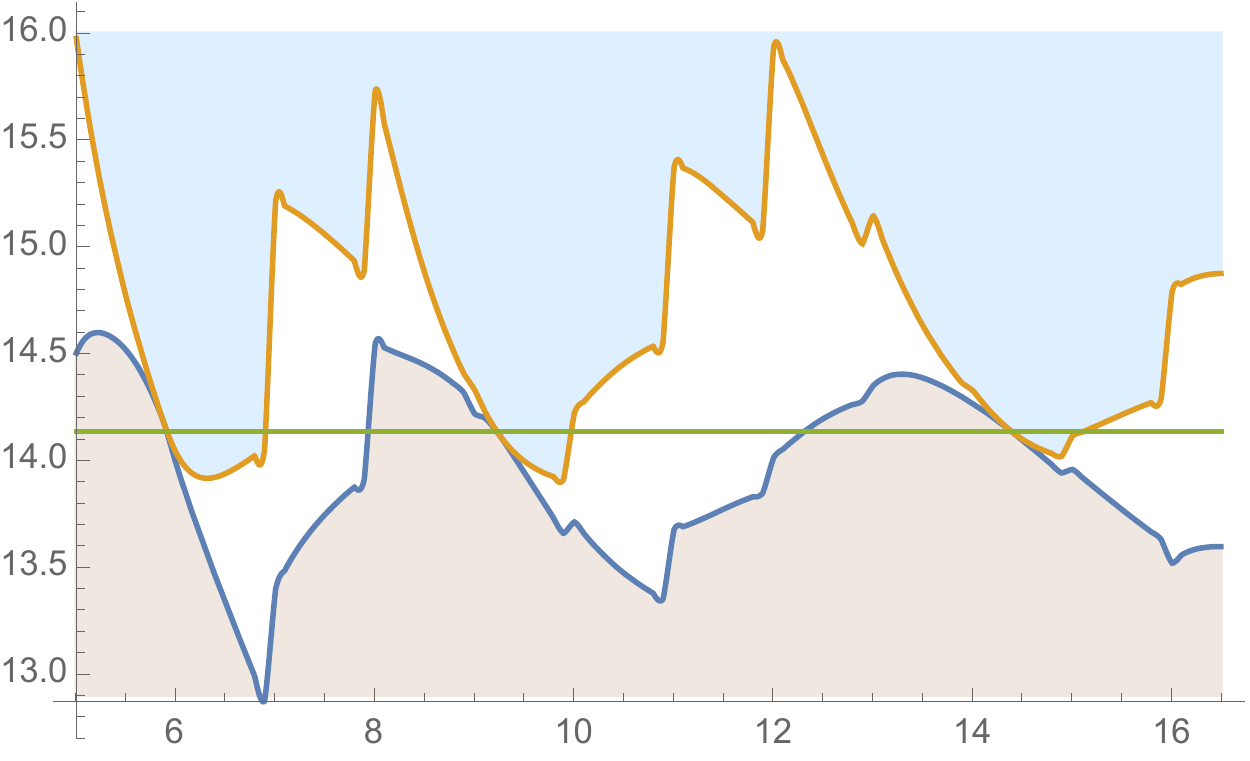}
\end{center}
\caption{First  eigenvalue, the lower graph is that of $\lambda_1(D(\lambda,k+1))$ and the upper graph is that of $\lambda_1(D(\lambda,k))$. The horizontal line is the imaginary part of the first zero of zeta\label{firsteigen}}
\end{figure}
  \begin{figure}[H]	\begin{center}
\includegraphics[scale=0.5]{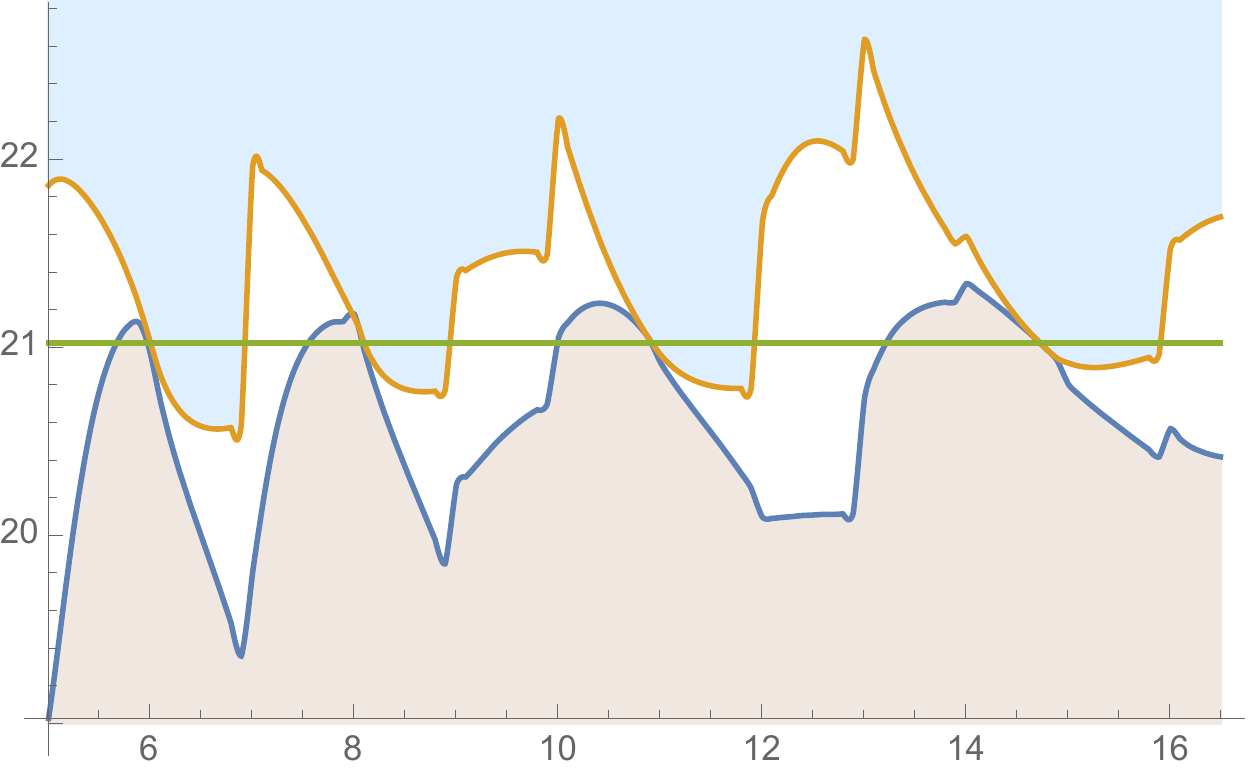}
\end{center}
\caption{Second  eigenvalue, the lower graph is that of $\lambda_2(D(\lambda,k+1))$ and the upper graph is that of $\lambda_2(D(\lambda,k))$. The horizontal line is the imaginary part of the second zero of zeta\label{secondeigen}}
\end{figure}

Applying the criterion $\lambda_n(D(\lambda,k))\sim \lambda_n(D(\lambda,k+1))$ to determine the relevant values of $\mu\in I= [5,16.5]$, \ie by minimizing the difference $\lambda_n(D(\lambda,k))- \lambda_n(D(\lambda,k+1))$ on the finite set of $\mu\in \frac{1}{10}\Z\cap I$, one obtains the approximate list of first $31$ zeros of zeta shown in Figure \ref{usedplusminus}. 
 \begin{figure}[H]	\begin{center}
\includegraphics[scale=0.3]{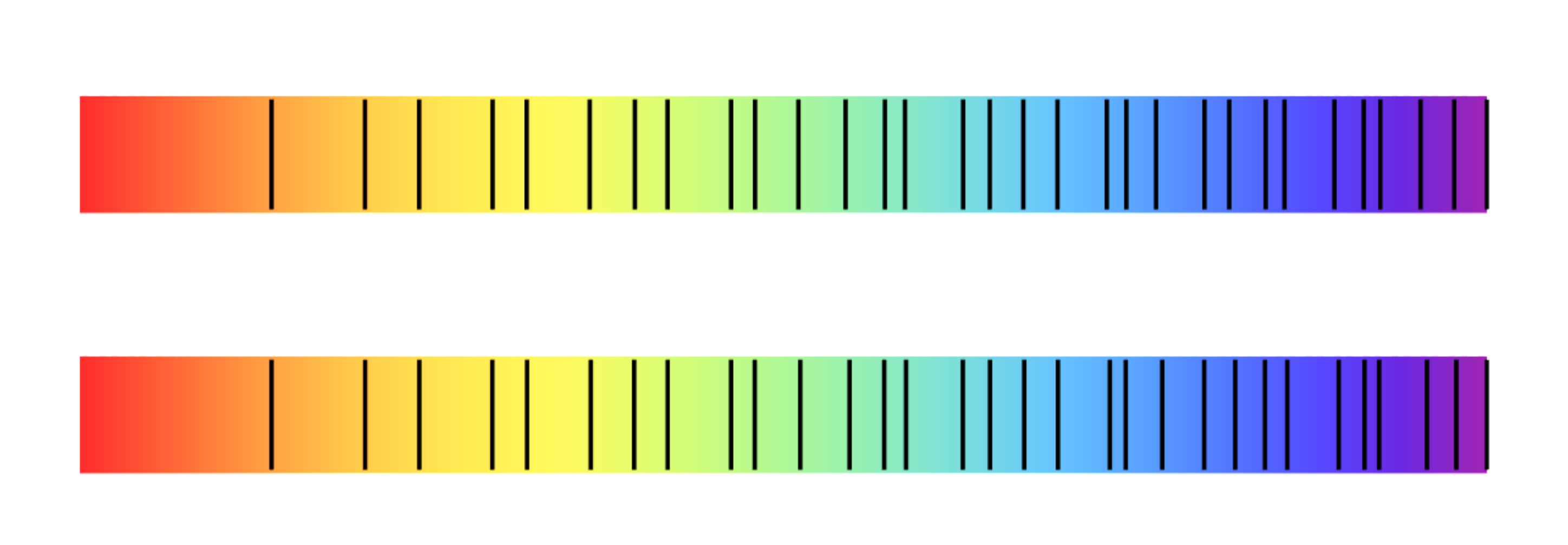}
\end{center}
\caption{Using the criterion $\lambda_n(D(\lambda,k))\sim \lambda_n(D(\lambda,k+1))$\label{usedplusminus}}
\end{figure}

\subsection{Continuous evolution of  non-zero eigenvalues for a fixed number of prolate conditions}\label{sectevoleigen}

\begin{figure}[H]	\begin{center}
\includegraphics[scale=0.45]{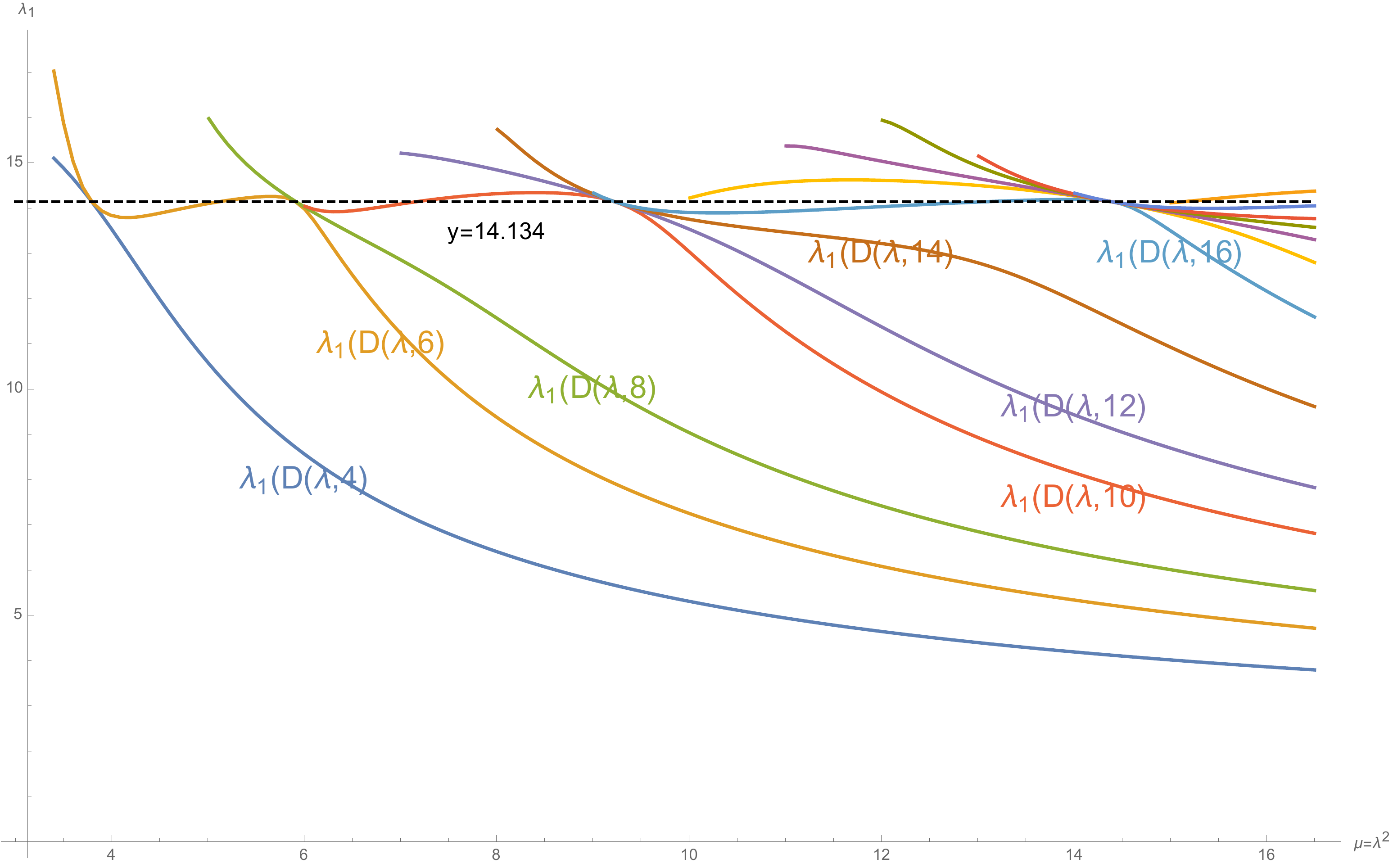}
\end{center}
\caption{Evolution of the first non-zero eigenvalue of $D(\lambda,2\ell)$. The dashed horizontal line is the value of the imaginary part $\zeta_1$ of the first zero of zeta. \label{allconditions}}
\end{figure}
When dealing with the operators $D(\lambda,k)$, with $k$ close to the largest allowed value $\nu(\lambda^2)\sim 2 \lambda^2$ one introduces necessarily a discontinuity due to the discrete nature of the variable $k$. To avoid it one can, for fixed $k$, consider the dependence of the eigenvalues  $\lambda_n(D(\lambda,k))$ as long as $\lambda$ is sufficiently large so that $k<\nu(\lambda^2)$. One finds that 
 for the values $\ell=2,3$, the $\lambda_n(D(\lambda,2\ell))$ agree around $\mu\sim 3.8$ and that their common value is close to $\zeta_1$. This fact is all the more remarkable that when $\mu<4$ \ie $\lambda<2$ there is no summation involved in the \eqref{ephirough}. For  $\ell=3,4,5$, the $\lambda_n(D(\lambda,2\ell))$ agree around $\mu\sim 5.95$ and again we find that their common value is close to $\zeta_1$. For  $\ell=5,6,7,8$, the $\lambda_n(D(\lambda,2\ell))$ agree around $\mu\sim 9.2$ and again their value is close to $\zeta_1$. For  $\ell=8,9,10,11,12,13$ the $\lambda_n(D(\lambda,2\ell))$ agree around $\mu\sim 14.4$ and their value is close to $\zeta_1$. The special values of $\mu$ at which the graphs meet appear to form  a geometric progression. One finds that the ratio of consecutive terms is $\sim \exp(2\pi /\zeta_1)$ and, more generally that for the $n$-th eigenvalue the special values of $\mu$ form a geometric progression with scale ratio $\sim \exp(2\pi /\zeta_n)$ where $\zeta_n$ is the imaginary part of the $n$-th zero of zeta. These ``experimental" facts will be theoretically explained by Theorem \ref{spectralreal}.
\subsection{Quantization of length $\log \mu$}\label{sectquantiz}
 The fact that many graphs  of the eigenvalues  $\lambda_n(D(\lambda,k))$ meet at some specific points of the plane suggests that one could  push the comparison even further and compare these points with the spectrum of the unperturbed operator $D_0(\lambda)$. In terms of the coordinates $(x,y)$ where $x=\mu=\lambda^2$ and $y=\lambda_n(D(\lambda,k))$, the spectrum of $D_0(\lambda)$ is characterized by the quantization condition $x^{iy}=1$. The subset of the plane defined by this condition is the union of the graphs of the functions $2 \pi n/\log x$.

 \begin{figure}[H]	\begin{center}
\includegraphics[scale=0.45]{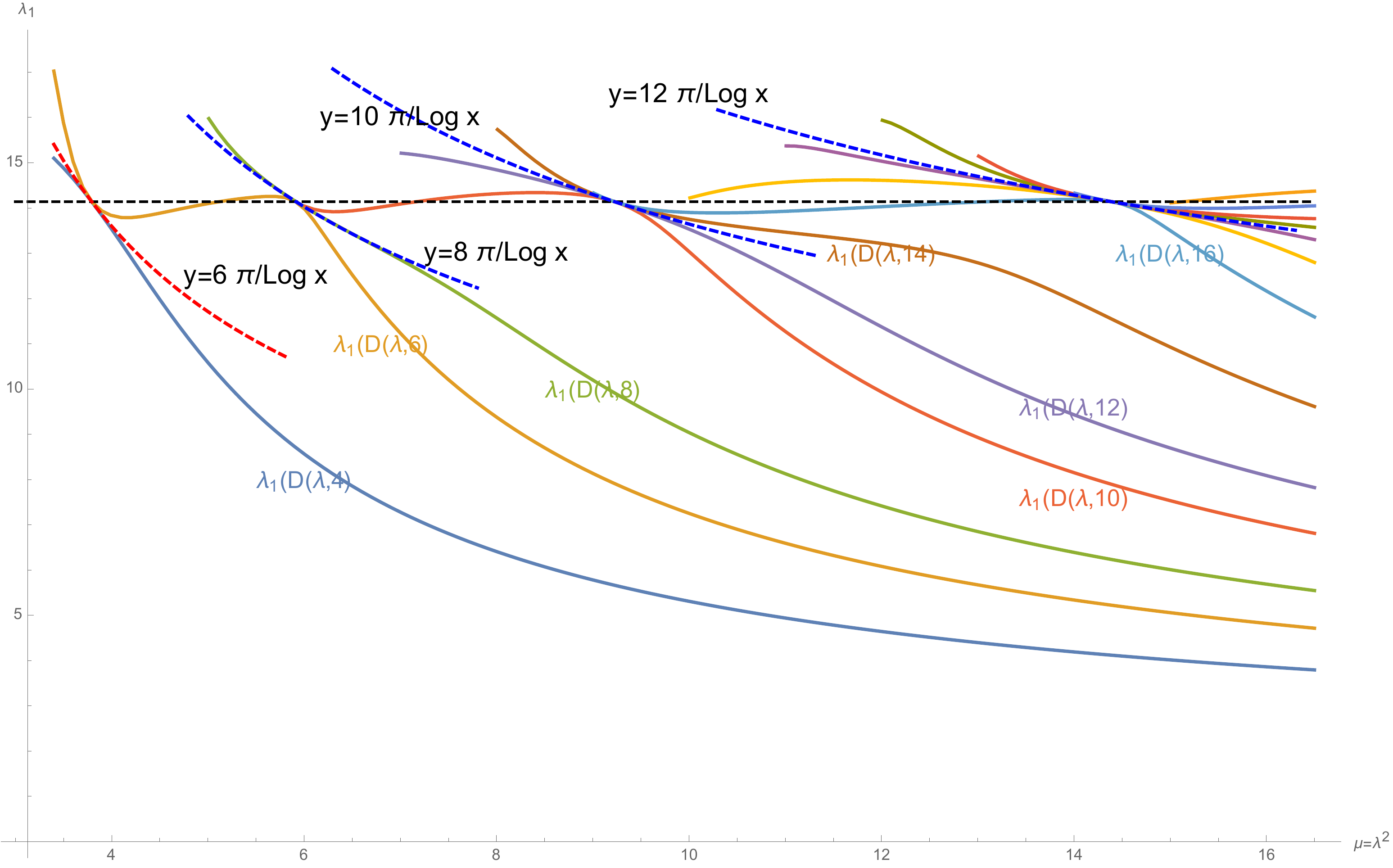}
\end{center}
\caption{Coincidence with solutions of $x^{iy}=1$\label{banderilla}}
\end{figure}
Figure \ref{banderilla} shows a perfect agreement between these graphs and the meeting points of the eigenvalue graphs. Independently of this result, one can measure how far the point $(\mu,\lambda_n(D(\lambda,k))$ is from fulfilling the quantization condition by writing it in the form 
$$
\mu^{i\lambda_n(D(\lambda,k))}=1\iff \vert \mu^{i\lambda_n(D(\lambda,k))}-1\vert =0
$$
and by plotting the graphs of these functions for each integer $n$. They are shown  in Figure \ref{quant1} for $n=1$ and in Figure \ref{quant2} for $n=2$. The key fact here is that the values of $\mu$ at which these functions vanish coincide with the previously determined values where  $\lambda_1(D(\lambda,k+1))\sim \lambda_1(D(\lambda,k))$ of Figures \ref{firsteigen} and \ref{secondeigen}.
 \begin{figure}[H]	\begin{center}
\includegraphics[scale=0.5]{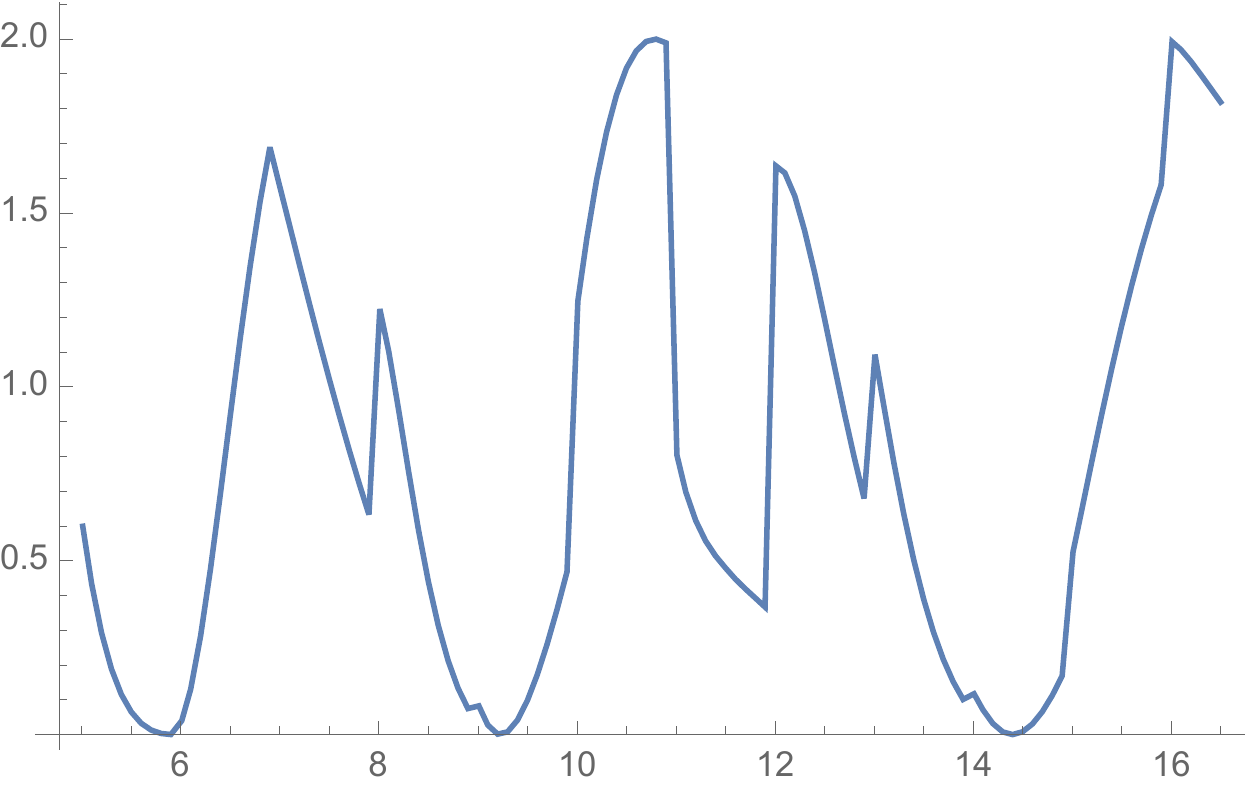}
\end{center}
\caption{Graph of $\vert \mu^{i\lambda_1(\mu)}-1\vert$\label{quant1}}
\end{figure}
\begin{figure}[H]	\begin{center}
\includegraphics[scale=0.5]{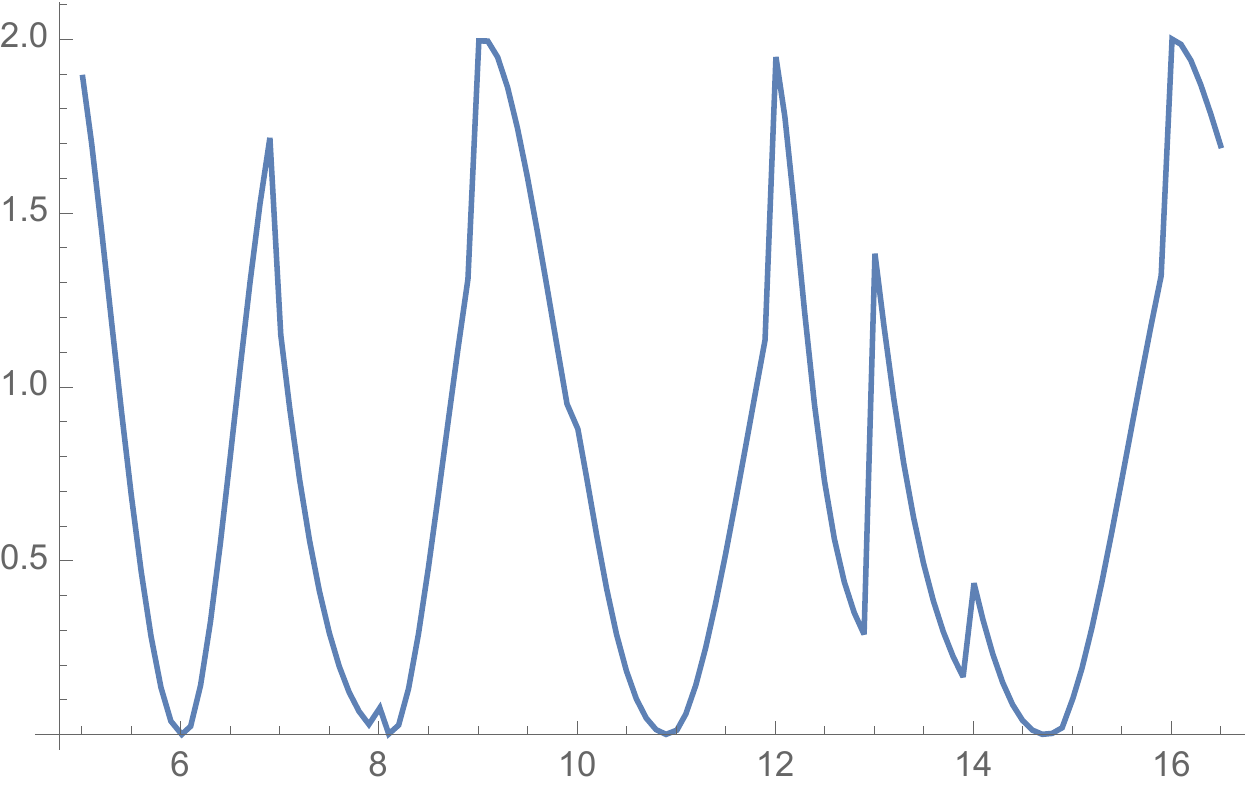}
\end{center}
\caption{Graph of $\vert \mu^{i\lambda_2(\mu)}-1\vert$\label{quant2}}
\end{figure}

\subsection{The criterion of common eigenvector for $D(\lambda,k)$ and $D_0(\lambda)$}\label{sectcommon}
The agreement of the quantization with the meeting points of the graphs of the eigenvalues suggests that all the eigenvectors of the $D(\lambda,k)$ involved agree with each other and are in fact eigenvectors of the unperturbed operator $D_0(\lambda)$. This gives a very strong criterion obtained by measuring the Hilbert space distance of an eigenvector $\xi_n(D(\lambda,k))$ for $D(\lambda,k)$ with the  eigenvector of $D_0(\lambda)$  which has the same rotation number. In Figures \ref{oscill1} and \ref{oscill2} the norm of the difference is plotted and one gets the agreement of the zeros with  the  values determined by the three previous criteria.  Finally Figure \ref{usedeigenvect} compares the first $31$ eigenvalues selected using the last criterion with the imaginary parts of the first $31$ zeros of the Riemann zeta function.
 \begin{figure}[H]	\begin{center}
\includegraphics[scale=0.4]{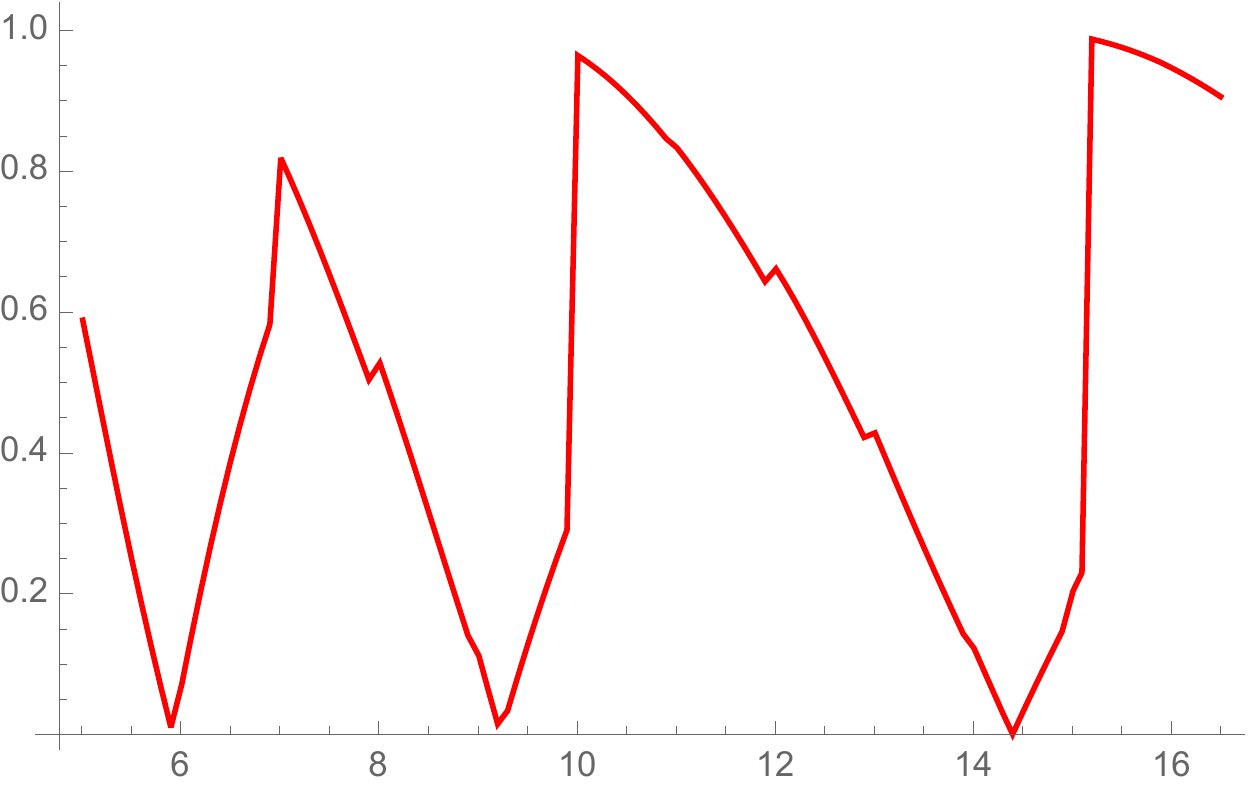}
\end{center}
\caption{Distance of eigenvector of $D(\lambda,k)$ for $\lambda_1$ to eigenvectors of $D_0(\lambda)$\label{oscill1}}
\end{figure}
\begin{figure}[H]	\begin{center}
\includegraphics[scale=0.4]{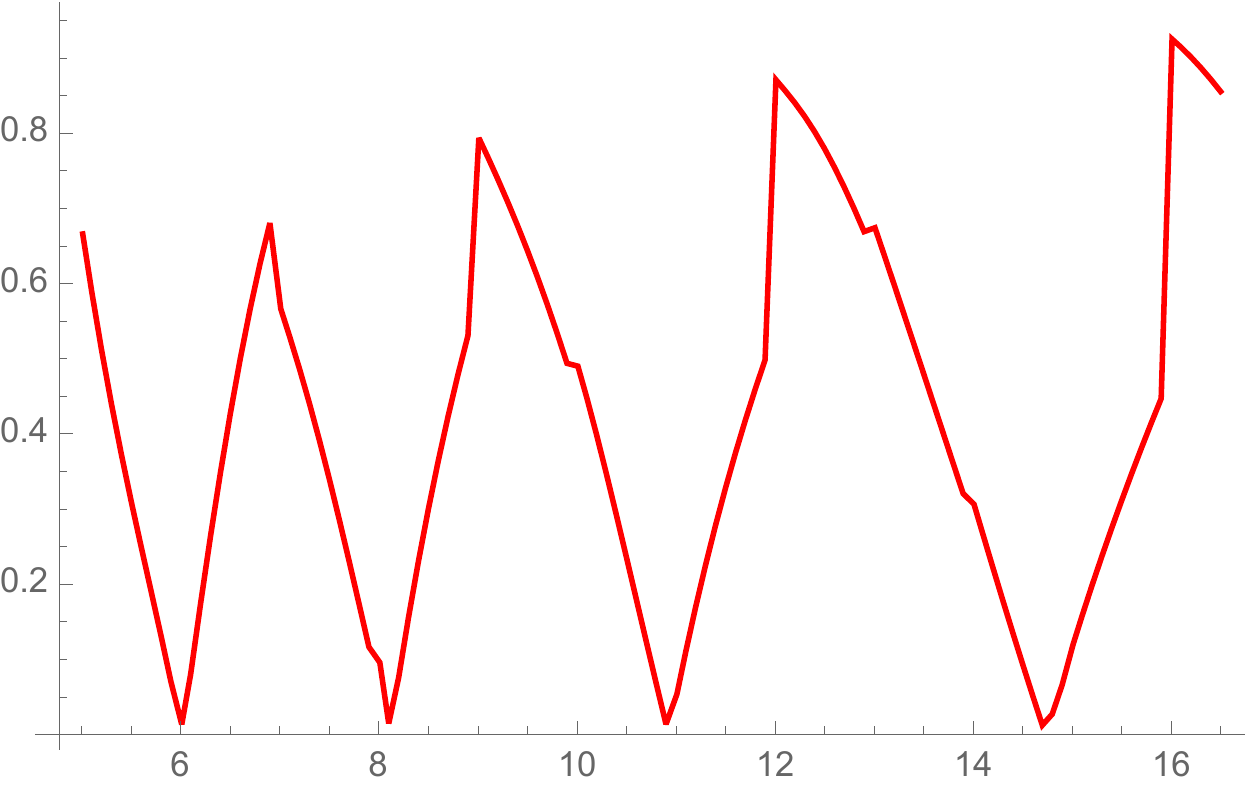}
\end{center}
\caption{Distance of eigenvector $\xi_2(D(\lambda,k))$ to eigenvectors of $D_0(\lambda)$\label{oscill2}}
\end{figure}

\begin{figure}[H]	\begin{center}
\includegraphics[scale=0.5]{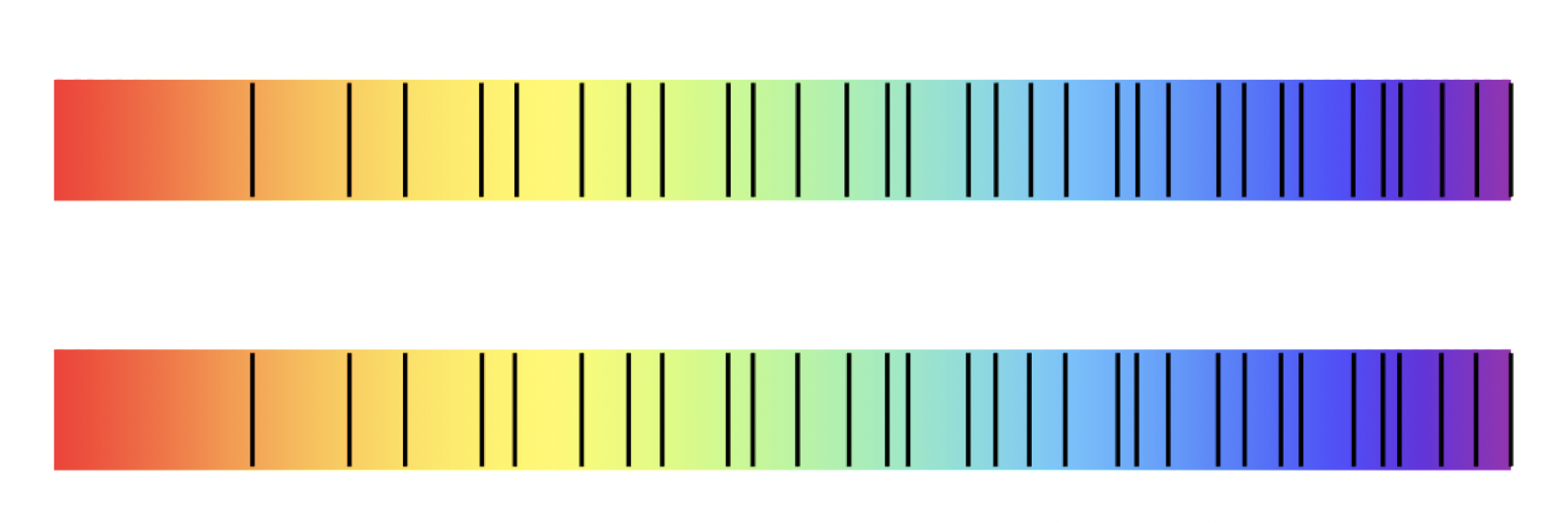}
\end{center}
\caption{Using the criterion $\xi_n(D(\lambda,k))$ eigenvector of $D_0(\lambda)$, one obtains the $31$ eigenvalues compared above with the imaginary parts of the first $31$ zeros of the Riemann zeta function. \label{usedeigenvect}}
\end{figure}

\newpage

\section{$\zeta$-cycles}\label{sectzetacycles}

 The aim of this section is  to provide a theoretical explanation for the numerical computations reported in the previous part of this paper, and in particular to give a theoretical justification for the close similarity of the spectrum of the operator $D(\lambda,k)$ in the spectral triple $\theta(\lambda,k)$ (see Section \ref{spectrip}) and the low lying zeros of the Riemann zeta function. The goal we shall pursue here is to relate these intriguing numerical results with the spectral realization   of the zeros of the Riemann zeta function, as developed in \cite{Co-zeta}. The new theoretical concept emerging is that of a $\zeta$-cycle $C$. In the following part we first explain how to define scale invariant Riemann sums for functions defined on $[0,\infty)$ with vanishing integral. This technique is then implemented in the definition of a  linear map $\Sigma_\mu \cE: \sr0\to L^2(C)
$  which plays a central role in this development and enters in the definition of the $\zeta$-cycle  (Definition \ref{zc}). In \S \ref{sec6.2} we prove that $\zeta$-cycles are stable under finite covers, and finally we state and prove the main result of this paper, namely Theorem \ref{spectralreal}. This result naturally selects a  family of Hilbert spaces $\cH(L):=\Sigma_\mu \cE(\sr0)^\perp$ naturally associated to the critical zeros  of the Riemann zeta function.

\subsection{Scale invariant Riemann sums and the map $\Sigma_\mu \cE$ }\label{sectriemannsums}

Let $\mu>1$ and $\Sigma_\mu$ be the linear map defined on functions $g:\R_+^*\to \C$ by the following formula
\begin{equation}\label{sigmap}
	(\Sigma_\mu g)(u):=\sum_{k\in\Z} g(\mu^ku).
\end{equation}
This definition makes sense pointwise provided $g$ decays fast enough at $0$ and $\infty$ in $\R_+^*$. The map $\cE$ is defined as follows 
\begin{equation}\label{mapE}
	(\cE f)(u):=u^{1/2}\sum_{n>0} f(nu).
\end{equation}
It is, by construction, proportional to a Riemann sum for the integral of $f$. \newline
 We let $\sr0$ be the linear space of real valued even Schwartz functions $f\in \cS(\R)$  such that $f(0)=0=\int f(x)dx$.
The following lemma describes the ``well-behavior'' of the map $\cE$.
 
\begin{lemma}\label{fouriertruncated1} Let $f$ be a function of bounded variation on $(0,\infty)$, of rapid decay for $u\to \infty$,  $O(u^2)$ when $ u\to 0$, and such that $\int_0^{\infty}f(t)dt=0$. Then the following properties hold \newline
$(i)$~$\cE(f)(u)$ is well-defined pointwise, is 
$O(u^{1/2})$ when $u\to 0$ and of rapid decay for $u\to\infty$.\newline
$(ii)$~The series \eqref{sigmap} defining $\Sigma_\mu\cE(f)$ is geometrically convergent, and defines a bounded measurable  function on  $\R_+^*/\mu^{\Z}$.
\end{lemma}
\begin{proof} $(i)$~The sum $S(u):=u\,\displaystyle{\sum_{n=0}^\infty} f(nu)$ is a Riemann sum for the integral $\int_0^{\infty}f(x)dx=0$. One has $f(0)=0$, and the following equality holds
$$
u\sum_{n=0}^{\infty} f(nu)=- \sum_{n=0}^{\infty}\int_{n u}^{(n+1)u}((n+1)u-t)df(t)
$$
since integration by parts in the Stieltjes integral shows that 
$$
\int_{n u}^{(n+1)u}((n+1)u-t)df(t)=\int_{n u}^{(n+1)u}f(t)dt-uf(nu)
$$
while $\int_0^{\infty}f(t)dt=0$ by hypothesis.
Since $\vert ((n+1)u-t)\vert \leq u$ for $t\in [nu, (n+1)u]$, one obtains the upper-bound
$$
{\bigg|} \sum_0^{\infty} f(nu){\bigg|} \leq {\sum_{n=0}^{\infty}}\int_{n u}^{(n+1)u}\vert df(t)\vert.
$$
The integral of the measure $\vert df(t)\vert$ is finite since $f$ is of bounded variation. We thus derive 
$$
{\bigg|} {\sum_{n=0}^{\infty}} f(nu){\bigg|} \leq \int_0^\infty \vert df(t)\vert 
$$
and {from this} it follows that $\vert\cE(f)(u)\vert =O(u^{1/2})$ for $u\to 0$.\newline
$(ii)$~Since $f(u)$ is of rapid decay for $u\to \infty$, one has $\vert f(u)\vert\leq Cu^{-N}$, $N>1$ and {this implies}
$$
\sum_{n\geq 1} \vert f(nu)\vert \leq Cu^{-N}\sum_{n\geq 1}n^{-N}=C'u^{-N}{.}
$$ 
Thus $\cE(f)(u)$ is of rapid decay for $u\to \infty$. Let $u\in [\lambda^{-1},\lambda]$. 
The terms of the  series ${\displaystyle{\sum_\Z \cE(f)}}(\mu^ku)$ converge geometrically for $k>0$;  for $k\leq 0$  $(i)$ gives  $\vert \cE(f)(\mu^ku)\vert\leq C \mu^{k/2} $ and hence  the required uniform geometric convergence follows. \end{proof}

The scaling action of $\R^*_+$ on functions  is  defined by $(\rep(\lambda)f)(x):=f(\lambda^{-1}x)$. 
Next lemma describes the behavior of the scaling action in relation to the map $\cE$
\begin{lemma}\label{scalingact}$(i)$~The Schwartz space $\sr0$ is globally invariant under the scaling action $\rep$ and with $\mu>1$, the following equalities hold 
\begin{equation}\label{mapEtheta}
	\cE \circ\lambda^{-1/2}\rep(\lambda)=\rep(\lambda)\circ \cE, \ \ \rep(\lambda)\Sigma_\mu=\Sigma_\mu \rep(\lambda)
\end{equation}	
$(ii)$~The  scaling action $\rep$ induces an action of the multiplicative group $C_\mu= \R_+^*/\mu^{\Z}$ on  $\Sigma_\mu \cE(\sr0)$.\newline
$(iii)$~Let $f$ be  a function as in Lemma \ref{fouriertruncated1} that coincides near zero with a smooth even function, then $\Sigma_\mu\cE(f)$ belongs to the closure of $\Sigma_\mu \cE(\sr0)$ in $L^2(C_\mu)$.
\end{lemma}
\begin{proof} The conditions defining the subspace $\sr0\subset \cS(\R)$ are invariant under the scaling action. One has
$$
\cE(\rep(\lambda)f)(u)=u^{1/2}\sum_{n> 0} f(n\lambda^{-1}u)=\lambda^{1/2}\rep(\lambda)(\cE(f))(u).
$$
Moreover one has: $\rep(\lambda)\Sigma_\mu=\Sigma_\mu \rep(\lambda)$. Thus, since $\sr0$ is invariant under the scaling action, the same invariance holds for its image  $\Sigma_\mu \cE(\sr0)$ on which the scaling action is now periodic of period $\mu$. From this fact one derives an  induced  action of the multiplicative group $C_\mu=\R_+^*/\mu^{\Z}$. Let $f$ be as in Lemma \ref{fouriertruncated1}. Let $\epsilon >0$ and $\rho\in C_c^{\infty}(\R_+^*)$ have support in a small neighbourhood of $1$ and be such that, for the norm in $L^2(C_\mu)$, $$\Vert\rep(\rho)\Sigma_\mu \cE(f)-\Sigma_\mu \cE(f)\Vert<\epsilon.$$
By applying \eqref{mapEtheta} one has, for a  $\tilde\rho\in C_c^{\infty}(\R_+^*)$ with the same support as $\rho$: $\rep(\rho)\Sigma_\mu \cE(f)=\Sigma_\mu \cE(\rep(\tilde\rho)(f))$. Finally, the hypothesis on $f$ show that the function $\rep(\tilde\rho)(f)$ belongs to $\sr0$. \end{proof}

\subsection{Zeros of zeta and $\zeta$-cycles}\label{sec6.2}
We identify a circle of length $L=\log \mu>0$ with the quotient space $C_\mu:=\R_+^*/\mu^{\Z}$ viewed as a homogeneous space over the multiplicative group $\R_+^*$. This space is endowed with the measure $d^*u$ associated to the Haar measure of the multiplicative group $\R_+^*$. One thus obtains a canonical bundle of $\R_+^*$-homogeneous spaces over the base $(0,\infty)$. 

We keep the notations introduced in the previous part.

\begin{definition}\label{zc} A {\bf $\zeta$-cycle} is  a circle $C$ of length $L=\log \mu$  such that the subspace $\Sigma_\mu \cE(\sr0)$ is not dense in the Hilbert space $L^2(C)$.
\end{definition}
As for closed geodesics, the $\zeta$-cycles are stable under finite covers.

\begin{proposition}\label{traversals} Let $C$ be a $\zeta$-cycle of length $L=\log \mu$, then for any positive integer $n>0$ the n-fold cover of $C$ is a $\zeta$-cycle.	
\end{proposition}
\begin{proof} Let $\pi:C_n\to C$ be the $n$-fold cover of $C$.   From the adjunction of 
the operation $\pi^*:L^2(C)\to L^2(C_n)$ with the operation of sum on the preimage of a point, it follows that if a vector $\xi\in L^2(C)$ belongs to  the orthogonal to $\Sigma_\mu \cE(\sr0){\subset L^2(C)}$ with $\mu=\exp L$, then $\pi^*\xi$ is orthogonal to $\Sigma_{\mu^n} \cE(\sr0)$.\end{proof}

We are now ready to state and prove our main result. The spectral realization of the zeros of the Riemann zeta function of \cite{Co-zeta} admits the following geometric variant 
\begin{figure}[H]	\begin{center}
\includegraphics[scale=0.5]{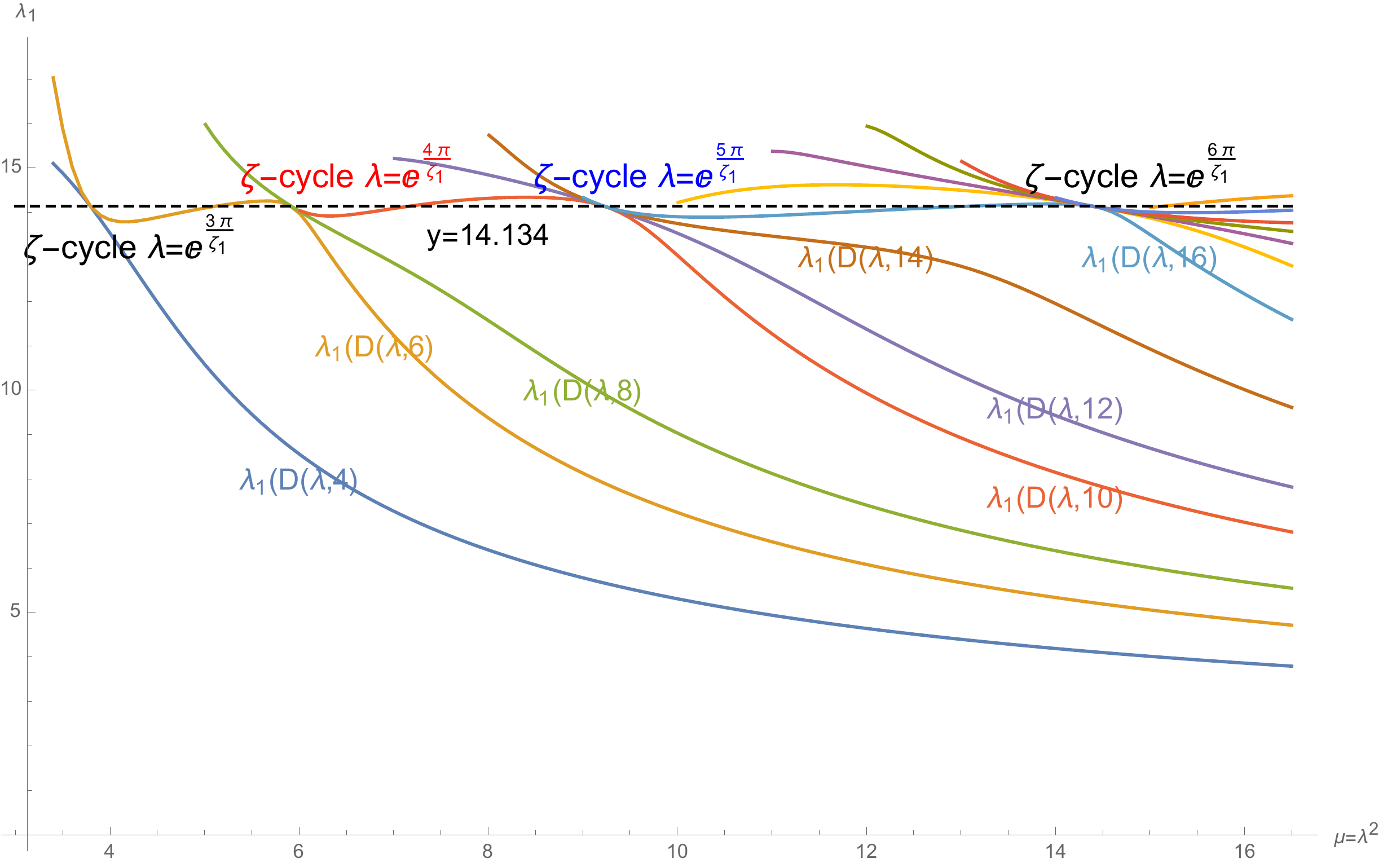}
\end{center}
\caption{Example of $\zeta$-cycles. They are shown here for the first non-zero eigenvalue $\lambda_1(D(\lambda,k))$.  The graphs  touch each other at the points $P(k)=(\exp(\frac{2\pi k}{\zeta_1}),\zeta_1)$. \label{quant2}}
\end{figure}

\begin{theorem}\label{spectralreal} $(i)$~Let $C$ be a $\zeta$-cycle. Then the spectrum of the action of the multiplicative group $\R_+^*$ on the orthogonal complement of $\Sigma_\mu \cE(\sr0)$ in $L^2(C)$ is formed by imaginary parts of zeros of zeta on the critical line.\newline
Conversely:\newline
$(ii)$~Let $s>0$ be such that $\zeta(\frac 12+is)=0$, then any real circle $C$ of length an integral multiple of $2\pi /s$ is a zeta cycle and its spectrum, for the action of $\R^*_+$ on $\Sigma_\mu \cE(\sr0)\subset L^2(C)$, contains $is$.\end{theorem}

\begin{proof} $(i)$~The action of the multiplicative group $\R_+^*$ on the orthogonal of $\Sigma_\mu \cE(\sr0)$ in $L^2(C)$ is periodic and factors through the action of the multiplicative group $G=\R_+^*/\mu^{\Z}$. Since $G$ is a compact abelian group the representation of $G$ is a direct sum of unitary characters. Let $\chi$ be any such unitary character, then there then exists $s\in \R$ with $\mu^{is}=1$,  such that $\chi(u)=u^{is}$ for all $u\in G=\R_+^*/\mu^{\Z}$. The orthogonality property of an eigenvector with eigenvalue $\chi$ with respect to the subspace $\Sigma_\mu \cE(\sr0)\subset L^2(C)$ implies the following vanishing
$$
\int_G \chi(u)\Sigma_\mu \cE(f)(u)d^*u=0\qqq f\in \sr0{.}
$$
In turn, this  implies  the vanishing of the following integral 
$$
\int_{\R_+^*}u^{is} \cE(f)(u)d^*u=0\qqq f\in \sr0.
$$
Let in particular $f(x):=  e^{-\pi  x^2} \pi x^2 \left(-2 \pi  x^2+3\right)$. One easily checks that $\int_0^{\infty}f(x)dx=0$ and that $f\in \sr0$. Furthermore one has 
$$
\int_{\R_+^*}u^{is} f(u)d^*u=\left(\frac 14 +s^2\right) \pi ^{-\frac{1}{4}-\frac{is}{2}} \Gamma \left(\frac{1}{4}+\frac{is }{2}\right)
$$
and, as we shall prove in general in the following part for functions $f\in \sr0$, one also has 
$$
\int_{\R_+^*}u^{is} \cE(f)(u)d^*u=\zeta(\frac 12+is)\int_{\R_+^*}u^{is} f(u)d^*u.
$$
This fact entails that for the specific choice of $f$ made above one obtains the equality
 $$\int_{\R_+^*}u^{is} \cE(f)(u)d^*u=(\frac 14 +s^2)\zeta_\Q\left(\frac 12+is\right)$$ 
where $\zeta_\Q$ denotes the complete zeta function. Thus  one derives that $\frac 12+is$ is a zero of zeta.\newline
$(ii)$~Let $s>0$ be such that $\zeta(\frac 12+is)=0$ and let $L=2\pi n/s$,  with $n>0$ a positive  integer. To show that the circle $C$ of length $L$ is a zeta cycle, we first prove that  
$$
\int_{\R_+^*}u^{is} \cE(f)(u)d^*u=0\qqq f\in \sr0.
$$
Indeed, let $f\in \sr0$, then with $w$ being the unitary identification $w(f)(x)=x^{1/2}f(x)$,  the multiplicative Fourier transform  $\fourier_\mu(w(f))=\psi$:  
$\psi(z)=\int_{\R_+^*}f(u)u^{\frac 12-iz}d^*u$ is holomorphic in the half plane $\Im(z)>-5/2$ since $f(u)=O(u^2)$ for $u\to 0$. For $n>0$, one obtains
$$
\int_{\R_+^*}u^{1/2}f(nu)u^{-iz}d^*u=n^{-1/2+iz}\int_{\R_+^*}v^{1/2}f(v)v^{-iz}d^*v
$$
and for $\Im(z)>1/2$, one derives by applying Fubini theorem
$$
\int_{\R_+^*}{\sum_n} u^{1/2}f(nu)u^{-iz}d^*u= \left({\sum_n} n^{-1/2+iz}\right)\int_{\R_+^*}v^{1/2}f(v)v^{-iz}d^*v
$$
so that for $z\in\C$ with $\Im(z)>1/2$ one obtains 
\begin{equation}\label{mapEzeta}
	\int_{\R_+^*}\cE(f)(u)u^{-iz}d^*u=\zeta(\frac 12-iz)\psi(z).
\end{equation}
To justify the use of Fubini theorem in proving \eqref{mapEzeta},  note that for $N>1$ one derives from Lemma \ref{fouriertruncated1} the following estimate 
$$
\sum_{n\geq 1} \vert f(nu)\vert \leq Cu^{-N}\sum_{n\geq 1}n^{-N}=C'u^{-N}\qqq u>1
$$ 
which shows that the series $\sum_{n\geq 1} \vert f(nu)\vert$ is of rapid decay for $u\to \infty$. For $u\to 0$ 
we use instead the rough estimate, due to the absolute integrability of $f$, of the form 
$$
\sum \vert f(nu)\vert =O(u^{-1}).
$$
This ensures  the validity of Fubini for $\Im(z)>1/2$. Now, we know that $\zeta(\frac 12-iz)$ has a pole at $z=i/2$, but since $\psi(i/2)=0$ this singularity does not affect the above  product $\zeta(\frac 12-iz)\psi(z)$ which is thus holomorphic in the half plane  $\Im(z)>-5/2$. By applying Lemma \ref{fouriertruncated1}, we see that the function $\cE(f)(u)$  is 
$O(u^{1/2})$ when $u\to 0$ and of rapid decay for $u\to\infty$. Thus $\int_{\R_+^*}\cE(f)(u)u^{-iz}d^*u$ is holomorphic in the half-plane $\Im(z)>-1/2$. Therefore one may conclude that \eqref{mapEzeta} holds when $z\in \R$ and,  if $\zeta(\frac 12+is)=0$, one  obtains 
$$
\int_{\R_+^*}\cE(f)(u)u^{is}d^*u=0\qqq f\in \sr0.
$$
At the beginning of this proof one has defined $L=2\pi n/s$: let now $\mu=\exp L$ then one has $\mu^{is}=\exp(2\pi i n)=1$.  In this way the function $\chi(u):=u^{is}$ is well-defined on $C=\R_+^*/\mu^{\Z}$ and the following vanishing holds in $L^2({C})$
$$
\langle \Sigma_\mu\cE(f)\mid \chi\rangle= \int_G \chi(u)\Sigma_\mu \cE(f)(u)d^*u=\int_{\R_+^*}\cE(f)(u)u^{is}d^*u=0\qqq f\in \sr0.
$$
This shows that ${C}$ is a zeta-cycle and that its spectrum contains $is$.\end{proof}

The above development provides us with a family of Hilbert spaces $\cH(L):=\Sigma_\mu \cE(\sr0)^\perp\subset L^2(C)$ and,  for each integer $n>0$, maps $\pi^*_n:\cH(L)\to \cH(nL)$ which lift the action of $\N^\times$ on $(0,\infty)$. Moreover  we also have an action $\rep(\lambda)$ of $\R_+^*$ on $\cH(L)$ and  we have shown that the linear maps $\pi^*_n$ are equivariant. Let $Z$ be the set of imaginary parts of critical zeros of the Riemann zeta function, one  finally deduces the following 
\begin{corollary}\label{repetitions}
\begin{equation}\label{hlnotzero}
	\cH(L)\neq \{0\}\iff \exists s\in Z,\, n\in \Z ~{\text{s.t.}}~ sL=2 \pi n.
\end{equation}	
\end{corollary}
\begin{proof} Assume first that $sL=n$ with $s$ and $n$ positive. Then it follows from  Theorem \ref{spectralreal} $(ii)$, that $\cH(L)\neq \{0\}$, since the circle of length $L$ is a zeta-cycle. Conversely, if $\cH(L)\neq \{0\}$ then the circle $C$ of length $L$ is a zeta-cycle and there exists by Theorem \ref{spectralreal} $(i)$, a positive $s\in Z$ and a non zero vector  $\xi \in\cH(L)$ such that $\rep(\lambda)(\xi)=\lambda^{is}\xi$ for all $\lambda\in\R_+^*$. Since the action of $\R_+^*$ on $L^2(C)$ is periodic of period $\mu=\exp(L)$ we have $\mu^{is}=1$ and this entails $sL\in 2\pi \Z$. \end{proof}

\section{Outlook}

In this paper we have unveiled a new compelling  relation between noncommutative geometry and 
the Riemann zeta function using the concept of spectral triple. The previous relations are
 \begin{itemize}
\item The BC system is a system of quantum statistical mechanics with spontaneous symmetry breaking which admits the Riemann zeta function as its partition function.
\item The  adele class space of $\Q$ is a noncommutative space, dual to the BC-system and directly related to the zeros of the $L$-functions with Grossencharacter \cite{Co-zeta}. 
\item The quantized calculus is a key ingredient of the semi-local trace formula and it provides  a source of positivity for the Weil quadratic form \cite{weilpos}.	
\end{itemize}
It turns out that the adele class space of $\Q$ in its topos theoretic incarnation as the Scaling Site  (the topos $\scal2=[0,\infty) \rtimes \nt$) is the natural parameter space for the circles of length $L$ which play a critical role in the present paper. Proposition \ref{traversals} gives the compatibility of $\zeta$-cycles with the action of $\nt$ by multiplication on the parameter $L$. The action of $\nt$ coming from coverings turns $L^2(C)$ into a sheaf over the Scaling Site $\scal2$. The family of subspaces $\Sigma_\mu \cE(\cS_0)\subset L^2(C)$ generate a subsheaf of  modules over the sheaf of smooth functions and one is then entitled to  consider the cohomology of the  quotient sheaf over $\scal2$.  Endowed with the $\R_+^*$-equivariance this cohomology provides the spectral realization of the critical zeros of zeta, taking care, in particular, of eventual  multiplicity. We shall discuss this fact in details in a forthcoming paper which, in particular, gives an application of the algebraic geometry over $\mathbb S$ developed in \cite{mothers}

Finally, the stability of $\zeta$-cycles under coverings is  reminiscent of the behavior of closed geodesics in a Riemannian manifold, suggesting to look for a mysterious ``cusp" whose closed geodesics would correspond to  $\zeta$-cycles.

\begin{Backmatter}

%\paragraph{Acknowledgments}We are grateful for 

\end{Backmatter}

\end{document}